\documentclass[a4paper,10pt]{article}

\usepackage[english]{babel}
\usepackage[utf8]{inputenc} 
\usepackage[T1]{fontenc}
\usepackage{lmodern}
\usepackage{amsmath,amssymb,amsthm,esint}
\usepackage[a4paper]{geometry}
\usepackage[cmyk]{xcolor}
\usepackage[unicode=true,pdfusetitle,bookmarks=true,bookmarksnumbered=false,bookmarksopen=false,
breaklinks=false,pdfborder={0 0 0},pdfborderstyle={},backref=false,colorlinks=true,pdfauthor={Anne-Laure Dalibard, Julien Guillod, Antoine Leblond}]{hyperref}
\usepackage[capitalize]{cleveref}
\usepackage{enumitem}

\newtheorem{theorem}{Theorem}[section]
\newtheorem{lemma}[theorem]{Lemma}
\newtheorem{proposition}[theorem]{Proposition}

\newtheorem{definition}[theorem]{Definition}

\newtheorem{remark}[theorem]{Remark}
\newtheorem{corollary}[theorem]{Corollary}

\DeclareMathOperator{\supp}{supp}

\numberwithin{equation}{section}

\newcommand{\bzero}{\mathbf{0}}

\newcommand{\BL}{\mathrm{BL}}

\newcommand{\mcC}{C}
\newcommand{\rd}{\mathrm{d}}
\newcommand{\ud}{\,\mathrm{d}}

\renewcommand{\div}{\mathrm{div}}
\newcommand{\be}{\boldsymbol{e}}

\newcommand{\Ep}{\mathrm{E}}
\newcommand{\eps}{\varepsilon}

\renewcommand{\hat}{\widehat}

\newcommand{\Id}{\mathrm{Id}}
\newcommand{\bk}{{\boldsymbol{k}}}
\newcommand{\Lip}{\mathrm{Lip}}

\newcommand{\bn}{{\boldsymbol{n}}}

\newcommand{\p}{\partial}

\newcommand{\para}{\parallel}

\newcommand{\N}{\mathbf{N}}
\newcommand{\Z}{\mathbf{Z}}
\newcommand{\R}{\mathbf{R}}
\newcommand{\T}{\mathbf{T}}

\newcommand{\bu}{\boldsymbol{u}}

\newcommand{\rw}{\mathrm{w}}

\newcommand{\bx}{\boldsymbol{x}}

\newcommand{\ie}{\emph{i.e.\,}}

\newcommand{\Om}{\Omega}
\newcommand{\thbl}{\theta^\BL}
\newcommand{\psbl}{\psi^\BL}
\newcommand{\thint}{\theta^{\mathrm{int}}}
\newcommand{\psiint}{\psi^{\mathrm{int}}}
\newcommand{\rem}{\mathrm{rem}}
\newcommand{\mtop}{\mathrm{top}}
\newcommand{\mbot}{\mathrm{bot}}
\newcommand{\Thtop}{\Theta_\mtop}
\newcommand{\Thbot}{\Theta_\mbot}
\newcommand{\Pstop}{\Psi_\mtop}
\newcommand{\Psbot}{\Psi_\mbot}
\newcommand{\threm}{\theta^{\rem}}
\newcommand{\psrem}{\psi^{\rem}}
\newcommand{\bX}{\mathbf{X}}
\newcommand{\thapp}{\theta^\mathrm{app}}

\newcommand{\ba}{\begin{aligned}}
\newcommand{\ea}{\end{aligned}}

\newcommand{\beq}{\begin{equation}}
\newcommand{\eeq}{\end{equation}}

\title{Long-time behavior of the Stokes-transport system\\ in a channel}

\author{Anne-Laure Dalibard\textsuperscript{1,2} \quad Julien Guillod\textsuperscript{1,2} \quad Antoine Leblond\textsuperscript{1}}
\footnotetext[1]{Sorbonne Université, CNRS, Université Paris Cité, Inria, Laboratoire Jacques-Louis Lions (LJLL), 75005 Paris, France}
\footnotetext[2]{École Normale Supérieure, Université PSL, Département de Mathématiques et applications, 75005, Paris, France}

\date{July 2025}

\begin{document}
	
	\maketitle
	
	\begin{abstract}
		We consider here a two-dimensional incompressible fluid in a periodic channel, whose density is advected by pure transport, and whose velocity is given by the Stokes equation with gravity source term.
		Dirichlet boundary conditions are taken for the velocity field on the bottom and top of the channel, and periodic conditions in the horizontal variable.
		We prove that the affine stratified density profile is stable under small perturbations in Sobolev spaces and prove convergence of the density to another limiting stratified density profile for large time with an explicit algebraic decay rate.
		Moreover, we are able to precisely identify the limiting profile as the decreasing vertical rearrangement of the initial density.
		Finally, we show that boundary layers are formed for large times in the vicinity of the upper and lower boundaries. These boundary layers, which had not been identified in previous works, are given by a self-similar Ansatz and driven by a linear mechanism. This allows us 
		to precisely characterize the long-time behavior beyond the constant limiting profile and reach more optimal decay rates.
	\end{abstract}
	
	\textbf{MSC classes}\quad{}35B35, 35B40, 35M13, 35D35, 35Q49, 76D07, 76D10
	
	\newpage
	
	\tableofcontents

	\newpage

	\newpage
	\section{Introduction}
	
	\begin{subequations}\label{eq:ST}
		The Stokes-transport system
		\begin{equation}
			\left\{
			\ba
			\p_t\rho + \bu\cdot\nabla\rho &= 0 \\
			-\Delta \bu + \nabla p &= -\rho\be_z \\
			\div\,\bu &= 0 \\
			\rho|_{t=0} &= \rho_0
			\ea
			\right.
		\end{equation}
		models the evolution of an incompressible inhomogeneous fluid with density $\rho,$ velocity and pressure fields $(\bu,p)$. 
		For physical reasons and without loss of generality, we assume that the initial density $\rho_0$ is non-negative.
		This equation will be studied in a two-dimensional periodic strip, namely $\Omega=\T\times(0,1)$ with variables $(x,z)\in\Omega$ and with Dirichlet boundary condition of the velocity field:
		\begin{equation}
			\bu = \bzero \quad \text{on} \quad \partial\Omega.
		\end{equation}
	\end{subequations}
	
	It consists of a coupling of the transport equation for the density of the fluid, with a velocity field satisfying for all times the Stokes equation with gravity forcing $-\rho\be_z$ where $\be_z$ is the unitary vertical vector.
	This equation has been studied in particular in \cite{hofer2018sedimentation} and \cite{mecherbet2020sedimentation} showing that \eqref{eq:ST} is a model obtained as the homogenization limit of inertialess particles in a fluid satisfying Stokes equation. A more recent paper \cite{grayer} shows that this system is obtained as a formal limit where the Prandtl number is infinite.
	In this paper, the domain is chosen as $\Omega=\T\times(0,1)$, which describes a physically meaningful situation including Dirichlet boundary conditions.

	\paragraph{Well-posedness.}
	The well-posedness of this system has been shown in \cite{antontsev2000free} for piecewise constant initial data in bounded domains of $\R^n$ 
 and in  \cite{leblond} for arbitrary $L^\infty$ data in bounded domains of $\R^2$ and $\R^3$ or in the infinite strip $\R\times(0,1)$,
 the well-posedness in $\Omega=\mathbf{T}\times(0,1)$ being a direct consequence.

	Well-posedness in Sobolev spaces is required for our results. Since this result does not seem to appear in the literature, we provide a concise proof of the global well-posedness of this problem in \cref{app:wp} for the sake of completeness. More precisely, for any $\rho_0\in H^m$ with $m\geq3$, there exists a unique strong solution $(\rho,\bu)$ of \eqref{eq:ST} with $\rho\in\mcC(\R_+;H^m(\Omega))$	and $\bu\in\mcC(\R_+;H^{m+2}(\Omega))$.
	Well-posedness in other domains and spaces has also been proven, see for example the recent results \cite{mecherbetsueur,inversi}.
	
	\paragraph{Steady states.}
	Before going further let us observe that the stationary states, \ie states such that $\p_t\rho = 0$, of this system are precisely the \emph{stratified} density profiles, which means in this paper density profiles depending only on the vertical variable $z$. Indeed, for such a map $\rho_s = \rho_s(z)$,
	\begin{equation*}
		\left(\rho,\bu,p\right) = \left(\rho_s,\bzero,-\int^z\rho_s(z')\ud z'\right).
	\end{equation*}
	is a solution of \eqref{eq:ST}. To show the converse, let us introduce the potential energy associated to a density profile $\rho$,
	\begin{equation*}
		\Ep(\rho) := \int_\Omega z\rho\ud x\ud z.
	\end{equation*}
	The energy balance is
	\begin{equation} \label{eq:ep}
		\frac\rd{\rd t}\Ep(\rho)
		= \int_\Omega z\p_t\rho
		=-\int_\Omega z\bu\cdot\nabla\rho
		=\int_\Omega \bu\cdot\be_z\rho
		=-\|\nabla\bu\|_{L^2}^2,
	\end{equation}
	where the divergence-free and the Dirichlet boundary conditions on $\bu$ are used in the integration by parts. The last equality is simply the basic estimate of the Stokes equation.
	The potential energy dissipates exactly through the viscosity effects. From this observation we see that the whole evolution is non-reversible; the fluid only rearranges in states of lower potential energy. Moreover, a stationary state is exactly a state for which $\bu = \bzero$, therefore it means that the density $\rho$ and the pressure $p$ must satisfy
	\begin{equation*}
		\nabla p = -\rho\be_z,
	\end{equation*}
	so that the pressure is independent of the $x$ variable, implying $\rho$ depends only on the $z$ variable.

{The aim of this paper is to study the long-time  behavior of perturbations of stratified initial data in the stable regime, with lighter fluid on top and heavier fluid on the bottom.
We will prove three different results: the first one, \cref{thm:stab}, provides the stability of such stratified profiles, together with some decay estimates.
The second one, \cref{thm:BL-linear}, gives an explicit asymptotic decomposition of solutions of a linear version of \eqref{eq:ST} as $t\to \infty$. 
In particular, we identify boundary layer profiles in the vicinity of the top and bottom boundaries.
Eventually, in \cref{thm:BL-nonlinear}, we go back to the nonlinear system \eqref{eq:ST} and provide a more precise description of the solutions as $t\to \infty$, stepping on the analysis from \cref{thm:BL-linear}. A striking consequence of our results lies in the fact that the boundaries slow down the relaxation towards the asymptotic state. This new observation could probably be adapted to other systems, see \cref{rmk:extension-IPM}.
}
 
	\paragraph{Main stability result.}

	For simplicity and in the rest of this paper, we consider perturbations of the affine profile $\rho_s(z) = 1-z$, although more general profiles such that $\p_z \rho_s<0$  could be considered, see \cref{general-profiles}.

	Our main stability result for perturbations vanishing on the boundary is the following:
	\begin{theorem}
		\label{thm:stab}
		There exists a small universal constant $\eps_0 >0$ such that for any $\rho_0 \in H^6(\Omega)$ satisfying $\|\rho_0-\rho_s\|_{H^6} \leq \eps_0$ and $\rho_0-\rho_s\in H^2_0(\Omega)$, the solution $\rho$ of \eqref{eq:ST} satisfies:
		\begin{align} \label{eq:thm-decay}
			\|\rho-\rho_\infty\|_{L^2(\Omega)} &\lesssim \frac{\eps_0}{1+t},& \|\rho-\rho_\infty\|_{H^4(\Omega)} \lesssim \eps_0,
		\end{align}
		where $\rho_\infty$ is given by the decreasing vertical rearrangement of $\rho_0$:
		\begin{equation*}
			\rho_\infty(z) := \int_0^\infty \boldsymbol{1}_{0 \leq z \leq |\{\rho_0>\lambda\}|}\ud\lambda.
		\end{equation*}
	\end{theorem}
	Note that the condition on $H^2_0(\Omega)$ is equivalent to the following requirements, discussed in the following paragraphs: 
	\begin{equation*}
		\rho_0|_{\partial\Omega} = \rho_s|_{\partial\Omega}, \qquad \p_\bn \rho_0|_{\partial\Omega} = \p_\bn\rho_s|_{\partial\Omega}.
	\end{equation*}
	This theorem will be proven in \cref{sec:stab}, and we provide a scheme of proof at the end of this Section.
	
	\paragraph{Remarks on the main stability result.}

	\begin{itemize}
		\item  Since the set of steady-states is not discrete, it is expected that $\rho_s$ is not asymptotically stable, and that the long-time behavior is given by a slightly modified density profile. In general, this asymptotic profile depends on the entire nonlinear dynamics in a very non-explicit way. However the transport equation is remarkable since it preserves the measure of the level sets. This property combined with the fact that the asymptotic profile is strictly decreasing (as a smooth perturbation of $\rho_s$) allows us to identify the asymptotic profile as the decreasing vertical rearrangement of $\rho_0$, which can be computed directly from $\rho_0$ without dependence on the full non-linear dynamics. See \cref{sub:finalstate} for details.

        \item In fact the identification of the limit as the decreasing vertical rearrangement is quite general and only requires two properties. First the density needs to converge to a stratified (\emph{i.e} independent of $x$) and decreasing limiting profile. Second, the density should satisfy a transport equation with well-defined characteristics. This is in particular the case for the incompressible porous medium equation, see \cref{sub:finalstate} for details.
  
		\item This result proves the stability of the particular state $\rho_s(z)=1-z$. We believe that the result generalizes and our proofs adapt to the case of stratified $\rho_s\in H^6$ satisfying
		\begin{equation*}
			\sup_{(0,1)}\p_z\rho_s < 0.
		\end{equation*}
		This remark is detailed at the end of \cref{sub:linearized}. We note that without the monotony assumption, lighter fluid might be below heavy one, and physical instabilities - similar to the Rayleigh-Bénard or Rayleigh-Taylor instabilities - are expected to develop (see \cite{drazin2004hydrodynamic}). Some weak convergence up to extraction toward a stationary state could be proven, but the limit might be a non-trivial $\omega$-limit set in general. In any case, it is not clear whether convergence to the rearranging steady-state holds.
		
		\item One can of course wonder about the strong regularity requirement in \cref{thm:stab}. 
			It turns out that one can adapt a strategy developed by Kiselev and Yao \cite{kiselevyao} about the instability of the incompressible porous media equation. Indeed, the arguments are  essentially geometric, and the result is the same: there exist smooth perturbations small in $H^{2-}(\Omega)$-norm such that $\limsup_{t\to\infty} \| \rho(t) - \rho_s \|_{H^s(\Omega)} = \infty$ for any $s>1$. Therefore, this shows the existence of a regularity threshold between $H^2(\Omega)$ and $H^6(\Omega)$ between stability and instability.
		The details are provided in the thesis of Antoine Leblond \cite{leblondPhD}.
		
		\item Finally, an interesting question is  the optimality of the $(1+t)^{-1}$ decay in \eqref{eq:thm-decay}. The dynamics of the equation preserve the fact that the perturbation and its normal derivative are vanishing on $\p\Omega$ \ie $\rho - \rho_s\in H^2_0(\Omega)$. For higher normal derivatives this property is not preserved, and this is the main reason why the time-decay is limited.
		This is one of the main motivations to study the formation of boundary layers in this system, together with the possibility to allow non-vanishing perturbations on $\p\Omega$.
    \cref{thm:BL-nonlinear} below indicates that the optimal decay rate under the assumptions of \cref{thm:stab} is very likely $(1+t)^{-9/8}$. We refer to the discussion at the top of page \pageref{optimal-rate} for more details.

	\end{itemize}

	\paragraph{Related results and comparison with the incompressible porous medium equation.}
		In \cite{gancedogranero} the interface problem for \eqref{eq:ST} is considered also in the domain $\Omega=\T\times(0,1)$. The interface problem treats the case where the density is equal to two different constants below and above an interface $\Gamma(t)\subset\Omega$. The question lies in the regularity of the interface, the well-posedness for $L^\infty$ densities being established in \cite{leblond,antontsev2000free}.
	More precisely, the authors prove local well-posedness for the interface in $C^{1,\gamma}$ for $0<\gamma<1$
	as well as the global well-posedness and decay of small perturbation in $H^3(\T)$ of the flat interface with lighter fluid on top.
	The proof is very different from ours as it uses a contour dynamics equation, but the spirit of the stability result is pretty similar.
	
		Let us also compare the results and properties of the Stokes-transport equation and of the incompressible porous medium equation, namely \eqref{eq:ST} where the Stokes equation is replaced by Darcy's law,
	\begin{equation} \label{eq:Darcy}
		\left\{
		\ba
		\p_t\rho+\bu\cdot\nabla\rho &= 0\\
		\bu+\nabla p &= -\rho\be_z \\
		\div\,\bu &= 0\\
		\rho|_{t=0} &= \rho_0.
		\ea
		\right.
	\end{equation}
	This equation has been intensively studied and we only cite comparable results. The question of well-posedness is much more difficult than for the Stokes-transport equation. In particular global well-posedness like \cref{thm:wp} seems to remain an open question. Local in time well-posedness has been shown in \cite{cordoba2007,xue2009,Yu2014,Constantin2015} whereas  ill-posedness through non-uniqueness in some spaces has been shown in \cite{Cordoba2010,Shvydkoy2011,Isett2015}.
	
	Concerning classical global solutions, the only known results have been proven for initial data close enough in Sobolev space to the stratified initial data $\rho_s(z)=1-z$ by \cite{elgindi} in $\mathbf{R}^2$ and $\mathbf{T}^2$ and later generalized in \cite{castro2019global} to the domain $\mathbf{T}\times(0,1)$.
	More precisely, these results prove that the profile $\rho_s(z)=1-z$ is asymptotically stable under small perturbations in $H^m$ for some $m$.
  Let us also mention the recent work \cite{park2025}, which revisits these results, relying mainly on energy estimates.
	
	In $\mathbf{T}\times(0,1)$, the boundary conditions identified and used by Castro, Cordoba and Lear \cite{castro2019global} ensure that the main linearized structure remains stable by differentiation, which makes the analysis of that work similar to the one of the periodic or whole space case. This analysis has then been extended by the same authors to the Boussinesq system with a damping velocity term in \cite{castro2019asymptotic}.
	In particular integrations by parts of high-order derivatives are possible to obtain uniform bounds in Sobolev spaces of high enough regularity. By using similar boundary conditions for Stokes-transport in $\mathbf{T}\times(0,1)$, the results of \cite{castro2019global} could be adapted in a straightforward way. 
	In our situation, \emph{the presence of the Dirichlet boundary condition is the major obstacle}.  In particular,  uniform bounds in high regularity Sobolev spaces are no longer valid, as \cref{thm:BL-nonlinear} below will highlight. This is due to the presence of boundary layers, as explained above. More details are provided below in the scheme of the proof.
	
	Eventually, let us mention that  the existence of the limiting profile was obtained in \cite{elgindi,castro2019global} through a fixed point argument. One contribution of the present paper is to precisely identify the long-time asymptotic profile as the decreasing vertical rearrangement of $\rho_0$. As explained previously, our method to identify the limiting profile is robust and in particular also applies to show that the long-time asymptotic profile for the incompressible porous media equation is also given by the decreasing vertical rearrangement.

	\paragraph{Linear asymptotic expansion for non-vanishing perturbation on $\p\Omega$.}
	\cref{thm:stab} is only valid under the assumption that the perturbation and its normal derivative are vanishing on $\p\Omega$ \ie when $\rho_0 - \rho_s\in H^2_0(\Omega)$. If the perturbation does not vanish on the boundary, this question is non trivial even for the linearized equations around $\rho_s=1-z$: denoting by $\theta$ the perturbed density, we consider
	\begin{equation}\label{eq:ST-linear}
		\left\{
		\ba
		\p_t\theta - \bu\cdot\be_z &= 0 \\
		-\Delta \bu + \nabla p &= -\theta\be_z \\
		\div\,\bu &= 0 \\
		\theta|_{t=0} &= \theta_0.
		\ea
		\right.
	\end{equation}
It can be easily checked that \cref{thm:stab} is also valid for \eqref{eq:ST-linear}. In other words, if $\theta_0\in H^6\cap H^2_0$, then $\|\theta(t)\|_{L^2}\lesssim (1+t)^{-1}$.
	Note that there is no smallness assumption in this case because the system is linear.
If $\theta_0$ or $ \p_\bn \theta_0$ do not vanish on the boundary, however,
		it turns out that $\theta$ vanishes as $t\to \infty$ but with a much slower rate.
	This is due to the formation of boundary layers of typical size $t^{-1/4}$ as $t\to \infty$, in the vicinity of $z=0$ and $z=1$.
	More precisely, we will prove the following result in \cref{sec:linear-BL}:
	
	\begin{theorem} \label{thm:BL-linear}
		Let $\theta_0\in H^s(\Om) $ for some $s$ sufficiently large. 
		Then the solution of \eqref{eq:ST-linear} satisfies:
		\begin{align*}
			\theta = \bar{\theta}_0 + \theta^\BL + O(t^{-1}) \quad \text{in} \quad L^2(\Omega) \quad \text{as} \quad t\to\infty,
		\end{align*}
		where $\bar{\theta}_0(z) = \frac{1}{2\pi}\int_0^{2\pi} \theta_0(x,z) \ud x$ is the horizontal average of the initial data and $\theta^\BL$ is the boundary layer part whose leading terms are
		\begin{equation}\label{BL-Ansatz-linear}
			\theta^\BL = \Theta^0_\mtop(x,t^{1/4}(1-z)) + \Theta^0_\mbot(x,t^{1/4}z) + \text{l.o.t}
		\end{equation}
		with $\Theta^0_\mtop$ (resp. $\Theta^1_\mbot$) decaying exponentially as $Z_\mtop=t^{1/4}(1-z)\to\infty$ (resp. as $Z_\mbot=t^{1/4}z\to\infty$). Furthermore, for $t\geq 1$,
	 \[
	 C^{-1} \|\theta_0\vert_{\p \Om}\|_{L^2} t^{-1/8} + O(t^{-3/8})\leq	\|\theta^\BL\|_{L^2(\Omega)} \leq C \|\theta_0\vert_{\p \Om}\|_{L^2} t^{-1/8} + O(t^{-3/8}).
		\]
   
	\end{theorem}
	
    \paragraph{Remarks on the linear asymptotic expansion}
    \begin{itemize}
	\item In fact, the definition of $\thbl$ is more involved, and is given as a sum in powers of $t^{-1/4}$ of different boundary layer profiles. For instance, in the vicinity of $z=0$ and for $t>1$,
	\[
	\thbl= \sum_{j=0}^4 t^{-j/4} \Theta^j_\mbot (x, t^{1/4} z).
	\]
	Furthermore, the construction can be iterated. Up to a stronger regularity requirement on the initial data, we could probably push the expansion of $\thbl $ further and prove that $\theta=\thbl + O(t^{-k})$ for $k$ arbitrarily large. In this case, the definition of $\thbl$ has to be modified in order to include profiles up to $j=4k$. We shall give more details on this matter in \cref{rmk:higher-order-expansion}.
	
	\item Note that the scaling of boundary layers is consistent with the estimates of \cref{thm:stab}: heuristically,
	one power of $t^{1/4}$ is lost with each differentiation (with respect to $z$.)
	
    \end{itemize}

	\paragraph{Non-linear asymptotic expansion.}
	Let us now go back to the nonlinear problem in the case where $\rho_0 - \rho_s\in H^2_0(\Omega)$.
	In this case, $\rho(t)-\rho_s$ and $\p_\bn (\rho(t)-\rho_s)$ vanish on the boundary for all $t\geq 0$ (see \cref{lem:traces}).
	As a consequence, the advection term is negligible in the vicinity of the boundary, and we expect the dynamics to be driven by a linear mechanism in this zone at main order.
Building on the analysis of \cref{thm:BL-linear}, we then
	derive uniform bounds in $H^8(\Omega)$, modulo some boundary layer terms:
	\begin{theorem} \label{thm:BL-nonlinear}
		There exists $\eps_0 >0$ small such that for any $\rho_0 \in H^{14}(\Omega)$ satisfying $\|\rho_0-\rho_s\|_{H^{14}} \leq \eps_0$ and $\rho_0 - \rho_s\in H^3_0(\Omega)$, the solution $\rho$ of \eqref{eq:ST} satisfies:
		\begin{align*}
			\rho = \rho_\infty + \thbl + O(t^{-2}) \quad \text{in} \quad L^2(\Omega) \quad \text{as} \quad t\to\infty,
		\end{align*}
		where $\thbl$ is the boundary layer part, given by
		\begin{equation*}\label{BL-Ansatz-NL}
		\thbl = \frac{1}{t}\Theta_\mtop(x,t^{1/4}(1-z)) + \frac{1}{t}\Theta_\mbot(x,t^{1/4}z) + \text{l.o.t.}
		\end{equation*}
            with $\Theta_\mtop$ (resp. $\Theta_\mbot$) decaying exponentially as $Z_\mtop=t^{1/4}(1-z)\to\infty$ (resp. as $Z_\mbot=t^{1/4}z\to\infty$).
	\end{theorem}
	
A more precise version of the Theorem, including $H^s$ estimates on the remainder, will be provided in \cref{prop:BL-NL}.
	We note that $\|\thbl\|_{L^2(\Omega)} \lesssim (1+t)^{-9/8}$, so this result strongly suggests that the optimal decay of $\rho-\rho_\infty$ is like $t^{-9/8}$ in $L^2(\Omega)$, which is close to the rate $t^{-1}$ obtained in \cref{thm:stab}. \label{optimal-rate}
  \cref{thm:BL-nonlinear} also shows that the decay rate is dictated by boundary layers. 
	Nevertheless, it is not excluded that the non-linear dynamics drive the system to the case where these boundary layer terms always vanish, although we expect this behavior to be rather unlikely. Let us emphasize that the formation of boundary layers, let alone the construction of boundary layer profiles, had not been identified in previous works, even in the linear setting of \cref{thm:BL-linear}.
We believe that our analysis could be extended to the incompressible porous media (IPM) system \eqref{eq:Darcy}, for which similar boundary layers are expected to develop, see \cref{rmk:extension-IPM} below.
Let us also recall  that in the cases without boundaries (see \cite{elgindi,castro2019global,park2025}), the rate of decay of $\rho-\rho_\infty$ can be arbitrarily large, provided the initial data is sufficiently smooth. 

Let us now say a few words about the case when the initial data $\rho_0$ of \eqref{eq:ST} is such that $\rho_0-\rho_s $ or $\p_\bn(\rho_0-\rho_s)$ do not vanish on the boundary. We expect the scaling of the boundary layers to be different. Indeed, if the Ansatz of the linear case \eqref{BL-Ansatz-linear} is plugged into \eqref{eq:ST}, we find that the quadratic term becomes dominant close to the boundary, and cannot be balanced by other terms in the equation.
As a consequence, studying \eqref{eq:ST} when $\rho-\rho_s\notin H^2_0(\Om)$ goes beyond the scope of this paper. 
We expect that the boundary layer equations become nonlinear in this setting.

Note that \cref{thm:BL-nonlinear} requires more stringent assumptions on the initial data than \cref{thm:stab}, since the initial perturbation is assumed to be small in $H^{14}$ (rather than $H^6$), and its second normal derivative is also assumed to vanish on the boundary (\ie $\p_z^2 (\rho-\rho_s)\vert_{\p\Om}=0$). Actually, the latter condition can be slightly weakened, see \cref{prop:BL-NL} for a more precise statement.

\begin{remark}[Extension to the incompressible porous medium equation]
\label{rmk:extension-IPM}
We believe that \cref{thm:BL-nonlinear} could be extended to the IPM equation \eqref{eq:Darcy} when the initial datum $\rho_0$ is sufficiently smooth and such that $\rho_0-\rho_s \in H^1_0$ (\ie the trace of $\rho_0-\rho_s$ vanishes on the boundary). 
In this case, $(\rho(t)-\rho_s)\vert_{\p\Om}=0$ for all $t\geq0$ (see \cref{lem:traces}).

In this setting, the boundary layer Ansatz from \cref{thm:BL-nonlinear} should be replaced with
\[
\thbl=\frac{1}{t}\Thtop^\mathrm{IPM}(x, t^{1/2}(1-z))+ \frac{1}{t}\Thbot^\mathrm{IPM}(x, t^{1/2}z) + \text{l.o.t.},
\]
where the profiles $\Theta_a^\mathrm{IPM}$ for $a\in \{\mtop, \mbot\}$ satisfy
\[
\ba 
- \Theta_a^\mathrm{IPM} + \frac{1}{2} Z \p_Z \Theta_a^\mathrm{IPM}&= \p_x \Psi_a^\mathrm{IPM}\\
\p_Z^2  \Psi_a^\mathrm{IPM}&= \p_x \Theta_a^\mathrm{IPM}.
\ea
\]
This system should be compared with \eqref{eq:Thetaj-BL-EDP}, and is endowed with the boundary conditions
\[
\Psi_a^\mathrm{IPM}\vert_{Z=0}=  \Theta_a^\mathrm{IPM}\vert_{Z=0}=0,\qquad \p_Z^2 \Theta_a^\mathrm{IPM}\vert_{Z=0}=\gamma_a,
\]
where $\gamma_\mtop(x)=\lim_{t\to\infty} \p_z^2 \theta(t,x,1)$, $\gamma_\mbot(x)=\lim_{t\to\infty} \p_z^2 \theta(t,x,0)$. 

Therefore the situation is very similar to the one of \cref{thm:BL-nonlinear}: the main difference lies in the thickness of the boundary layer  ($t^{-1/2}$ for IPM vs. $t^{-1/4}$ for Stokes-transport), which is consistent with the order of the damping term ($\p_x^2 \Delta^{-1}$ for IPM vs. $\p_x^2 \Delta^{-2}$ for Stokes-transport).
Furthermore, if $\p_z^{2\ell} \theta_0|_{\p\Omega}=0$ for $0\leq \ell \leq k$ and for some $k\geq 1$, then the above Ansatz should be replaced by
\[
\thbl=\frac{1}{t^{k+1}}\Thtop^\mathrm{IPM}(x, t^{1/2}(1-z))+ \frac{1}{t^{k+1}}\Thbot^\mathrm{IPM}(x, t^{1/2}z) + \text{l.o.t.}.
\]
Note that this is consistent with the results of \cite{park2025} (see also \cite{castro2019asymptotic}), in which the author proves that
$\| \rho(t)-\rho_s\|_{L^2}\lesssim t^{-k - 1/2}$ under a slightly more stringent version of the previous assumption. 

However, if the trace of $\rho_0-\rho_s$ on the boundary does not vanish, the situation is different. In this scenario, the nonlinear terms are expected to be of leading order close to the boundary, and we expect that nonlinear boundary layers are created.

Interpolating between the IPM system and the Stokes-transport system, it is also natural to wonder what happens for fractional equations such as
\[
\p_t \theta + \bu\cdot \nabla \theta= \p_x^2 (-\Delta)^{-\alpha} \theta,
\]
with $\alpha\in (1,2)$. One should however define carefully the fractional operator $ (-\Delta)^{-\alpha}$ in this setting, since the domain $\T\times (0,1)$ is bounded in the vertical direction (the boundedness in the horizontal direction is not really an issue since we can rely on a Fourier definition of the fractional laplacian in the horizontal variable.) 
One canonical choice is to use a spectral definition of the fractional laplacian. However, in the present setting, there are two possible choices for the eigenbasis: the eigenfunctions of the laplacian, or the ones of the bilaplacian, described in \cref{lem:basis}. These two choices seem to lead to different operators, and in particular, they are incompatible with one another.

Therefore it seems better to consider the so-called ``restricted fractional laplacian'': for $\psi \in H^{2\alpha}(\Om)$ such that $\psi\vert_{\p\Om}=0 $ and $\p_z \psi\vert_{\p\Om}=0 $ if $\alpha>3/2$, extend $\psi$ by zero outside $\Om$, and define 
\[(-\Delta)^\alpha \psi:= C_\alpha \mathrm{PV} \int_{\T \times \mathbf{R}} \frac{\Delta \psi (x',z') - \Delta \psi (x,z)}{|(x,z)- (x',z')|^{2\alpha }}\rd x'\; \rd z'.\]
    The equation for $\psi$ then becomes 
    \[
    (-\Delta)^\alpha \psi= \p_x \theta\text{ in }\Om, \qquad \psi\vert_{\Om^c}=0.
    \]
The main advantage of this choice is to be compatible with the end-cases $\alpha=1$ (IPM) and $\alpha=2$ (Stokes-transport). However, due to the nonlocal nature of the fractional laplacian, having a description of the boundary layer formation seems much more involved.

\end{remark}

	\subsection*{Schemes of proofs} 
	Here we explain the main steps and difficulties of the proofs of \cref{thm:stab,thm:BL-linear,thm:BL-nonlinear}. 
	
	\paragraph{Rewriting of the equation.}
	Since perturbations of $\rho_s(z)=1-z$ are considered, it is natural to introduce the perturbation $\theta$ as
	\begin{equation*}
		\rho = \rho_s + \theta
	\end{equation*}
	with initial perturbation $\theta_0=\rho_0-\rho_s$.
	Substituting this expression in \eqref{eq:ST} and recalling that stratified states do not contribute to the velocity field in the Stokes equation, we obtain the following equation for the perturbation $\theta$:
	\begin{equation*}
		\left\{
		\ba
		\p_t\theta + \bu\cdot\nabla\theta &= u_z \\
		-\Delta \bu + \nabla p &= -\theta\be_z \\
		\div\,\bu &= 0 \\
            \bu|_{\p\Omega} &= \bzero \\
		\theta|_{t=0} &= \theta_0.
		\ea
		\right.
	\end{equation*}

    We note that we used the notation $\bu=(u_x,u_z)$ and that $x$ and $z$ indices always denote the horizontal and vertical components and never derivatives with respect to $x$ or $z$.
 
	The Stokes equation can be simplified by introducing the stream function of the divergence-free velocity field $\bu$ through $\bu = \nabla^\perp\psi = (-\p_z\psi,\p_x\psi)$. Substituting it in the Stokes equation and considering the curl of this equation, we get
	\begin{equation*} \label{eq:intro-bilap}
		\left\{
		\ba
		\p_t\theta + \bu\cdot\nabla\theta &= u_z \\
		\Delta^2\psi &= \p_x\theta \\
		\bu &= \nabla^\perp\psi \\
		\psi|_{\p\Omega} = \p_\bn\psi|_{\p\Omega} &= 0\\
            \theta|_{t=0} &= \theta_0.
		\ea
		\right.
	\end{equation*}
	Notice that this writing is consistent with the previous observation that any $z$ dependent perturbation of the density does not affect the velocity field.
	
	Once the steady states of \eqref{eq:ST} are identified as the stratified density profiles, \ie functions depending only on $z$, it is natural to decompose the perturbation $\theta(t,x,z)$ as the sum of its horizontal average $\bar\theta(t,z)$ and  its complement $\theta'(t,x,z)$ with zero horizontal average, following \cite{elgindi} and others:
	\begin{align*}
		\theta(t,x,z) &= \bar\theta(t,z) + \theta'(t,x,z) & \bar\theta(t,z) &= \frac{1}{2\pi}\int_0^{2\pi}\theta(t,x,z) \ud x.
	\end{align*}
	We note that contrary to \cite{elgindi,castro2019global}, $\bar \theta$ denotes the average rather than the fluctuation, as this seems a more natural notation. In particular our notation is comparable to the standard notation used for the Reynolds-averaged Navier--Stokes equations.
	
	This decomposition is actually orthogonal in any Sobolev space $H^m$ and one can project the transport equation onto the two appropriate complementary subspaces, leading to
	\begin{equation} \label{eq:STi}
		\left\{
		\ba
		\p_t\theta' + (\bu\cdot\nabla\theta')' &= (1-\p_z\bar\theta)u_z & \theta'|_{t=0} &= \theta'_0 \\
		\p_t\bar\theta + \overline{\bu\cdot\nabla\theta'} &= 0 & \bar\theta|_{t=0} &= \bar\theta_0 \\
		\Delta^2\psi &= \p_x\theta' & \psi|_{\p\Omega} &= 0 \\
		\bu &= \nabla^\perp\psi & \p_\bn\psi|_{\p\Omega} &= 0. \\
		\ea
		\right.
	\end{equation}
	Although more complicated at first sight, this equation allows us to distinguish the evolution of $\theta'$ and of the average perturbation $\bar\theta$. This is needed since the whole perturbation cannot be expected to decay in Sobolev spaces due to its pure transport. Only the average-free part $\theta'$ is decaying.
	
	\paragraph{Toy problem on the torus.}
	In order to get an intuition of the decay of $\|\theta'\|_{L^2(\Omega)}$ and to highlight the specific difficulties of our work, we will first explain the strategy in the case when the problem is set on the torus, in order to avoid the issues associated with the boundary conditions. More precisely, we consider the following linear problem for $\theta'$ on the torus $\T^2$:
	\begin{equation} \label{eq:ST-torus}
		\left\{
		\ba
		\p_t \theta' &= (1-G) \p_x\psi + S,\\
		\Delta^2 \psi &= \p_x \theta' ,\\
		\theta'|_{t=0} &= \theta'_0,
		\ea
		\right.
	\end{equation}
	where $G$ is a given small function of $t$ and $z$, whose finality is to be replaced by $\p_z\bar\theta$. 
	The source term $S$, which will include the nonlinearities of the system, will be omitted in this short presentation for simplicity.
	Note that \eqref{eq:ST-torus} differs from our original system through the periodic boundary conditions on $\psi$ in the vertical variable.
	The choice of periodic boundary conditions simplifies the analysis in several ways, which we will detail below.

For any $s\geq0$, applying $\Delta^{\frac{s}{2}}$ to the first equation of \eqref{eq:ST-torus} and projecting on $\Delta^{\frac{s}{2}}\theta'$, we obtain, after several integrations by parts in the right-hand side,
	\begin{equation} \label{eq:intro-Hs}
		\ba
		\frac{1}{2} \frac{\rd}{\rd t} \|\Delta^{\frac{s}{2}}\theta'\|_{L^2}^2 
		&= - \int_{\T^2}\Delta^{\frac{s}{2}}((1-G)\psi) \Delta^{\frac{s}{2}+2}\psi \\
		&= - \int_{\T^2}\Delta^{\frac{s}{2}+1}((1-G)\psi) \Delta^{\frac{s}{2}+1}\psi \\
		&\leq -(1-\bar C \|G\|_{H^{s+2}}) \|\Delta^{\frac{s}{2}+1}\psi\|_{L^2}^2.
		\ea
	\end{equation}
	where $\bar C$ is a universal constant. As a consequence, if $\bar C\|G(t)\|_{H^{s+2}}<1$, then the $H^s$ norm of $\theta'$ is non-increasing, and whence uniformly bounded.
	
	Then, the decay of $\|\theta'(t)\|_{L^2}$ is deduced by using the following Gagliardo--Nirenberg interpolation inequality, which in the case of the torus can be proved simply by Fourier analysis:
	\begin{equation} \label{eq:intro-interp}
		\|\p_x^{-1}\Delta^{2}\phi\|_{L^{2}}^2\lesssim \frac1{K}\Vert\Delta\phi\Vert_{L^{2}}^2+K^{2}\Vert\p_x^{-3}\Delta^{4}\phi\Vert_{L^{2}}^2
	\end{equation}
    where $\p_x^{-1} f$ denotes the anti-derivative of $f$ with null horizontal average, and $K>0$  is an arbitrary positive constant.

	More precisely, combining \eqref{eq:intro-Hs} %
	and \eqref{eq:intro-interp} with $\phi= \Delta^{\frac{r}{2}} \psi$ for some $r\geq 0$ leads to:
	\begin{equation*}\ba
		\frac{\rd}{\rd t} \|\Delta^{\frac{r}{2}}\theta'\|_{L^{2}}^{2}
		&\lesssim-\|\Delta^{\frac{r}{2}+1}\psi\|_{L^{2}}^{2}\\
	&	\lesssim K^{3}\Vert\Delta^{\frac{r}{2} +4}\p_x^{-3}\psi\Vert_{L^{2}}^{2}-K\|\p_x^{-1}\Delta^{\frac{r}{2} + 2}\psi\|_{L^{2}}^{2}\\
&		\lesssim K^{3}\Vert\Delta^{\frac{r}{2}+2}\partial_{x}^{-2}\theta'\Vert_{L^{2}}^{2}-K\|\Delta^{\frac{r}{2}}\theta'\|_{L^{2}}^{2}
		\ea
	\end{equation*}
	recalling that $\Delta^2\psi = \p_x\theta'$. Taking $K \simeq (1+t)^{-1}$ we deduce:
	\beq\label{diff-eq-Hs}
	\frac{\rd}{\rd t} \|\Delta^{\frac{r}{2}}\theta'\|_{L^{2}}^{2} + \frac{3}{1+t} \|\Delta^{\frac{r}{2}}\theta'\|_{L^{2}}^{2} \lesssim \frac{1}{(1+t)^3} \|\Delta^{\frac{r}{2}+2}\partial_{x}^{-2}\theta'\|_{L^{2}}^{2}.
	\eeq
  Note that here the factor 3 could be made arbitrarily large by taking a larger multiplicative constant in $K$, but any constant strictly larger than 2 is sufficient for the argument.

	Since $\Delta^{\frac{r}{2}+2}\partial_{x}^{-2}\theta'$ is uniformly bounded in $L^2(\T^2)$ by $\|\p_x^{-2}\theta'_0\|_{H^{r+4}}$, this integrates into
	\begin{equation} \label{eq:intro-decay}
	\forall t\geq 0,\quad  \|\Delta^{\frac{r}{2}}\theta'(t)\|_{L^{2}}\lesssim \frac{\|\p_x^{-2}\theta'_0\|_{H^{r+4}}}{1+t}.
	\end{equation} 
Note that the index of regularity $r$ is arbitrary. Hence, plugging this estimate back into \eqref{diff-eq-Hs} and using an induction argument, we find that for any $\alpha\geq 0$, $r\geq 0$, 
\beq
\label{decay-torus-Hs}
\forall t\geq 0,\quad \| \Delta^{\frac{r}{2}}\theta'(t)\|_{L^2}\leq (1+t)^{-\alpha/2} \|\p_x^{-\alpha}\theta_0'\|_{H^{r+2\alpha}}.
\eeq
Let us emphasize that when $G=0$, this estimate can be proved directly from the Fourier representation formula 
\[
	\theta'(t,x,z)=\sum_{\substack{\bk\in \Z^2\\ k_x\neq 0}} \hat\theta_{\bk}(0)\exp\left(- \frac{k_x^2}{|\bk|^4}t\right) \exp(i\bk\cdot(x,z)).
	\]
	Hence the decay rate can be expected to be somewhat optimal. Moreover, in the case of the torus,  the rate of decay can be as large as desired, the cost being the regularity required on $\theta_0'$. Note that for $r=0$, $\alpha=2$, we find the decay rate announced in \cref{thm:stab}.

	\paragraph{Difficulties with Dirichlet boundary conditions.}
	Let us now explain the main differences between \eqref{eq:ST-torus} on $\T^2$ and the original system \eqref{eq:STi} on $\Omega=\T\times(0,1)$.
	The strategy will be identical.
	We first prove a uniform bound for $\theta'$ in $H^4(\Omega)$, and then use interpolation inequalities together with the energy estimate to obtain the decay estimate \eqref{eq:intro-decay}. However  the derivation of the different bounds will be substantially more involved.
	
	We shall prove the uniform $H^4(\Omega)$ bound for $\theta'$ directly from the equation without spectral analysis. More precisely, the estimate \eqref{eq:intro-Hs} remains valid for $s=0$ since $\psi|_{\p\Omega}=\p_\bn\psi|_{\p\Omega}=0$. Higher order uniform estimates in $H^s(\Omega)$ fail in general due to non-vanishing terms on the boundary. The question is therefore when the integration by parts done in \eqref{eq:intro-Hs} can be performed.
	The traces of $\theta'$ and $\p_\bn\theta'$ being zero, the traces of $\Delta^2\psi$ and $\p_\bn\Delta^2\psi$ are also vanishing (see \cref{sub:traces}) so integrations by parts in \eqref{eq:intro-Hs} can be done for $s=4$ provided $G=0$. Therefore a uniform $H^4$ bound can be deduced when $G=0$.
	When $G$ is nonzero, some traces no longer vanish. The strategy will be to treat them perturbatively, \ie not performing integration by parts on $\Delta^2(G\psi)\Delta^3\psi$.
A similar interpolation argument (\cref{lem:interpol}) allows us to then deduce the analogue of \eqref{eq:intro-decay} \ie that $\|\theta'\|_{L^2}$ is bounded by $(1+t)^{-1}$.

However, note that the higher decay \eqref{decay-torus-Hs} for $r>0$ and $\alpha>2$ does not hold in general, as \cref{thm:BL-nonlinear} shows. Indeed, the decay rate is prescribed by the boundary layer part of the solution, for which we have $\|\Delta^{\frac{r}{2}} \thbl\|_{L^2}\propto (1+t)^{-1+ \frac{r}{4} - \frac{1}{8}}$. Hence the $H^s$ norm of $\theta'$ for $s\geq 5$ is not expected to be bounded.

 {Finally, let us mention that proving some time integrability on the velocity field is crucial in order to obtain the convergence of $\bar \theta$.
As a consequence, the linear decay from \eqref{eq:intro-decay} is not entirely sufficient to complete the proof of \cref{thm:stab}. In previous works, this higher decay on the velocity field was obtained either thanks to high regularity bounds, or by taking advantage of the Fourier representation of the solution. Since none of these tools are available here, we rely on a different argument, involving bounds on the time derivative of $\theta'$.
}

	\begin{remark}[About the spectral decomposition]
	Since the equation is no longer set on the torus, but rather in the domain $\Omega=\T\times(0,1)$ endowed with boundary conditions, we can no longer perform a
	(discrete) Fourier transform in the vertical variable.
	However it is possible to analyze explicitly the eigenfunctions of the operator
	\beq\label{def:L}
	L:\theta\in L^2\mapsto \p_x \psi\in L^2,\quad \text{where } \Delta^2 \psi=\p_x \theta,\quad \psi\vert_{\p\Om}=\p_\bn\psi\vert_{\p\Om}=0,
	\eeq
	and show that the eigenvalues $(\lambda_\bk)_{\bk\in \Z\times \N}$ of the operator $L$ behave asymptotically as $k_x^2/|\bk|^4$ (see \cref{lem:basis}), so that the estimate \eqref{eq:intro-decay} remains true. Details on the spectral analysis are presented in the PhD thesis of Antoine Leblond \cite{leblondPhD}.
	
	\end{remark}

	\paragraph{Bootstrap.}
	The last step of the proof consists in bringing the previous linear analysis into the full nonlinear system. Intuitively, the strategy is the following: denote by $(0,T^*)$ the maximal time interval over which $\|\theta'\|_{L^2}\leq B(1+t)^{-1}$ and $\|\theta'\|_{H^4}\leq B$ are valid with a constant $B$. In fact more estimates need to be included in the bootstrap argument for technical reasons, see \eqref{eq:boot-hyp}.
	On this time interval, the quadratic terms can be treated perturbatively, provided $\|\theta_0\|_{H^4}$ is sufficiently small. Hence the bootstrap estimates hold with a constant which is better than $B$, and thus $T^*=\infty$. It follows that $\theta'$ converges towards zero in $L^2$, and that the time derivative of $\bar \theta$ is integrable. Hence $\bar\theta$ has a limit in $L^2$ as $t\to\infty$. This is the main part of the proof which is detailed in \cref{sub:bootstrap}.
	
	\paragraph{Identification of the limit.}
	Since $\theta'$ converges to zero in any $H^m$ for $m<4$ as $t\to\infty$ and $\bar\theta$ has a limit in $L^2$ as $t\to\infty$, the whole density $\rho=\rho_s+\theta=\rho_s+\bar\theta+\theta'$ converges to some limit $\rho_\infty=\rho_s+\bar\theta_\infty$ in $L^2$ and $\rho_\infty$ depends only on $z$. The term $\p_z\theta$ is small compared to $\p_z\rho_s=-1$, and so is its limit $\p_z\bar\theta_\infty$. Whence $\rho_\infty$ is strictly decreasing with respect to $z$, as is $\rho_s$.
	The transport of the density by the divergence-free field $\bu$ ensures that the level sets of $\rho$ are preserved by the time evolution, and by strong convergence this is also the case for the limit $\rho_\infty$. According to rearrangement theory, $\rho_\infty$ is therefore a rearrangement of $\rho_0$. One can show that there exists a unique decreasing vertical rearrangement of $\rho_0$, hence $\rho_\infty$ is uniquely determined.
	This part of the proof is detailed in \cref{sub:finalstate}.

	\paragraph{Linear boundary layers for system \eqref{eq:ST-linear}.}\label{scheme:linearBL}
	Let us now give a sketch of the proof of \cref{thm:BL-linear}. 
	We start with rather simple observations:
	\begin{itemize}
		\item First, it follows from the equation that
		$
		\p_t \theta\vert_{\p\Om}=\p_t \p_\bn \theta\vert_{\p\Om}=0.
		$
		Therefore, for all $t\geq 0$, 
		\[
		\theta\vert_{\p\Om}(t)= \theta_0\vert_{\p\Om},\quad \p_\bn  \theta\vert_{\p\Om}(t)= \p_\bn \theta_0\vert_{\p\Om}.
		\]

\item Taking the horizontal average of the evolution equation, we find that $\p_t \bar \theta=0$, and thus $\bar \theta(t)=\bar \theta_0$. Hence we focus on  the long-time behavior of $\theta'$.
  
		\item Let us denote by $(b_\bk)_{\bk\in \Z\times \N^*}$ the basis of eigenvectors of  the operator $L$ defined in \eqref{def:L} (see \cref{lem:basis}.) Then  we can always write
		\[
		\theta'(t)=\sum_{\bk\in \Z^* \times \N} \exp(-\lambda_\bk t) \widehat{\theta'_\bk}(t=0) b_\bk,
		\]
		with $( \widehat{\theta'_\bk}(t=0))_{\bk\in  \Z^* \times \N}\in \ell^2 $.
        We recall that $\lambda_\bk$ behaves asymptotically as $|k_x|^2/|\bk|^4$.
		It then follows from Lebesgue's dominated convergence theorem that $\theta'(t)\to 0$ in $L^2$ as $t\to \infty$.

	\end{itemize}
	Therefore $\theta'(t)$ vanishes in $L^2$ while keeping a constant - and non-zero - value on the boundary. As a consequence, 
	it is reasonable to expect that boundary layers are formed in the vicinity of $z=0$ and $z=1$ as $t\to \infty$. 
We then plug the Ansatz \eqref{BL-Ansatz-linear} into \eqref{eq:ST-linear} and identify the profiles $\Theta^0_\mtop$, $\Theta^0_\mbot$. 
	The role of $\Thtop^0$ (resp. of $\Thbot^0$) is to lift the trace of $\theta_0'$ at the top boundary $z=1$ (resp. at the bottom boundary $z=0$).
We find that these two profiles satisfy an ODE, with boundary conditions given  by
	\begin{align*}
		\Thtop^0(x,Z=0) &= \theta'_0(x,z=1),& \p_Z \Thtop^0(x,Z=0) &=0,\\
		\Thbot^0(x,Z=0) &= \theta'_0(x,z=0),& \p_Z \Thbot^0(x,Z=0) &=0.
	\end{align*}
	In a similar way, the next order boundary layer terms $\Thtop^1$ and $\Thbot^1$ lift the traces of $\p_\bn \theta_0$ on $\p\Om$.
	Hence the first step is to construct explicitly the boundary layer profiles in terms of $\theta_0$.
	By construction, the remainder $\theta' - \thbl$ vanishes on the boundary, together with its normal derivative. 
	We can then apply the decay analysis presented above to the remainder $\theta' - \thbl$, and we find that $\|(\theta' - \thbl )(t)\|_{L^2}\lesssim (1+t)^{-1}$.

	\paragraph{Boundary layers for system \eqref{eq:ST}.}
	We now turn towards \cref{thm:BL-nonlinear}. Note that the boundary layer term in \cref{BL-Ansatz-NL} is smaller than in \cref{BL-Ansatz-linear}. This is directly linked to the fact that under the assumptions of \cref{thm:BL-nonlinear}, $\theta=\rho-\rho_s$ vanishes on the boundary, together with its normal derivative.
	Therefore, the boundary layer term $\thbl$ (or rather $\Delta^2 \thbl$) now lifts the traces of $\Delta^2 \theta'$ and $\p_\bn \Delta^2 \theta'$.
	The overall strategy is the same as the one described above: we first identify the boundary layer part of the solution by rigorously constructing the boundary layer profiles $\Theta^j_\mbot$ and $\Theta^j_\mtop$. We then prove some decay estimates on the remainder $\threm= \theta'-\thbl$, noticing that $\Delta^2 \threm$ satisfies assumptions that are very close to the ones of \cref{thm:stab}.
	Note that the higher decay we obtain on $\threm$ is the main reason behind the strong regularity requirements on $\rho_0$.

	However, there are several new conceptual and technical  difficulties compared with \cref{thm:BL-linear}. The main one lies in the fact that the traces $\Delta^2\theta'\vert_{\p\Om}$ and $\p_\bn \Delta^2\theta'\vert_{\p\Om}$ are not constant with respect to time. They merely have a finite limit as $t\to \infty$.
	Hence we need to find an asymptotic expansion in powers of $(1+t)^{-1/4}$ for $\Delta^2\theta'\vert_{\p\Om}$ and $\p_\bn \Delta^2\theta'\vert_{\p\Om}$ as $t\gg 1$.
	The main boundary layer profiles $\Theta^0_\mbot$ and $\Thtop^0$ will lift the first term in this expansion (\ie the long time limit of   $\Delta^2\theta'\vert_{\p\Om}$), whereas the 
	next order profiles $\Thtop^j, \Thbot^j$ for $j\geq 2$ will lift the lower order terms. 
    Furthermore, in order to prove that $\Delta^2\theta'\vert_{\p\Om}$ converges in $H^s$ as $t\to \infty$ for some sufficiently large $s$, we shall need high regularity bounds on $\theta'$.
    Eventually, the proof of \cref{thm:BL-nonlinear} involves two nested bootstrap arguments: one on $\theta'$, which allows us to construct the boundary layer term $\thbl$ on the interval on which the bootstrap assumption is satisfied, and a second one on the remainder $\theta'-\thbl$, on a possibly smaller interval.

	\paragraph{Notation.} Throughout the paper, we write $A\lesssim B$ whenever there exists a universal positive constant $C$ such that $A\leq C B$.

	\section{Long time stability of stratified profiles\texorpdfstring{: proof of \cref{thm:stab}}{}}
	\label{sec:stab}

	This section is devoted to the proof of \cref{thm:stab}.
	The proof follows the steps highlighted in the introduction: we decompose $\theta$ into $\theta=\bar\theta + \theta'$, and we prove that $\theta'$ vanishes in $L^2$ with algebraic decay, while $\bar\theta$ converges in $L^2$ towards a profile $\bar\theta_\infty(z)$.
	To that end, we first study the linearized Stokes-transport system around a solution $\theta$ close to an affine profile.
	Thanks to a crucial interpolation inequality (see \cref{lem:interpol}), which somehow replaces the spectral decomposition in the periodic setting, we  quantify the $L^2$ decay of solutions of the linearized equation with a source term (see \cref{prop:bb-decay}).
	We then use a bootstrap argument to show that the decay predicted by the linear analysis persists for the nonlinear evolution.
	This allows us to prove that $\theta'(t)\to 0 $ and that $\bar\theta \to \bar\theta_\infty$ in $H^s(\Om)$ as $t\to \infty$, for all $s<4$.
	Eventually, we identify the asymptotic profile $\bar\theta_\infty$ in terms of the initial data.
	
	The organization of this section is the following. 
	We start in \cref{sub:traces} with some preliminary remarks concerning the traces of $\theta$ and $\p_\bn \theta$. 
	We then turn towards the analysis of the linearized system in \cref{sub:linearized}.
	The bootstrap argument is presented in  \cref{sub:bootstrap}.
	Eventually, we prove in \cref{sub:finalstate} that $\rho_\infty$ is the rearrangement of the initial data $\rho_0$.
	
	\subsection{Vanishing traces for \texorpdfstring{$\theta'$ and $\p_\bn \theta'$}{θ' and ∂\_n θ'}}
	\label{sub:traces}
	
 We prove here the following preliminary result:
	
	\begin{lemma}
		Let $\theta_0\in H^m(\Om)$ with $m\geq 3$, and let $\theta\in L^\infty_\text{loc}(\R_+, H^m)$ be the solution of \eqref{eq:STi}.
		Assume that $\theta_0=\p_\bn\theta_0=0$ on $\p\Om$. Then for all $t\geq 0$, 
		\[
		\theta(t)\vert_{\partial\Omega}= \p_\bn \theta(t)\vert_{\partial\Omega}=0.
		\]
		If additionally $\p_z^2 \bar\theta_0=0$ on $\p\Omega$, then $\p_z^2 \bar\theta(t)\vert_{\partial\Omega}=0$ for all $t\geq 0$.
		\label{lem:traces}
	\end{lemma}
	\begin{remark}
		If $\rho_0 \in H^m_0(\Omega)$ then the solution $\rho(t)$ of \eqref{eq:ST} belongs to  $H^m_0(\Omega)$ for all times. Indeed, the solution of the transport equation can be written as
		\begin{equation*}
			\rho(t) = \rho_0(\bX(t)^{-1}),
		\end{equation*}
		where $\bX  : \R_+\times\Omega \rightarrow \Omega$ is the characteristic function associated to $\bu$, defined as the solution of the ordinary differential equation
		\begin{equation*}
			\left\{
			\ba
			\dfrac\rd{\rd t}\bX(t) &= \bu(t,\bX(t)), \\
			\bX(0) &= \Id_\Omega.
			\ea
			\right.
		\end{equation*}
		We recall that $\bX(t)$ is a diffeomorphism of $\Omega$ for all times $t\in\R_+$. Since $\bu(t) \in H^1_0(\Omega)$ due to the homogeneous Dirichlet condition, the boundary $\p\Omega$ is stable for the characteristic function at all time $t>0$.
		In other words, $\bX(t)|_{\p\Omega} = \Id_{\p\Omega}$, and consequently $\bX(t)^{-1}|_{\p\Omega} = \Id_{\p\Omega}$. It follows that if $\rho_0\in H^1_0(\Omega)$, then $ \rho(t)|_{\p\Omega}=0$ for all $t\geq 0$. The claim for higher values of $m$ follows easily by induction.

		Note that the assumptions of \cref{lem:traces} are  different since $\rho_0=1-z +\theta_0$ does not vanish on the boundary. 
		
	\end{remark}

	\begin{proof}
We recall (see \cref{thm:wp}) that $\theta\in C(\R_+, H^m)\cap C^1(\R_+, H^{m-1})$. Therefore, taking the trace of \eqref{eq:ST}, we get
		\begin{equation*}
			\p_t\theta|_{\p\Omega} + \bu|_{\p\Omega}\cdot\nabla\theta|_{\p\Omega} = u_z|_{\p\Omega},
		\end{equation*}
		where $\bu|_{\p\Omega}=\bzero$. Hence $\p_t\theta|_{\p\Omega} = 0$ and the trace of $\theta$ is constant in time, equal to $0$. Since horizontal derivatives preserve this property, we even have  $\p_x^\ell\theta|_{\p\Omega} = 0$ for any $\ell$. 
		Let us now consider the normal derivative. We recall that $\p_\bn$ coincides (up to a sign) with $\p_z$.
		Applying one vertical derivative to the equation,
		\begin{equation*}
			\p_t\p_z\theta + \p_z\bu\cdot\nabla\theta + \bu\cdot\nabla\p_z\theta = \p_z u_z
		\end{equation*}
		where
		\begin{equation*}
			\p_z\bu|_{\p\Omega}\cdot\nabla\theta|_{\p\Omega} = \p_zu_x|_{\p\Omega}\p_x\theta|_{\p\Omega} + \p_zu_z|_{\p\Omega}\p_z\theta|_{\p\Omega}.
		\end{equation*}
		We recall that $\p_x\theta|_{\p\Omega} = 0$ and we use the divergence free condition to observe $\p_zu_z|_{\p\Omega}=-\p_xu_x|_{\p\Omega}=0$. In the end we get $\p_z\theta|_{\p\Omega}=0$ for all times, hence $\theta\in H^2_0(\Omega)$. Trying to go further, applying the same ideas, we get
		\begin{equation*}
			\p_t\p_z^2\theta|_{\p\Omega} = \p_z^2u_z|_{\p\Omega}.
		\end{equation*}
		However, $\p_z^2 u_z$ does not vanish on $\p\Om$, and therefore we cannot iterate the argument.
		Nevertheless, we get
		\begin{equation*}
			\p_t\p_z^2\bar\theta|_{\p\Omega} = \frac1{2\pi}\int_{\T}\p_z^2u_z|_{\p\Omega}=-\frac1{2\pi}\int_\T\p_x\p_zu_x|_{\p\Omega} = 0.
		\end{equation*}
		Note that for higher orders of derivation, we cannot infer any cancellation in general.
			\end{proof}
	
	\begin{definition}
		In the rest of the paper, we will set
		\[
		G(t,z)=\p_z \bar\theta(t,z).
		\]
	Under the assumptions of \cref{lem:traces}, we infer that $G\vert_{\p\Om}=\p_z G\vert_{\p\Om}=0$.
	\end{definition}

	\subsection{Study of the linearized system}
	\label{sub:linearized}
	This subsection is concerned with the study of the linear system
	\begin{equation} \label{eq:bb}
		\left\{
		\begin{aligned}
			\p_t\theta' &= (1-G)\p_x\psi +  S \quad \text{ in } (0, +\infty)\times \Om,\\
			\Delta^2 \psi &= \p_x\theta' \quad \text{ in } (0, +\infty)\times \Om,\\
			\psi|_{\p\Omega} &= \p_\bn\psi|_{\p\Omega} = 0,\quad 
			\theta'\vert_{t=0}=\theta_0',
		\end{aligned}
		\right.
	\end{equation}
	which is satisfied by $(\theta',\psi)$  in the first place, with $G=\p_z\bar\theta$ %
	and $S=-\left(\nabla^\perp\psi\cdot\nabla\theta'\right)'$. 
	It will also be satisfied for various derivatives of $(\theta', \psi)$  with different source terms $S$. The term $G$ will always be $\p_z\bar\theta$. 
	
	Our goal is to analyze the long time behavior of $\theta'$, under suitable decay assumptions on $S$. 
	For later purposes, we have decomposed our results into several separate statements, whose proofs are postponed to the end of the section.
	The first one is a uniform $L^2$ bound on the solutions when the source term is time integrable:
	\begin{lemma}[Uniform $L^2$ bound on solutions of the linearized system] \label{prop:bb-uniformbound}
		Let $G\in L^\infty(\R_+, H^2)$, $S\in L^\infty(\R_+, L^2)$, and $\theta_0'\in L^2$. Let $\theta'\in L^\infty(\R_+, L^2)$ be the unique  solution of \eqref{eq:bb}. Assume that $S$ can be decomposed as $ S=S_\perp+S_\para$ satisfying for some $\sigma,\delta>0$ and any $t\geq 0$
		\begin{equation} \label{eq:bb-S}
		\int_\Omega S_\perp(t,x)\theta'(t,x)\ud x = 0, \qquad \| S_\parallel(t)\|_{L^2} \lesssim \frac{\sigma}{(1+t)^{1+\delta}}.
		\end{equation}
		
		Thus, there exists a universal constant $\gamma_0 \in (0,1)$ such that if 
		\begin{equation} \label{eq:bb1-G}
			\|G\|_{H^2} \leq \gamma_0,
		\end{equation}
		then
		\begin{equation*}\label{eq:bb-result}
			\|\theta'\|_{L^2} \leq \|\theta_0'\|_{L^2} + C_\delta\sigma.
		\end{equation*}
	\end{lemma}
	
	\begin{remark}
		
		The term $S_\perp $ will often have the structure $S_\perp=\bu\cdot\nabla\theta'$: indeed, provided $\bu$ and $\theta'$ have sufficient regularity, the divergence free and homogeneous Dirichlet conditions ensure that
		\begin{equation*}
			\int_\Omega(\bu\cdot\nabla\theta')\theta' = \frac12\int_\Omega\bu\cdot\nabla|\theta'|^2 = -\frac12\int_\Omega\div\,\bu |\theta'|^2 + \frac12\int_{\p\Omega}\bu\cdot\bn |\theta'|^2 = 0.
		\end{equation*}

	\end{remark}

	Our second result, which is at the core of \cref{thm:stab}, gives a quantitative algebraic decay on $\theta'$:
	\begin{proposition} \label{prop:bb-decay}
			There exists a universal constant $\gamma_0\in(0,1)$ 
		such that the following result holds.	Let $T>0$, $G\in L^\infty(\R_+, H^2)$, $S\in L^\infty(\R_+, L^2)$, and $\theta_0'\in L^2$ such that $\p_x^{-2} \theta_0' \in H^4$. Let $\theta'\in C(\R_+, H^2)$ be the unique  solution of \eqref{eq:bb}. Assume that $\theta'$ and $\p_\bn\theta'$ vanish on $\p\Omega$, and that
		$ S$ decomposes  into $S=S_\perp + S_\para + S_\Delta$ with for some $\sigma,\delta >0$ and all $t\in [0, T]$,
		\begin{equation} \label{eq:bb2-S}
			\int_\Omega S_\perp (t,\bx)\theta' (t,\bx)\ud \bx= 0, \qquad \|S_\para(t)\|_{L^2} \leq \frac \sigma{(1+t)^{1+\delta}}, \qquad \|S_\Delta(t)\|_{L^2} \lesssim \frac{\|\Delta\psi\|_{L^2}}{(1+t)^{1/2}}.
		\end{equation}
		Assume moreover that $G$ satisfies \eqref{eq:bb1-G}, and that there exist $A,\alpha\geq0$ such that for all $t\in [0, T]$
		\begin{equation} \label{eq:bb2-A}
			\|\Delta^2\p_x^{-2}\theta'(t)\|_{L^2} \leq \frac{A}{(1+t)^\alpha}.
		\end{equation}
		Then
		\begin{equation*}
			\|\theta'(t)\|_{L^2}\lesssim \frac{\|\theta_0'\|_{L^2}+A+\sigma}{(1+t)^{\min(1+\alpha,\delta)}}\quad \forall t\in [0,T].
		\end{equation*}
	\end{proposition}
	In order to prove this quantitative decay, we shall need to analyze the structure of the dissipation term
	\[
	-\int_{\Om } \p_x \psi \theta'=\int_\Om |\Delta \psi|^2.
	\]
	In previous works for different but related models \cite{castro2019global}, at this stage, an explicit spectral decomposition of the solution was used, relying on Fourier series.
	Note that such a spectral decomposition is also available for the operator $\Delta^{-2} \p_x^2$ (see \cref{lem:basis}).
	However, since we cannot interpolate for an arbitrary regularity, we choose here to use a different approach.
	We replace this spectral analysis with the following result, which can be seen as an interpolation Lemma. It is noteworthy that in spite of its deceitfully simple form (and proof), this Lemma provides the correct scaling for the solutions.
	
	\begin{lemma} \label{lem:interpol}
		For any $\ell\geq0$, and for all $\psi\in H^{8+\ell-2}(\Omega)$ satisfying
		\[
		\Delta^{2}\psi|_{\partial\Omega}=\partial_\bn\Delta^{2}\psi|_{\partial\Omega}=0\,,
		\]
		we have for all $K>0$
		\[
		\|\p_x^{\ell-1}\Delta^{2}\psi\|_{L^{2}}^2\lesssim \frac1{K}\Vert\p_x^\ell\Delta\psi\Vert_{L^{2}}^2+K^{2}\Vert\p_x^{\ell-3}\Delta^{4}\psi\Vert_{L^{2}}^2.
		\]
	\end{lemma}
	\begin{proof}
		Since $\Delta^{2}\psi$ and $\partial_\bn\Delta^{2}\psi$ vanish on the boundary $\partial\Omega$, we have after three integrations by parts
		\begin{align*}
			\|\p_x^{\ell-1}\Delta^{2}\psi\|_{L^{2}}^{2} & =-\int_{\Omega}\p_x^{\ell}\Delta^{2}\psi\p_x^{\ell-2}\Delta^{2}\psi=\int_{\Omega}\nabla\p_x^{\ell}\Delta\psi\cdot\nabla\p_x^{\ell-2}\Delta^{2}\psi\\
			& =-\int_{\Omega}\p_x^{\ell}\Delta\psi\p_x^{\ell-2}\Delta^{3}\psi\leq\Vert\p_x^{\ell}\Delta\psi\Vert_{L^{2}}\Vert\p_x^{\ell-2}\Delta^{3}\psi\Vert_{L^{2}}.
		\end{align*}
		On another hand, we also have by integrations by parts
		\begin{align*}
			\Vert\p_x^{\ell-2}\Delta^{3}\psi\Vert_{L^{2}}^{2} & =-\int_{\Omega}\p_x^{\ell-1}\Delta^{3}\psi\p_x^{\ell-3}\Delta^{3}\psi=\int_{\Omega}\nabla\p_x^{\ell-1}\Delta^{2}\psi\cdot\nabla\p_x^{\ell-3}\Delta^{3}\psi\\
			& =-\int_{\Omega}\p_x^{\ell-1}\Delta^{2}\psi\p_x^{\ell-3}\Delta^{4}\psi\leq\Vert\p_x^{\ell-1}\Delta^{2}\psi\Vert_{L^{2}}\Vert\p_x^{\ell-3}\Delta^{4}\psi\Vert_{L^{2}}.
		\end{align*}
	Hence, using the second bound in the first inequality, we obtain
        \[
        \|\p_x^{\ell-1}\Delta^2\psi\|_{L^2}^2\leq \|\p_x^\ell\Delta\psi\|_{L^2}\|\p_x^{\ell-1}\Delta^2\psi\|_{L^2}^{\frac12}\|\p_x^{\ell-3}\Delta^4\psi\|_{L^2}^{\frac12}.
        \]
Gathering the similar terms on the left-hand side and applying Young's inequality yields, for any constant $K>0$,
		\begin{align*}
		\Vert\p_x^{\ell-1}\Delta^{2}\psi\Vert_{L^{2}}^{2}&\leq\Vert\p_x^{\ell}\Delta\psi\Vert_{L^{2}}^{\frac{4}{3}}\Vert\p_x^{\ell-3}\Delta^{4}\psi\Vert_{L^2}^{\frac{2}{3}} \\
        &\lesssim \left(K^{-\frac23}\|\p_x^\ell\Delta\psi\|_{L^2}^{\frac43}\right)^{\frac32} + \left(K^{\frac23}\|\p_x^{\ell-3}\Delta^4\psi\|_{L^2}^{\frac23}\right)^3 \\
        &\lesssim\frac{1}{K}\Vert\p_x^{\ell}\Delta\psi\Vert_{L^{2}}^{2}+K^{2}\Vert\p_x^{\ell-3}\Delta^{4}\psi\Vert_{L^{2}}^2.
		\end{align*}
	\end{proof}

	Let us now turn towards the proof of \cref{prop:bb-uniformbound} and \cref{prop:bb-decay}.

	\begin{proof}[Proof of \cref{prop:bb-uniformbound}]
		The energy estimate in \eqref{eq:bb} writes
		\begin{equation*}
			\frac12\frac\rd{\rd t}\|\theta'\|_{L^2}^2 = \int_\Omega(1-G)\p_x\psi\theta' + \int_\Omega  S\theta'.
		\end{equation*}
		A few integrations by parts provide, since $\psi\vert_{\partial\Omega}=\p_\bn \psi\vert_{\partial\Omega}=0$,
		\begin{equation} \label{eq:bb-ipp}
			\begin{aligned}
				\int_\Omega(1-G)\p_x\psi\theta' &= -\int_\Omega(1-G)\psi\Delta^2\psi 
								\\
				&=-\int_\Omega\Delta((1-G)\psi)\Delta\psi .
							\end{aligned}
			\end{equation}
		Using the Sobolev embeddings $H^2\subset L^\infty$ and $H^2\subset W^{1,4}$, we get
				\[
				\int_\Omega(1-G)\p_x\psi\theta' \leq -(1-C\|G\|_{H^2})\|\Delta\psi\|_{L^2}^2.\]
At this point we have
		\begin{equation} \label{eq:bb-energy}
			\frac12\frac\rd{\rd t}\|\theta'\|_{L^2}^2 \leq -(1-C\gamma_0)\|\Delta\psi\|_{L^2}^2 + \int_\Omega S\theta'= -(1-C\gamma_0)\|\Delta\psi\|_{L^2}^2 + \int_\Omega S_\para\theta'.
		\end{equation}
		So if $\gamma_0$ is small enough, in a universal way, the first term in the right-hand side  is non-positive. Therefore
		\begin{equation*}
			\frac\rd{\rd t}\|\theta'(t)\|_{L^2} \leq  \| S_\para(t)\|_{L^2} \leq  \frac\sigma{(1+t)^{1+\delta}}
		\end{equation*}
		and since $\delta>0$ this inequality integrates as
		\begin{equation*}
			\|\theta'\|_{L^2} \leq \|\theta'_0\|_{L^2} + C_\delta\sigma.
		\end{equation*}
	\end{proof}
	\begin{proof}[Proof of \cref{prop:bb-decay}]
		Back to \eqref{eq:bb-energy} and plugging the decomposition of $S$ we get
		\begin{equation*}
			\frac\rd{\rd t}\|\theta'\|_{L^2}^2 + (1-C\gamma_0) \|\Delta\psi\|_{L^2}^2 \leq \left(\|S_\para\|_{L^2} + \|S_\Delta\|_{L^2}\right)\|\theta'\|_{L^2}.
		\end{equation*}
        Then assumption \eqref{eq:bb2-S} and Young's inequality provide
        \begin{align*}
            \|S_\para\|_{L^2}\|\theta'\|_{L^2} &\leq \frac\sigma{(1+t)^{\frac12+\delta}}\frac{\|\theta'\|_{L^2}}{(1+t)^{\frac12}} \lesssim \frac{\sigma^2}{(1+t)^{1+2\delta}} + \frac{\|\theta'\|_{L^2}^2}{1+t}, \\
            \|S_\Delta\|_{L^2}\|\theta'\|_{L^2} &\lesssim \|\Delta\psi\|_{L^2}\frac{\|\theta'\|_{L^2}}{(1+t)^{\frac12}} \lesssim \gamma_0\|\Delta\psi\|_{L^2}^2 + \frac{1}{\gamma_0}\frac{\|\theta'\|_{L^2}^2}{1+t}.
        \end{align*}
        Hence if $\gamma_0$ is small enough, the dissipative term $\gamma_0\|\Delta\psi\|^2$ can be absorbed, and we have for some $c_0 \in (0,1)$,
		\begin{equation*}
			\frac\rd{\rd t}\|\theta'\|_{L^2}^2 + c_0\|\Delta\psi\|_{L^2}^2 \lesssim \frac{\sigma^2}{(1+t)^{1+2\delta}} + \frac{\|\theta'\|_{L^2}^2}{1+t}.
		\end{equation*}
		We now use the interpolation \cref{lem:interpol} with $\ell=0$, recalling that $\Delta^2\psi|_{\p\Omega}=\p_x\theta'|_{\p\Omega}=0$ and $\p_\bn\Delta^2\psi|_{\p\Omega}=\p_\bn\p_x\theta'|_{\p\Omega}=0$.  Choosing $K= \kappa/c_0(1+t)^{-1}$ with $\kappa>0$ arbitrary large independently of the data, we obtain
		\begin{equation*}
			\frac\rd{\rd t}\|\theta'\|_{L^2}^2 + \frac\kappa{1+t}\|\theta'\|_{L^2}^2 \lesssim \kappa^3\frac{\|\Delta^2\p_x^{-2}\theta'\|_{L^2}^2}{(1+t)^3} + \frac{\sigma^2}{(1+t)^{1+2\delta}}.
		\end{equation*}
		Plugging assumption \eqref{eq:bb2-A} provides 
		\begin{equation*}
			\frac\rd{\rd t}\|\theta'\|_{L^2}^2 + \frac\kappa{1+t}\|\theta'\|_{L^2}^2 \lesssim \frac{(\kappa^3 A^2+\sigma^2)}{(1+t)^{\min(3+2\alpha,1+2\delta)}}
		\end{equation*}
		which for a suitable choice of $\kappa$  integrates into
		\begin{equation*}
			\|\theta'(t)\|_{L^2} \lesssim \frac{\|\theta_0'\|_{L^2} + A+\sigma}{(1+t)^{\min(1+\alpha,\delta)}}\quad \forall t\in [0,T].
		\end{equation*}
	\end{proof}
	
	\begin{remark}[Stability for more general stationary profiles]\label{general-profiles}
Let us now explain how our results can be generalized to other stably stratified profiles $\rho_s$.
Let $\rho_s\in H^6(0,1)$ such that $\sup \p_z \rho_s<0$. 
The linear evolution equation on $\theta'$ can be written as
\begin{equation*}
	\p_t\theta' = -\p_z\rho_s\p_x\psi.
\end{equation*}
Multiplying the above equation by $-\theta'/\p_z \rho_s$, we obtain
\[
	\frac12\frac\rd{\rd t}\int_\Om \frac{1}{-\p_z \rho_s} \theta'(t,x)^2\ud x = \int_\Omega \p_x \psi \theta'=-\int_\Omega |\Delta \psi|^2.
\]
Since $\p_z \rho_s$ is bounded from above and below by negative constants, $-\int_\Om  (\theta')^2/\p_z \rho_s$ is equivalent to the $L^2$ norm, and the main linear estimate remains the same.
However, additional commutators stem from the nonlinear terms when we use the weight $-1/\p_z \rho_s$. For instance,
\begin{equation*}
\int_\Om \bu\cdot \nabla \theta' \ \frac{\theta'}{-\p_z \rho_s} = \frac{1}{2}\int_\Om (\theta')^2 \bu \cdot \nabla \frac{1}{\p_z\rho_s}.
\end{equation*}
These commutators will enter the terms $S_\parallel$ and $S_\bot$ from \cref{prop:bb-uniformbound} and \cref{prop:bb-decay}. Also, the constant $\gamma_0$ involved in the smallness condition \eqref{eq:bb1-G} on $G$ will now depend on $\p_z\rho_s$, but the result still holds. We leave the details to the reader, and stick to the case of a linearly stratified profile for simplicity. 

Let us now consider the case of non-stably stratified profiles. First, if $\p_z\rho_s>0$, we have the opposite sign in front of the dissipative term $\int_\Omega|\Delta\psi|^2$: at the linearized level, the perturbation grows. For the nonlinear equation, starting from a small perturbation, $\|\theta'(t)\|_{L^2}$ will have a transient growth for small times, until its norm becomes of order one and the nonlinear term can no longer be neglected. 
In fact, when $\p_z \rho_s$ is constant --- say $\p_z\rho_s=1$ --- the equation satisfied by $\theta$ is
\[
\p_t \theta + \bu\cdot \nabla \theta + u_2=0,
\]
and therefore
\[
\frac{1}{2}\frac{\rd}{\rd t} \|\theta(t)\|_{L^2}^2  =-\int_\Om \bu_2\rho = \int_\Om (\Delta \psi)^2\geq 0.
\]
In this case, the $L^2$ norm of $\theta$ is increasing on the whole interval $[0, +\infty)$, but remains  bounded since the $L^2$ norm of $\rho$ is conserved.
\\
Let us indicate a few facts about the long time behavior of solutions in the general case $\rho_0\in L^\infty$ (see Section 3.6 in \cite{leblondPhD} for a proof of these results). The velocity $\bu$ belongs to $L^2([0, +\infty), H^1_0(\Om))$. Furthermore the $\omega$-limit set in $H^{-1}$ of $\rho_0$ is non-empty, and contained in the set of stratified rearrangements of $\rho_0$. However it is not known in general whether this  $\omega$-limit set is a singleton.\\
We now go back to the case where $\rho_0$ is a small perturbation of a stratified state $\rho_s$.
When $\p_z\rho_s$ is not of constant sign, we cannot conclude \emph{a priori}, even at the linearized level. Indeed, if a density profile is not monotonous, then we cannot guarantee the proper sign in front of the integral $\partial_z\rho|\Delta\psi|^2$. Therefore, if a profile admits a non-monotonous function in any of its neighborhoods, in arbitrary high regularity, and in particular in $H^6$, the proof of our stability result does not hold. 
In particular, when $\p_z\rho_s\leq 0$ and $\p_z \rho_s$ vanishes at a single point $z_0\in (0,1)$, this point is also an inflection point and therefore $\rho_s(z)= \rho_s(z_0) + O((z-z_0)^3)$ in a neighborhood of $z_0$. Perturbing $\rho_s$ by a function of the type $\varepsilon (z-z_0)^2\chi(z-z_0)$ with  a cut-off function $\chi\in C^\infty_0(\R)$ breaks the monotony. Thus, even in the case when $\rho_s$ is monotonous, but has a vanishing derivative at single point, we cannot conclude.

	\end{remark}

	\subsection{Bootstrap argument}
	\label{sub:bootstrap}
	
	The purpose of this subsection is to prove, thanks to a bootstrap argument, that under the assumptions of \cref{thm:stab}, the solution $\theta'$ of \eqref{eq:STi} enjoys the same decay rates as the ones predicted by the linear analysis (see \cref{prop:bb-decay}).
	More precisely, we shall prove the following result:

	\begin{proposition} \label{prop:stab-precise}
		Let $\theta_0\in H^6(\Omega)$ such that $\theta_0\vert_{\p\Om}=\p_\bn\theta_0\vert_{\p\Om}=0$.
		There exists $\eps_0>0$ such that if $\|\theta_0\|_{H^6}\leq \eps\leq\eps_0$ the solution of \eqref{eq:STi} satisfies
		\begin{equation} \label{eq:stab-precise-bound}
			\|\p_x^3\theta'(t)\|_{L^2} \lesssim \frac\eps{1+t},\quad \|\p_x\theta'(t)\|_{H^4}\lesssim \eps, \quad \|\bar\theta(t)\|_{H^5} \lesssim \eps,\quad \forall t>0.
		\end{equation}
	\end{proposition}
	
	\begin{remark}
		The interplay between horizontal derivatives of $\theta$  and the considered regularities is consistent with the operator $\Delta^{-2}\p_x^2$ from the linearized system
		\begin{equation*}
			\p_t\theta' = \p_x\psi = \Delta^{-2}\p_x^2 \theta'.
		\end{equation*}
		Note that $\Delta^{-2}$ denotes the operator solving the bilaplacian $\Delta^2\psi=f$ equation endowed with the boundary condition $\psi|_{\p\Omega}=\p_\bn\psi|_{\p\Omega}=0$.
	\end{remark}
	A proof of \cref{prop:stab-precise} is provided in the rest of this section. Remarks motivating the necessity of the bootstrap hypothesis and the method in general are included throughout. 
	We also present our understanding of the obstacle to the iteration of this method to higher regularity on the perturbation.

	\paragraph{Bootstrap assumption and general argument.}
	Let $0<B<1$. For some $C_0> 1$, to be chosen later, let  $\theta_0 \in H^2_0\cap H^6$ such that $\|\theta_0\|_{H^6} \leq B/C_0$. In particular $\|\p_x^3\theta'_0\|_{L^2} \leq B/C_0$, $\|\p_x\theta_0'\|_{H^4}\leq B/C_0$, and $\|\p_z\bar\theta_0\|_{H^2} \leq B/C_0$. Let us note $\psi_0 :=\psi|_{t=0}$. We also have, according to \cref{lem:bilap}, with universal positive constants gathered under the same notation $C$,
	\begin{equation*}
		\|\psi_0\|_{H^4} \leq C\|\p_x\theta'_0\|_{L^2} \leq CB/C_0,
	\end{equation*}
	and therefore, 
	\begin{equation*}
		\begin{aligned}
			\|\p_t\p_x\theta'\vert_{t=0}\|_{L^2} &\leq \|1-\p_z\bar\theta_0\|_{L^\infty}\|\p_x^2\psi_0\|_{L^2} + \|\p_x(\nabla^\perp\psi_0\cdot\nabla\theta'_0)\|_{L^2} \\
                        &\leq (1+\|\p_z\bar\theta_0\|_{H^2})\|\p_x^2\psi_0\|_{L^2} + \|\p_x(\nabla^\perp\psi_0\cdot\nabla\theta'_0)\|_{L^2} \\
                        &\leq C(B/C_0 + (B/C_0)^2).
		\end{aligned}
	\end{equation*}
        Up to a choice of $C_0>1$ large enough, we find that all the bounds here above are strictly smaller than $B$. Therefore, by continuity of the Sobolev norms of $\theta$, ensured by \cref{thm:wp}, there exists a maximal time  $T^*\in\R_+\cup\{+\infty\}$ such that the following inequalities are satisfied on $[0,T^*)$:
	\begin{equation} \label{eq:boot-hyp}
		\begin{aligned}
			&\|\p_x^3\theta'(t)\|_{L^2}\leq \frac B{1+t}, & \|\p_x\theta'(t)\|_{H^4}\leq B,\\
			&\|\p_z\bar\theta(t)\|_{H^2}\leq B, &
			\|\p_t\p_x\theta'(t)\|_{L^2}\leq \frac B{(1+t)^2}.
		\end{aligned}
	\end{equation}
	We recall that these decay rates follow the behavior of the linearized system, see \cref{prop:bb-decay}.

	Let us assume by contradiction that $T^*<+\infty$. We show below by a bootstrap argument that hypothesis \eqref{eq:boot-hyp}, combined with \cref{prop:bb-uniformbound} and \cref{prop:bb-decay}, actually leads to an improvement of the inequalities, satisfied with some new constant $0<\underline{B} <B$ which contradicts  the maximality of $T^*$. Whence $T^*=+\infty$ and inequalities \eqref{eq:boot-hyp} hold for all times. 
	
	\begin{remark}
	Let us anticipate a little on the choice of the constant $B$. We will choose $B\leq \gamma_0$, so that assumption  \eqref{eq:bb1-G} is satisfied on $(0, T^*)$.
	\end{remark}

	\paragraph{Preliminary bounds.} \label{ssub:decays}
	
	Throughout the proof we require estimates on $\theta'$ and $\psi$ derived from the bootstrap hypothesis \eqref{eq:boot-hyp}. 
	For the sake of readability, we introduce the following short-hand notation: 
	\begin{equation*}
		\underbrace{\|f\|^\alpha \|g\|^\beta  }_{\alpha r+\beta r'}\qquad \text{when} \qquad \|f\| \lesssim \frac{B}{(1+t)^r} \text{ and }  \|g\| \lesssim \frac{B}{(1+t)^{r'}}.
	\end{equation*}

	First, from an integration by parts, since $\theta'\vert_{\p \Om}= \p_\bn \theta'\vert_{\p \Om}=0$ (see \cref{lem:traces})
	\begin{equation*} \label{eq:boot-intcrit}
		\|\p_x^2\theta'\|_{H^2}^2 \lesssim\int_\Omega \p_x^2\Delta\theta'\p_x^2\Delta\theta' = -\int_\Omega \p_x^3\theta'\p_x\Delta^2\theta' \leq \|\p_x^3\theta'\|_{L^2}\|\p_x\theta'\|_{H^4}.
	\end{equation*}
	We deduce, by assumption \eqref{eq:boot-hyp}, for all $t\in (0, T^*)$
	\begin{equation}\label{est:dx2thetaH2}
		\|\p_x^2\theta'(t)\|_{H^2} \lesssim \underbrace{\|\p_x^3\theta'(t)\|_{L^2}^{1/2}\|\p_x\theta'(t)\|_{H^4}^{1/2}}_{\frac{1}{2}\times 1 + \frac{1}{2}\times 0} \lesssim \frac B{(1+t)^{1/2}}.
	\end{equation}
	We also get by interpolation, for any $0\leq m\leq4$, for all $t\in (0, T^*)$
	\begin{align}\label{est:dx-theta-Hm}
		\|\p_x\theta'(t)\|_{H^m} \lesssim \underbrace{\|\p_x\theta'(t)\|_{L^2}^{1-m/4}\|\p_x\theta'(t)\|_{H^4}^{m/4} }_{(1-\frac{m}{4})\times 1 + \frac{m}{4}\times 0}\lesssim \frac{B}{(1+t)^{1-m/4}}.
	\end{align}
	We will frequently use \emph{Agmon's inequality} in dimension 2, namely
	\begin{equation*} \label{eq:agmon}
		\forall f \in H^1_0\cap H^2(\Omega), \qquad \|f\|_{L^\infty} \lesssim \|f\|_{L^2}^{1/2}\|f\|_{H^2}^{1/2},
	\end{equation*}
	together with the following direct consequence
	\begin{equation*} \label{eq:agmon2}
		\forall f\in  H^2_0\cap H^4(\Omega), \qquad \|\nabla f\|_{L^\infty} \lesssim \|f\|_{L^2}^{1/2}\|f\|_{H^4}^{1/2}.
	\end{equation*}
	We infer in particular, for all $t\in (0, T^*)$
	\begin{align}\label{est:dx2-theta-infty}
		\|\p_x^2\theta'(t)\|_{L^\infty} \lesssim \underbrace{\|\p_x^2\theta'(t)\|_{L^2}^{1/2}\|\p_x^2\theta'(t)\|_{H^2}^{1/2}}_{\frac{1}{2}\times 1 + \frac{1}{2}\times \frac{1}{2}} &\lesssim \frac{B}{(1+t)^{3/4}}, \\
	\notag \|\nabla\p_x\theta'(t)\|_{L^\infty} \lesssim \underbrace{\|\p_x\theta'(t)\|_{L^2}^{1/2}\|\p_x\theta'(t)\|_{H^4}^{1/2}}_{\frac{1}{2}\times 1 + \frac{1}{2}\times 0} &\lesssim \frac{B}{(1+t)^{1/2}}.
	\end{align}
	We also need estimates on $\psi$. Any Sobolev norm of order larger than 4 inherits the decays from $\theta'$ thanks to \cref{lem:bilap}, providing for $t\in (0, T^*)$
	\begin{align}\label{est:dx-psi-immediat}
		\|\p_x^2\psi(t)\|_{H^4} \lesssim \|\p_x^3\theta'(t)\|_{L^2} &\lesssim \frac{B}{1+t}, \\
		\|\p_x\psi(t)\|_{H^6} \lesssim \|\p_x^2\theta'(t)\|_{H^2} &\lesssim \frac{B}{(1+t)^{1/2}}. \notag
	\end{align}
	We also need  higher order decays on $\p_x\psi$ in $L^2(\Omega)$. We access this quantity thanks to the control of $\p_t\theta'$ by rewriting
	\begin{equation*}
		\p_x\psi=\frac{\p_t\theta'+(\bu\cdot\nabla\theta')'}{1-G}.
	\end{equation*}
	We know that $\|G\|_{L^\infty} \lesssim \|G\|_{H^2} \leq B$ so it is enough to have $B$ smaller than a universal constant to ensure that the inverse of $(1-G)$ is well-defined, which allows us to estimate $\p_x\psi$ and $\p_x^2\psi$ in $L^2$. We illustrate the computation for $\p_x\psi$ since the same reasoning applies for $\p_x^2\psi$ with a few extra terms.
	\begin{align*}
		\|\p_x\psi(t)\|_{L^2} &\lesssim \|\p_t\theta'(t)\|_{L^2} + \|\bu\cdot\nabla\theta'(t)\|_{L^2} \\
		&\lesssim \|\p_t\theta'(t)\|_{L^2} + \|\p_z\psi(t)\|_{L^\infty}\|\p_x\theta'(t)\|_{L^2}+ \|\p_x\psi(t)\|_{L^2}\|\p_z\theta'(t)\|_{L^\infty} \\
		&\lesssim \underbrace{\|\p_t\theta'(t)\|_{L^2}}_2 + \underbrace{\|\p_x\theta'(t)\|_{L^2}^2}_{2\times2} + \|\p_x\psi(t)\|_{L^2}\frac B{(1+t)^{1/2}}.
	\end{align*}
	Hence for $B>0$ small enough once again we get
	\begin{equation*}
		\|\p_x^2\psi(t)\|_{L^2} \lesssim \frac B{(1+t)^2}.
	\end{equation*}
	By interpolation and Agmon inequalities we deduce, in the same fashion as above, the following decay estimates, with the latest valid for $0\leq m\leq 4,$
	\begin{align}\label{est:dx-psi-OK}
		&\|\p_x^2\psi(t)\|_{H^2}\lesssim  \frac B{(1+t)^{3/2}}, \qquad \|\p_x^2\psi(t)\|_{L^\infty} \lesssim \frac{B}{(1+t)^{7/4}}, \qquad \|\nabla\p_x^2\psi(t)\|_{L^\infty}\lesssim \frac{B}{(1+t)^{3/2}}, \\
		&\|\nabla^2\psi(t)\|_{L^\infty} \lesssim \frac B{(1+t)^{5/4}},\quad
		\|\p_x^2\psi(t)\|_{H^m} \lesssim \|\p_x^2\psi(t)\|_{L^2}^{1-m/4}\|\p_x^2\psi(t)\|_{H^4}^{m/4} \lesssim \frac{B}{(1+t)^{2-m/4}}. \notag
	\end{align}
	\begin{remark}
		Let us come back briefly to the derivation of estimates \eqref{est:dx-psi-OK}, which ensure that $\psi$ decays \emph{faster} than $\theta'$. 
        Note that such a fast decay is necessary to close the estimates: indeed, $\|\bu(t)\|_{W^{1,\infty}}$ should be time integrable in order that $\theta(t)$ converges as $t\to \infty$.
  		Formally, one needs to take an anti-derivative in space of the equation, \ie apply the operator $\Delta^{-2} \p_x^2$ to \eqref{eq:STi}.
		However, because of the nonlinear term, this is a rather tedious operation.
		Therefore we rather derive estimates on $\p_t \theta' $, and use \eqref{eq:STi} in order to infer estimates on $\psi$. Note that the two operations (taking a time derivative or an anti-space derivative) are equivalent at main order, since the linear operator is $\p_t - \Delta^{-2} \p_x^2$.
		This idea, although simple, seems to us to be new.
	\end{remark}
	
	\paragraph{$H^6$ bound on the solution and $H^4$ bound on $G$.}
	
	In our nonlinear bootstrap argument, we shall need some high Sobolev bound on the solution. In order to lighten the proof of the bootstrap as much as possible, we isolate in the present paragraph this technical step.
	
	\begin{lemma}
		\label{lem:H6}
		Let $\theta_0\in H^6(\Omega)$ such that $\theta_0\vert_{\p\Om}=\p_\bn\theta_0\vert_{\p\Om}=0$.
		Let $T^*$ be the maximal time on which the assumptions \eqref{eq:boot-hyp} are satisfied.
		Then for all $t\in (0,T^*)$,
		\begin{align*}
			\|\theta'(t)\|_{H^6} &\lesssim B(1+t)^{1/2},&
			\| G(t)\|_{H^4} &\leq B/C_0 + C B^2,&
			\|\p_t G(t)\|_{L^\infty} &\lesssim \frac{B^2}{(1+t)^2}.
		\end{align*}
		
	\end{lemma}
	\begin{proof}
		We cannot estimate $\theta'$ in $H^6$ directly from its evolution equation since it requires an assumption on $ G=\p_z\bar\theta \in H^6$, and therefore on $\theta$ in $H^7$. To get around this, we directly perform an estimate from the whole perturbed evolution equation, namely
		\begin{equation*}
			\p_t\theta + \bu\cdot\nabla\theta = u_z .
		\end{equation*}
		For any derivative of order 6 (and less) of the previous equation, we obtain
		\begin{equation*} \label{eq:pls-est6}
			\frac12\frac\rd{\rd t}\|\p^6\theta\|_{L^2}^2 + \int_\Omega [\p^6,\bu\cdot\nabla]\theta \p^6\theta = \int_\Omega \p^6\p_x\psi \p^6\theta,
		\end{equation*}
		where the commutator comes from the incompressibility assumption and the no-slip boundary condition. Hence we get
		\begin{equation*}
			\frac\rd{\rd t}\|\theta\|_{H^6} \lesssim \|\p_x\psi\|_{H^6}+\|[\p^6,\bu\cdot\nabla]\theta\|_{L^2}.
		\end{equation*}
		The first term is dealt with thanks to the bilaplacian regularization (\cref{lem:bilap}) and estimate \eqref{est:dx2thetaH2},
		\begin{equation*}
			\|\p_x\psi\|_{H^6} \lesssim \|\p_x^2\theta'\|_{H^2} \lesssim \frac{B}{(1+t)^{1/2}}.
		\end{equation*}
		Notice that this 1/2-algebraic decay, issued from the linear system, is critical to prove the $1/2$-algebraic growth control of $\theta$ in $H^6(\Omega)$. 
		Concerning the nonlinear term,  we rely on the following \emph{tame estimate}, valid for any $m\in\N$,
		\begin{equation} \label{eq:tame}
			\forall f,g \in H^m\cap L^\infty(\Omega), \qquad \|fg\|_{H^m} \lesssim \|f\|_{L^\infty}\|g\|_{H^m} + \|f\|_{H^m}\|g\|_{L^\infty},
		\end{equation}
		which leads to
		\begin{equation} \label{eq:commu}
			\forall f \in H^m\cap W^{1,\infty}(\Omega), g\in H^{m-1}\cap L^\infty(\Omega), \quad \|[\partial^m,f]g\|_{L^2} \lesssim \|\nabla f\|_{L^\infty}\|g\|_{H^{m-1}} + \|f\|_{H^m}\|g\|_{L^\infty}.
		\end{equation}
		Hence, we decompose the nonlinear commutator into
		\begin{equation*}
			[\p^6,\bu\cdot\nabla]\theta = -[\p^6,\p_z\psi\p_x]\theta + [\p^6,\p_x\psi\p_z]\theta.
		\end{equation*}
		Each part can be estimated, thanks to \eqref{eq:commu}, as follows
		\begin{align*}
			\|[\p^6,\p_z\psi]\p_x\theta\|_{L^2} 
			&\lesssim \|\nabla\p_z\psi\|_{L^\infty}\|\p_x\theta\|_{H^5} + \|\p_z\psi\|_{H^6}\|\p_x\theta\|_{L^\infty} \\
			&\lesssim \underbrace{\|\nabla^2\psi\|_{L^\infty}}_{5/4}\|\theta\|_{H^6} + \underbrace{\|\p_x\theta'\|_{H^3}\|\p_x\theta'\|_{L^\infty}}_{1/4+3/4} \\
			&\lesssim \frac{{B}}{(1+t)^{5/4}}\|\theta\|_{H^6} + \frac{{B}^2}{1+t},
		\end{align*}
		and
		\begin{align*}
			\|[\p^6,\p_x\psi]\p_z\theta\|_{L^2} 
			&\lesssim \|\nabla\p_x\psi\|_{L^\infty}\|\p_z\theta\|_{H^5} + \|\p_x\psi\|_{H^6}\|\p_z\theta\|_{L^\infty} \\
			&\lesssim \underbrace{\|\nabla\p_x\psi\|_{L^\infty}}_{3/2}\|\theta\|_{H^6} + \underbrace{\|\p_x^2\theta'\|_{H^2}}_{1/2}\|\nabla\theta\|_{L^\infty} \\
			&\lesssim \frac{B}{(1+t)^{3/2}}\|\theta\|_{H^6} + \frac{B^2}{1+t} + \frac{B^2}{(1+t)^{1/2}}
		\end{align*}
		where we observed in particular that
		\begin{equation*}
			\|\nabla\theta\|_{L^\infty}\leq \|\nabla\theta'\|_{L^\infty} + \|G\|_{L^\infty} \lesssim \frac{{B}}{(1+t)^{1/2}} + B.
		\end{equation*}
		In the end, gathering and summing up all these bounds provides
		\begin{equation*}
			\frac\rd{\rd t}\|\theta\|_{H^6} \lesssim \frac{B}{(1+t)^{5/4}}\|\theta\|_{H^6} + \frac{B^2}{(1+t)^{1/2}},
		\end{equation*}
		and we get
		\begin{equation*}
			\|\theta\|_{H^6} \lesssim \|\theta_0\|_{H^6} +{B}(1+t)^{1/2} \lesssim {B^2}(1+t)^{1/2}.
		\end{equation*}

		Eventually, let us prove decaying bounds on $\p_t G$ and uniform bounds on $G$. 
		We recall that $G=\p_z\bar\theta$ depends only on the variables $t$ and $z$. From the evolution equation and one integration by parts we observe (omitting the factors $1/2\pi$ for clarity)
		\begin{equation*}
			\p_t\bar\theta = -\overline{\bu\cdot\nabla\theta'}=\int_\T \left(\p_z\psi\p_x\theta'-\p_x\psi\p_z\theta' \right)=  -\p_z\int_\T\p_x\psi\theta',
		\end{equation*}
		so we can write
		\begin{equation*}
			\p_t G= -\p_z^2\int_\T\p_x\psi\theta'.
		\end{equation*}
		
		The same arguments as above lead to
		\begin{equation*}
			\|\p_tG\|_{L^2(0,1)} \lesssim \underbrace{\|\p_x\psi\|_{L^\infty}\|\theta'\|_{H^2}}_{7/4+1/2} + \underbrace{\|\p_x\psi\|_{H^2}\|\theta'\|_{L^\infty}}_{3/2+3/4} \lesssim \frac{B^2}{(1+t)^{2+1/4}}.
		\end{equation*}
		Using the $H^6$ estimate, we also have
		\begin{equation*}
			\|\p_tG\|_{H^4(0,1)} \lesssim \underbrace{\|\p_x\psi\|_{L^\infty}\|\theta'\|_{H^6}}_{7/4-1/2} + \underbrace{\|\p_x\psi\|_{H^6}\|\theta'\|_{L^\infty}}_{1/2+3/4} \lesssim \frac{B^2}{(1+t)^{1+1/4}}.
		\end{equation*}
		Since the right-hand side of the above inequality is time-integrable, we infer that 
		\begin{equation*}
			\|G\|_{H^4(0,1)} \leq \|G(0)\|_{H^4(0,1)} + \int_0^t \|\p_tG\|_{H^4(0,1)} \leq B/C_0 + CB^2.
		\end{equation*}
		Moreover, for all $t\in (0, T^*)$,
		\begin{equation*} \label{eq:boot-dtGoo}
			\|\p_tG\|_{L^\infty}\lesssim \underbrace{\|\p_tG\|_{L^2}^{3/4}\|\p_tG\|_{H^4}^{1/4}}_{3/4\times9/4 +1/4\times5/4} \lesssim \frac{B^2}{(1+t)^2}.
		\end{equation*}
		
	\end{proof}
	
	\begin{remark}
	In the estimate of $\p^6\theta$,
		it would be tempting to proceed to the same computations as in \eqref{eq:bb-ipp} in order to exhibit a dissipative term, which would allow us to ignore its contribution as for lower order derivatives. Doing so requires to control the boundary integrals, which do not vanish \emph{a priori} in this case,
		\begin{equation*}
			\begin{aligned}
				\int_\Omega \partial^6\p_x\psi \partial^6\theta &= -\int_\Omega \partial^6\psi \Delta^2 \partial^6\psi \\
				&= -\|\Delta \partial^6\psi\|_{L^2}^2 + \int_{\p\Omega} (\p_\bn \partial^6\psi\Delta \partial^6\psi -\partial^6\psi \p_\bn \partial^6\Delta\psi).
			\end{aligned}
		\end{equation*}
		For instance, trying to bound the integral involving the higher order of $z-$derivatives on $\psi$ provides at best
			\begin{equation*}
			\left|\int_{\p\Omega}\partial^6\psi\p_z\Delta \partial^6\psi\right| \lesssim \|\p_x\psi\|_{H^{6 + \frac{1}{2} + \delta}}\|\p_x^{-1}\psi\|_{H^{9 + \frac{1}{2} + \delta}} \lesssim B^2(1+t)^{1/4+ \frac{\delta}{2}}.
		\end{equation*}
		This estimate ensures no better growth control than $\|\theta'\|_{H^6} \lesssim (1+t)^{3/4}$, which is not enough to close the bootstrap and get the control by $(1+t)^{1/2}$.  
	\end{remark}

	\paragraph{Improvements of the bootstrap bounds.}

	We now improve the uniform bound on $\theta'$ and $\p_x\theta'$ in $H^4(\Omega)$, relying on the linear analysis from \cref{sub:linearized}. Since $\|\theta'\|_{H^4} \leq \|\p_x\theta'\|_{H^4}$, it is enough to treat $\p_x\theta'$. Also we have according to \cref{lem:bilap} the inequality
	\begin{equation*}
		\|\p_x\theta'\|_{H^4} \lesssim \|\Delta^2\p_x\theta'\|_{L^2}
	\end{equation*}
	since $\p_x\theta'$ belongs in particular to $H^2_0\cap H^4(\Omega)$ as detailed in \cref{sub:traces}, so it is enough to deal with $\p_x\Delta^2\theta'$ in $L^2(\Omega)$. 
	
	\begin{lemma}[Uniform bound for $\|\p_x \theta'\|_{H^4}$] 
		
		\label{lem:H4}
		As long as the bootstrap hypothesis \eqref{eq:boot-hyp} holds we have
		\begin{equation} \label{eq:boot-H4}
			\|\p_x\Delta^2\theta'(t)\|_{L^2} \leq B/C_0 +CB^2.
		\end{equation}
	\end{lemma}
	\begin{proof}
		In view of the application of \cref{prop:bb-uniformbound} to $\Delta^2\p_x\theta'$, we observe that its evolution is governed by the equation
		\begin{equation*}
			\p_t\Delta^2\p_x\theta' = (1-G)\p_x\Delta^2\p_x\psi - [\Delta^2\p_x,G]\p_x\psi - \Delta^2\p_x(\bu\cdot\nabla\theta')',
		\end{equation*}
		which is of the form \eqref{eq:bb} with $\Delta^2\p_x\psi=\p_x^2\theta'$ and $\p_z\Delta^2\p_x\psi = \p_z\p_x^2\theta'$ vanishing on the boundary $\p\Omega$ and with
		\begin{equation*}
			S = \underbrace{-[\Delta^2\p_x,G]\p_x\psi - [\Delta^2\p_x,\bu\cdot\nabla]\theta'}_{ S_\para}  \underbrace{-\bu\cdot\nabla\Delta^2\p_x\theta'}_{ S_\perp}.    
		\end{equation*}
		We already know that $\|G\|_{H^2}$ satisfies the smallness assumption \eqref{eq:boot-hyp} for $B>0$ small enough. We show that $ S_\para$ presents an algebraic decay strictly larger than $1$, as in \eqref{eq:bb-S}. To do so we apply the tame estimate \eqref{eq:commu} to the two commutator terms.
		Let us emphasize that we need to be thorough by substituting $\bu = \nabla^\perp\psi$ such that the transport operator can be written as
		\begin{equation*}
			\bu\cdot\nabla=-\p_z\psi\p_x+\p_x\psi\p_z.
		\end{equation*}
		Hence the nonlinear term presents formally only a vertical derivative of order 1. This makes a difference in the estimates and allows us to reach more optimal decay rates.

		On the one hand we get for the perturbation due to $G$, using \eqref{est:dx-psi-OK}
		\begin{equation*}
			\|[\Delta^2,G]\p_x^2\psi\|_{L^2} \leq \|G\|_{H^4}\underbrace{\|\p_x^2\psi\|_{L^\infty}}_{7/4} + \|\nabla G\|_{L^\infty}\underbrace{\|\p_x^2\psi\|_{H^3}}_{5/4} \lesssim \frac{B^2}{(1+t)^{5/4}},
		\end{equation*}
	Note that we used here the uniform $H^4$ bound on $G$ from \cref{lem:H6}.
		On the other hand the contribution of $[\Delta^2\p_x,\bu\cdot\nabla]\theta'$ splits into four terms as follows
		\begin{align*}
			[\Delta^2\p_x,\bu\cdot\nabla]\theta' = &-\Delta^2(\p_x\p_z\psi\p_x\theta') + \Delta^2(\p_x^2\psi\p_z\theta') \\ &-[\Delta^2,\p_z\psi]\p_x^2\theta' + [\Delta^2,\p_x\psi]\p_z\p_x\theta'.
		\end{align*}
		We estimate each term accordingly. We have for instance, using the bootstrap assumption \eqref{eq:boot-hyp}, the preliminary bounds \eqref{est:dx2-theta-infty} and \eqref{est:dx-psi-OK}, and \cref{lem:H6},
		\begin{align*}
			\|\Delta^2(\p_x^2\psi\p_z\theta')\|_{L^2} &\lesssim \|\p_x^2\psi\|_{H^4}\|\p_z\theta'\|_{L^\infty} + \|\p_x^2\psi\|_{L^\infty}\|\p_z\theta'\|_{H^4} \\
			&\lesssim \underbrace{\|\p_x^3\theta'\|_{L^2}\|\nabla\theta'\|_{L^\infty}}_{1+1/2} + \underbrace{\|\p_x^2\psi\|_{L^\infty}\|\theta'\|_{H^5}}_{7/4-1/4} \lesssim \frac{B^2}{(1+t)^{3/2}}.
		\end{align*}
	The limiting decay comes from one of the commutators, which we estimate thanks to \eqref{eq:commu} together with the bounds \eqref{est:dx-psi-OK}, \eqref{est:dx-theta-Hm} and \eqref{est:dx2-theta-infty}
	\begin{equation*}
        \| [\Delta^2,\p_z\psi]\p_x^2\theta'\|_{L^2} \lesssim \underbrace{ \|\nabla\p_z \psi\|_{\infty} \|\p_x^2 \theta'\|_{H^3} }_{\frac{5}{4}+0} + \underbrace{\|\psi\|_{H^5} \|\p_x^2 \theta\|_{L^\infty}}_{\frac{3}{4} + \frac{3}{4}} \lesssim \frac{B^2}{(1+t)^{5/4}}.
    \end{equation*}	
        Gathering these estimates provides
		\begin{equation*}
			\| S_\para\|_{L^2} \lesssim \frac{B^2}{(1+t)^{5/4}},
		\end{equation*}
		and \cref{prop:bb-uniformbound} applies, ensuring
		\begin{equation*}
			\|\Delta^2\p_x\theta'\|_{L^2} \leq\|\Delta^2\p_x\theta'_0\|_{L^2} + CB^2.
		\end{equation*}
	\end{proof}
	
	\begin{lemma}[Decay of $\|\p_x^3 \theta'\|_{L^2}$]
				\label{lem:L2}
		As long as the bootstrap hypothesis \eqref{eq:boot-hyp} holds, we have
		\begin{equation*}
			\|\p_x^3\theta'\|_{L^2} \lesssim \frac{B/C_0 + B^2}{1+t}.
		\end{equation*}
	\end{lemma}
	\begin{proof}
		Note that $\p_x^3\theta'$ satisfies \eqref{eq:bb} with the source term
		\begin{equation*}
			S=S_\para=-\p_x^3(\bu\cdot\nabla\theta').
		\end{equation*}
		We can bound the whole term $S=S_\para$ as follows,
		\begin{equation*}\label{eq:boot-SL2}
			\begin{aligned}
				\| S\|_{L^2}
				&\leq \|\p_x^3(\p_z\psi\p_x\theta')\|_{L^2} + \|\p_x^3(\p_x\psi\p_z\theta')\|_{L^2}\\
				&\lesssim \|\p_x^3\p_z\psi\|_{L^2}\|\p_x\theta'\|_{L^\infty}+\|\p_z\psi\|_{L^\infty}\|\p_x^4\theta'\|_{L^2} \\
				&\quad+ \|\p_x^4\psi\|_{L^2}\|\p_z\theta'\|_{L^\infty}+\|\p_x\psi\|_{L^\infty}\|\p_x^3\p_z\theta'\|_{L^2} \\
				&\lesssim \underbrace{\|\p_x^2\psi\|_{H^2}\|\p_x\theta'\|_{L^\infty}}_{3/2+3/4>2} + \underbrace{\|\nabla\psi\|_{L^\infty}\|\p_x^2\theta'\|_{H^2}}_{3/2+1/2= 2}\\
				&\quad+\underbrace{\|\p_x^2\psi\|_{H^2}\|\nabla\theta'\|_{L^\infty}}_{3/2+1/2=2} +\underbrace{\|\p_x\psi\|_{L^\infty}\|\p_x^2\theta'\|_{H^2}}_{7/4+1/2>2} \lesssim \frac{B^2}{(1+t)^2}.
			\end{aligned}
		\end{equation*}
		Assumption \eqref{eq:bb2-S} is satisfied with $\delta=1$. Additionally, the norm of $(\Delta^2\p_x^{-2})\p_x^3\theta' = \Delta^2\p_x\theta'$ is bounded according to \eqref{eq:boot-H4}, so assumption \eqref{eq:bb2-A} is satisfied with $A= \|\Delta^2\p_x\theta'_0\|_{L^2} + CB^2$ and $\alpha=0$. Moreover, the traces of $\p_x^3\theta'$ and $\p_\bn\p_x^3\theta'$ vanish as a direct consequence of \cref{lem:traces}. Therefore $\min(1+\alpha,\delta)=1$ and \cref{prop:bb-decay} provides
		\begin{equation*}
			\|\p_x^3\theta'\|_{L^2} \lesssim \left(\| \p_x^3 \theta'_0\|_{L^2} + \|\Delta^2\p_x\theta'\|_{L^\infty((0,t),L^2)} + B^2\right)\frac{1}{1+t}.
		\end{equation*}
		Using inequality \eqref{eq:boot-H4}, we obtain the desired estimate.
	\end{proof}
	
	\begin{lemma}[Stronger decay of $\|\p_t \p_x \theta'(t)\|_{L^2}$] \label{lem:dtL2}
		Under assumptions \eqref{eq:boot-hyp} we have, for all $t\in (0, T^*)$,
		\begin{equation*}
			\|\p_t\p_x\theta'\|_{L^2} \lesssim \frac{B/C_0+ B^2}{(1+t)^2}.
		\end{equation*}
	\end{lemma}
	\begin{proof}
	The pair $(\p_t\p_x \theta',\p_t \p_x \psi)$ satisfies  \eqref{eq:bb}  with
		\begin{equation*}
		S = -\p_tG\p_x^2\psi - \p_t\p_x(\bu\cdot\nabla\theta').
		\end{equation*}
		Note that $\p_t\p_x\theta'$ and $\p_\bn\p_t\p_x\theta'$ vanish on the boundary, from \cref{lem:traces}. 
		In order to apply \cref{prop:bb-decay}, we have to bound $(\Delta^2\p_x^{-2})\p_t\p_x\theta'$ in $L^2(\Omega)$.
    Going back to \eqref{eq:STi}, we have
		\begin{equation*}
			\Delta^2\p_x^{-1}\p_t\theta' = \Delta^2((1-G)\psi) - \Delta^2\p_x^{-1}(\bu\cdot\nabla\theta')),
		\end{equation*}
		the norm of which can be estimated as
		\begin{equation*}
			\ba
			\|\Delta^2\p_x^{-1}\p_t\theta'\|_{L^2} &\leq (1+\|G\|_{H^4})\|\Delta^2\psi\|_{L^2} + \|\Delta^2\p_x^{-1}(\bu\cdot\nabla\theta')\|_{L^2}\\
			&\lesssim \|\p_x \theta'\|_{L^2} + \underbrace{\|\psi\|_{H^5}\|\nabla\theta'\|_{L^\infty}}_{3/4+1/2}+\underbrace{\|\nabla\psi\|_{L^\infty}\|\theta'\|_{H^5}}_{3/2-1/4} \\
			&\lesssim \frac{B/C_0 + B^2}{1+t} + \frac{B^2}{(1+t)^{5/4}}.
			\ea
		\end{equation*}
		Hence assumption \eqref{eq:bb2-A} is satisfied with $\alpha=1$ and $A=C(B/C_0 + B^2)$.
		
		Let us set $ S_\perp:=\bu\cdot\nabla\p_t\p_x\theta'$, indeed orthogonal to $\p_t\p_x\theta'$ in $L^2(\Omega)$.
We further define
\begin{align*}
	S_\para &:= -\p_tG\p_x^2\psi -\p_x\bu\cdot\nabla\p_t\theta',\\
	S_\Delta&:= \p_t\p_x\bu\cdot\nabla\theta'-\p_t\bu\cdot\nabla\p_x\theta',
\end{align*}		
so that $S_\perp + S_\para + S_\Delta=S$. Let us now check 	that $S_\para$ and $S_\Delta$ satisfy the assumptions of \cref{prop:bb-decay}.
		
		The first term in $S_\para$ can be bounded directly as follows, using \cref{lem:H6}
		\begin{align*}
			\|\p_tG\p_x^2\psi\|_{L^2} \leq \underbrace{\|\p_tG\|_{L^\infty}\|\p_x^2\psi\|_{L^2}}_{2+2} \lesssim \frac{B^2}{(1+t)^4}.
		\end{align*}
		The second requires for instance a bound on $\p_t\theta'$ in $H^1(\Omega)$, obtained directly from the evolution equation,
		\begin{align*}
			\|\p_t\theta'\|_{H^1} &\lesssim (1+\gamma_0)\|\p_x\psi\|_{H^1} + \|\bu\cdot\nabla\theta'\|_{H^1} \\
			&\lesssim (1+\gamma_0)\|\p_x\psi\|_{H^1} + \|\bu\|_{H^1}\|\nabla\theta'\|_{L^\infty}+\|\bu\|_{L^\infty}\|\nabla\theta'\|_{H^1} \\
			&\lesssim (1+\gamma_0)\underbrace{\|\p_x\psi\|_{H^1}}_{7/4} + \underbrace{\|\psi\|_{H^2}\|\nabla\theta'\|_{L^\infty}}_{3/2+1/2}+\underbrace{\|\nabla\psi\|_{L^\infty}\|\theta'\|_{H^2}}_{3/2+1/2} \lesssim \frac{B}{(1+t)^{7/4}}.
		\end{align*}
		Hence
		\begin{equation*}
			\|\p_x\bu\cdot\nabla\p_t\theta'\|_{L^2}\lesssim\|\p_x\bu\|_{L^\infty}\|\p_t\theta'\|_{H^1} \lesssim \underbrace{\|\nabla\p_x\psi\|_{L^\infty}\|\p_t\theta'\|_{H^1}}_{3/2+7/4} \lesssim\frac{B^2}{(1+t)^{13/4}},
		\end{equation*}
		and $S_\para$ satisfies the assumption \eqref{eq:bb2-S} with $\sigma=CB^2$ and $\delta=9/4$. Continuing our computations,
		\begin{align*}
			\|\p_t\p_x\bu\cdot\nabla\theta'\|_{L^2} \lesssim\|\nabla\p_t\p_x\psi\|_{L^2}\underbrace{\|\nabla\theta'\|_{L^\infty}}_{1/2} \lesssim \frac{B\|\Delta\p_t\p_x\psi\|_{L^2}}{(1+t)^{1/2}},
		\end{align*}
		and the same consideration applies for
		\begin{equation*}
			\|\p_t\bu\cdot\nabla\p_x\theta'\|_{L^2} \lesssim \|\nabla\p_t\psi\|_{L^2}\underbrace{\|\nabla\p_x\theta'\|_{L^\infty}}_{1/2} \lesssim \frac{B\|\Delta\p_t\p_x\psi\|_{L^2}}{(1+t)^{1/2}}.
		\end{equation*}
		Hence
		\begin{equation*}
			S_\Delta =\p_t\p_x\bu\cdot\nabla\theta'-\p_t\bu\cdot\nabla\p_x\theta'
		\end{equation*}
		indeed satisfies assumption \eqref{eq:bb2-S}. Finally  \cref{prop:bb-decay} applies with $\min(1+\alpha,\delta)=2$ and we obtain
		\begin{equation*}
			\|\p_t\p_x\theta'\|_{L^2}\lesssim \frac{B/C_0 +B^2}{(1+t)^2}.
		\end{equation*}
		
	\end{proof}

	\paragraph{Conclusion.}
	
	Let us close the bootstrap argument. Assuming $\|\theta_0\|_{H^6} \leq B/C_0$, we had, by continuity in time of the Sobolev norms of $\theta$ ensured by \cref{thm:wp}, existence of a maximal time $T^* \in\R_+\cup\{+\infty\}$ such that \eqref{eq:boot-hyp} is satisfied for any $t\in[0,T^*)$, reported here
	\begin{equation} \label{eq:recall}
		\begin{aligned}
			&\|\p_x^3\theta'\|_{L^2}\leq \frac B{1+t}, & \|\p_x\theta'\|_{H^4}\leq B,\\
			&\|G\|_{H^2}\leq B, &
			\|\p_t\p_x\theta'\|_{L^2}\leq \frac B{(1+t)^2}.
		\end{aligned}
	\end{equation}
	These decay estimates induce, as shown in \cref{lem:H6,lem:H4,lem:L2,lem:dtL2} that \eqref{eq:recall} holds for another constant $\underline{B}$ defined as
	\begin{equation*}
		\underline{B} = \frac{CB}{C_0} + CB^2,
	\end{equation*}
	where $C>0$ is universal. By choosing $C_0$ large enough and $B$ small enough, we have $\underline{B} < B$
	and inequalities \eqref{eq:recall} are strictly satisfied for any $t\in[0,T^*)$. Therefore $T^*$ must be $+\infty$, otherwise the continuity of $t\mapsto \|\theta(t)\|_{H^6}$ would imply the existence of a larger validity time interval for \eqref{eq:recall}. In the end, these bounds are valid for all times, and setting $\eps_0 := B/C_0$ closes the demonstration of \cref{prop:stab-precise}.
	
		\begin{remark}[Generalisation at any order]\label{rmk:higher-reg}
		Motivated by the fact that the following perturbated subproblem is stable under horizontal derivation
		\begin{equation*}
			\left\{
			\begin{aligned}
				& \p_t\p_x^\ell\theta' = (1-G)\p_x^{\ell+1}\psi \\
				& \Delta^2\p_x^{\ell-1}\psi = \p_x^\ell\theta',\\
				& \p_x^{\ell-1} \psi|_{\p\Omega} = \p_\bn\p_x^{\ell-1}\psi|_{\p\Omega} = 0,
			\end{aligned}
			\right.
		\end{equation*}
		we could expect to propagate arbitrarily high horizontal regularity on $\theta'$. Nevertheless, our proof relies on the control
		\begin{equation*}
			\|\theta\|_{H^6} \lesssim (1+t)^{1/2},
		\end{equation*}
		which we can obtain thanks to the classical divergence free condition on $\bu$ canceling the extra-derivative term. Let us try to do the same on $\p_x^\ell\theta$; we write
		\begin{equation*}
			\p_t\p_x^\ell\theta + \sum_{k=0}^{\ell-1} C_{\ell,k} \p_x^{\ell-k}\bu\cdot\nabla\p_x^k\theta + \bu\cdot\nabla\p_x^\ell\theta= \p_x^{\ell+1}\psi,
		\end{equation*}
		and multiply by $\p_x^\ell \theta'$.
		Then the  estimation does not close, even though one of its terms does not contribute, just as in the initial equation. Indeed, we have
		\begin{equation*}
			\frac12\frac\rd{\rd t}\|\p_x^\ell\theta\|_{H^6}^2 + \sum_{k=0}^{\ell-1}C_{\ell,k}\int_\Omega \partial^6(\p_x^{\ell-k}\bu\cdot\nabla\p_x^k\theta)\partial^6\p_x^\ell\theta + \frac12\underbrace{\int_\Omega \bu \cdot\nabla|\partial^6\p_x^\ell\theta|^2}_{=0} \leq \|\p_x^\ell\psi\|_{H^6}\|\p_x^\ell\theta\|_{H^6}.
		\end{equation*}
		Note that crossed derivatives integrands, such as
		\begin{equation*}
			\int_\Omega \p_x^\ell\bu\cdot\nabla \partial^6\theta \partial^6\p_x^\ell\theta
		\end{equation*}
		 do not lead to a vanishing integral. Hence deriving an estimate on $\p_x^\ell \theta'$ in $H^6$ requires first the derivation of an estimate of $\theta$ in $H^{6+\ell}$, in the spirit of \cref{lem:H6}.
	We will derive such estimates in \cref{sec:nonlinear-BL} (see \cref{lem-high-reg-NLBL}), at the price of much stronger and more complicated bootstrap assumptions (see \eqref{hyp:bootstrap-theta'-NLBL}). 
	 
	\end{remark}
	
	\subsection{Convergence as \texorpdfstring{$t\to \infty$}{t → ∞} and identification of the asymptotic profile}
	\label{sub:finalstate}
	
	Regarding the asymptotic behavior of the density for the Stokes-transport system without any assumption on the type of initial data, we can only say that if $\rho$ converges toward some $\rho_\infty$ in $H^{-1}$, this limit profile is stratified. Indeed, the energy balance \eqref{eq:ep} ensures that $\bu\in L^2(\R_+,H^1)$, and since $\bu$ is also $\Lip(\R_+,H^1)$ by linearity of the Stokes system, we infer that $\|\bu(t)\|_{H^1} \to 0$ as $t\to\infty$, but without any information about its decay rate. At least we have
	\begin{equation*}
		\|\nabla p + \rho\be_z\|_{H^{-1}} \lesssim \|\bu\|_{H^1} \underset{t\to\infty}{\longrightarrow} 0.
	\end{equation*}
	The $H^{-1}$ convergence of $\rho$ leads to the existence of $p_\infty$ such that
	\begin{equation*}
		\nabla p_\infty = -\rho_\infty\be_z.
	\end{equation*}
	Observing that $\p_xp_\infty=0$ and that the domain $\Omega=\T\times(0,1)$ is convex ensures that $p_\infty$ and $\rho_\infty$ are both independent of the horizontal coordinate $x$.
	
	In the context of a small perturbation of the stationary profile $\rho_s(z)=1-z$ we obtained explicit decay rates for Sobolev norms of $\bu$. We show that these decays are sufficient to ensure the strong convergence of $\rho$ toward a limit profile $\rho_\infty$. Moreover, the smallness of the perturbation $\theta$ does not affect the vertical monotonicity of the whole density $\rho$, from which we deduce that $\rho_\infty$ is exactly the vertical rearrangement of $\rho_0$.
	
	\begin{proposition} \label{prop:acv}
		Under the assumptions of \cref{thm:stab}, the whole density $\rho$ converges in $H^m$ for any $m<4$ towards its vertical decreasing rearrangement.
	\end{proposition}
	
	\begin{proof}%
		The proof is divided in the following steps:
		\paragraph{Convergence.}
		It is enough to show that $\p_t\rho$ belongs to $L^1(\R_+,H^m)$ for $m<4$, which implies the strong convergence of $\rho(t)$ in $H^m$ and existence of a limit $\rho_\infty$. Let us estimate $\p_t\rho$ in $H^m$ for any $0\leq m \leq 4$, using the tame estimates \eqref{eq:tame}
		\begin{align*}
			\|\p_t\rho\|_{H^m} & = \|\bu\cdot\nabla\rho\|_{H^m} \leq \|\p_z\psi\p_x\rho\|_{H^m} + \|\p_x\psi\p_z\rho\|_{H^m} \\
			& \lesssim \|\nabla\psi\|_{L^\infty}\|\p_x\rho\|_{H^m} + \|\psi\|_{H^{m+1}}\|\p_x\rho\|_{L^\infty} + \|\p_x\psi\|_{L^\infty}\|\rho\|_{H^{m+1}}+\|\p_x\psi\|_{H^m}\|\p_z\rho\|_{L^\infty}.
		\end{align*}
		Recalling that
		$\p_z\rho = -1+G+\p_z\theta'$ is bounded in $H^5(\Omega)$, that $\p_x\rho = \p_x\theta'$ decays as $(1+t)^{-1+m/4}$ for $m\leq 4$, as well as the  decay estimates \eqref{est:dx-psi-immediat} and \eqref{est:dx-psi-OK}, we find
		\begin{equation*}
			\|\p_t\rho\|_{H^m} \lesssim \frac{\|\rho_0\|_{H^6}^2}{(1+t)^{2-m/4}},
		\end{equation*}
		which is integrable for any $m<4$, hence the convergence.
		
		\paragraph{Stratified limit.}
		Since $\rho$ converges, so do $\theta' = (\rho -\rho_s)'$ and $\bar\theta = \overline{\rho-\rho_s}$. We obtained in \eqref{eq:stab-precise-bound} that $\theta'$ vanishes in $H^m$ for $m<4$, and therefore the limit $\rho_\infty$ is stratified. 
		Hence $\rho_\infty$ can be written as the sum of $\rho_s$ and the limit $\bar\theta_\infty$ of $\bar\theta$. In view of \eqref{eq:stab-precise-bound} this limit satisfies in particular $\|\p_z\bar\theta_\infty\|_{L^\infty} \leq C\eps_0$, with the notation of \cref{thm:stab}. At least for $\eps_0>0$ such that $C\eps_0 < 1 = -\p_z\rho_s$, we know that $\sup_{(0,1)}\p_z\rho_\infty < 0$, which means that $\rho_\infty$ is strictly decreasing with respect to $z$.
		
		\paragraph{Rearrangement.}
		The divergence free character of the velocity field $\bu$ ensures that all $L^q$ norms and the cumulative distribution function of $\rho(t)$ are preserved along time, in the sense
		\begin{equation} \label{eq:cfeq}
			\forall \lambda \geq 0, \qquad |\{\rho(t)>\lambda\}| = |\{\rho_0>\lambda\}|.
		\end{equation}
		This property transfers to the limit state $\rho_\infty$ by $L^q$ strong convergence of $\rho$. According to rearrangement theory such as that developed in \cite[Chapter 3]{liebloss}, we say that two maps are rearrangements of each other if they have the same level sets, in the sense of \eqref{eq:cfeq}. Adapting slightly the construction of \cite{liebloss}, we know there exists a unique vertical decreasing rearrangement of $\rho_0:\Omega\to\R_+$, which can be defined as
		\begin{equation*}
			\rho_0^*(z) :=\int_0^\infty \boldsymbol{1}_{0\leq z\leq|\{\rho_0>\lambda\}|}\ud\lambda.
		\end{equation*}
		In the end, we know that $\rho_\infty$ is a decreasing rearrangement of $\rho_0$, therefore it is $\rho_0^*$ by uniqueness.
	\end{proof}
	
	\begin{remark}
Note that the above argument extends immediately to the settings investigated by Elgindi \cite{elgindi} and Castro, C\'{o}rdoba and Lear \cite{castro2019global} for the incompressible porous media problem, as mentioned in the Introduction.
		\end{remark}

	Notice that we actually have $\|\p_z\theta\|_{L^\infty} \lesssim \eps_0$ for all times. Therefore the total density has a strictly negative vertical derivative, for all $x \in \T$ and for all times $t\in\R_+$, since
	\begin{equation*}
		\p_z\rho(t,x,\cdot) = -1 + \p_z\theta(t,x,\cdot),
	\end{equation*}
	and the density reordering is essentially horizontal. This is a rare case in which we can describe the asymptotic profile. This intuition of having heavy fluids sinking under the lighter ones prompts to wonder if, at least in a weak sense, the density profile should always converge toward the vertical rearrangement of the initial datum, unless it is already stratified. This question remains open, both for the Stokes-transport equation and for the incompressible porous media.
	
	\section{Formation of large time boundary layers in the linear setting\texorpdfstring{: proof of \cref{thm:BL-linear}}{}}
	\label{sec:linear-BL}
	
	The purpose of this section is to prove \cref{thm:BL-linear}. We consider the linear problem
	\beq
	\label{linear-sans_G}
	\begin{cases}
		\p_t \theta=\p_x \psi\quad \text{in } (0,+\infty)\times \Om,\\
		\Delta^2 \psi=\p_x \theta \quad \text{in } \Om,\quad 
		\psi\vert_{\p\Om}=\p_\bn \psi\vert_{\p\Om}=0,\\
		\theta(t=0)=\theta_0,
	\end{cases}
	\eeq
	with $\theta_0\in H^6(\Om)$ arbitrary. 
	The difference with the linear analysis of \cref{sub:linearized}, and in particular with \cref{prop:bb-decay}, lies in the fact that we do not assume that $\theta_0$ and $\p_\bn \theta_0$ vanish on the boundary.
	As a consequence, as explained in the scheme of the proof in the Introduction, boundary layers are created as $t\to \infty$ close to $z=0$ and $z=1$, and the purpose of this section is precisely to describe the mechanism driving the apparition of these boundary layers.
	We will therefore decompose $\theta$ as the sum of an interior term decaying like $t^{-1}$ in $L^2$, and some boundary layer terms which lift the traces of $\theta$ and $\p_\bn\theta$ on the boundary. This will lead us to \cref{thm:BL-linear}.
	We will then return to our nonlinear system \eqref{eq:STi} in \cref{sec:nonlinear-BL}.

In fact, we will prove a more precise version of \cref{thm:BL-linear}:
\begin{proposition}
Let $\theta_0\in H^s(\Om)$ for some $s>0$ sufficiently large.
Let $\theta\in C(\R_+, H^s)$ be the unique solution of \eqref{linear-sans_G}.
There exists a boundary layer profile $\thbl$, of the form
\[
\thbl=\sum_{j=0}^4 (1+t)^{-j/4} \left(\Thbot^j(x,Z_\mbot) + \Thtop^j(x,Z_\mtop)\right),
\]
with $Z_\mbot=z(1+t)^{1/4}$ and $Z_\mtop= (1-z)(1+t)^{1/4}$, such that $\thint=\theta-\thbl$ satisfies, for all $t\geq 0$,
\[
 \|\p_x^2 \thint(t)\|_{L^2} \lesssim \frac{\|\theta_0\|_{H^s}}{1+t},\quad
 \|\thint(t)\|_{H^4} \lesssim {\|\theta_0\|_{H^s}},\quad
 \| \Delta^{-2} \p_x^2 \thint(t)\|_{L^2} \lesssim \frac{\|\theta_0\|_{H^s}}{(1+t)^2}.
\]
Furthermore, there exists a constant $c>0$ such that $\|\Theta^j_a(\cdot, Z)\|_{H^4(\T)}\lesssim \|\theta_0\|_{H^s} \exp(-c Z^{4/5})$ for all $Z>0$.

\end{proposition}

	The organization of this section is the following. 
After motivating the Ansatz \eqref{BL-Ansatz-linear}, we formally derive the equation satisfied by the boundary layer profiles.
 We then construct the boundary layer part of the solution, denoted by $\thbl$, and we establish some properties.
	Eventually, we prove that $\theta-\thbl$ satisfies the assumptions of \cref{prop:bb-decay}, and we conclude.

	\subsection{Motivation for the Ansatz and derivation of the boundary layer equations}
We recall (see page \pageref{scheme:linearBL}) that a simple spectral decomposition suggests that the solution $\theta$ has strong variations close to the boundaries, and that $\theta(t)\vert_{\p\Om}=\theta_0\vert_{\p\Om}$, $\p_\bn\theta(t)\vert_{\p\Om}=\p_\bn\theta_0\vert_{\p\Om}$.
	Hence we take an Ansatz of the form
	\[
	\ba
	\theta(t)\simeq &\ \thint + \Thbot^0\left(x, (1+t)^\alpha z\right) + \Thtop^0(x, (1+t)^\alpha(1-z)) \\&+ (1+t)^{-\alpha } \Thbot^1 \left(x, (1+t)^\alpha z\right) + (1+t)^{-\alpha } \Thtop^1(x,(1+t)^{\alpha }(1-z))+\text{l.o.t.}\ea
	\]
	for some $\alpha>0$ to be determined, where
	\begin{gather*}
		\thint\vert_{\p\Om} = \p_z \thint\vert_{\p\Om}=0,\\
		\Thbot^j(x,Z)\to 0 \quad \text{and} \quad \Thtop^j (x,Z) \to 0 \quad \text{as} \quad Z\to \infty.
	\end{gather*}
	The role of $\Thtop^0$ (resp. of $\Thbot^0$) is to lift the trace of $\theta_0$ at the top boundary $z=1$ (resp. at the bottom boundary $z=0$).
	Hence we take
	\begin{align*}
		\Thtop^0(x,Z=0) &=\theta_0(x,z=1),& \p_Z \Thtop^0(x,Z=0) &=0,\\
		\Thbot^0(x,Z=0) &=\theta_0(x,z=0),& \p_Z \Thbot^0(x,Z=0) &=0.
	\end{align*}
	In a similar way, the next order boundary layer terms $\Thtop^1$ and $\Thbot^1$ lift the traces of $\p_\bn \theta_0$ on $\p\Om$, \ie
	\begin{align*}
		\Thtop^1(x,Z=0) &= 0,& \p_Z \Thtop^{{1}}(x,Z=0)&=-\p_z \theta_0(x,z=1),\\
		\Thbot^{{1}}(x,Z=0) &= 0,& \p_Z \Thbot^{{1}}(x,Z=0)&=\p_z \theta_0(x,z=0).
	\end{align*}
	Similarly, we assume that
	\[\ba
	\psi(t)\simeq &\ \psiint + (1+t)^{-4\alpha} \Psbot^0\left(x, (1+t)^\alpha z\right) +  (1+t)^{-4\alpha} \Pstop^0(x, (1+t)^\alpha(1-z))\\& +  (1+t)^{-5\alpha}\Psbot^1 \left(x, (1+t)^\alpha z\right) + (1+t)^{-5\alpha} \Psbot^1(x,(1+t)^{\alpha }(1-z)) + \text{l.o.t.}\ea
	\]
	where 
	\[\ba
	\p_{Z}^4 \Psi_a^j=\p_x \Theta_a^j, \\
 \Psi_a^j=\p_Z\Psi_a^j=0\text{ on }Z=0, \\
 \Psi_a^j\to 0 \text{ as }Z\to \infty,\quad a\in\{\mtop,\mbot\}.
 \ea
	\]
	Plugging these Ansatz into \eqref{linear-sans_G}, we find that at main order
	\[
	\alpha (1+t)^{-1} Z\p_Z \Theta_a^0=(1+t)^{-4\alpha}\p_x \Psi_a^0.
	\]
	Consequently, identifying the powers of $(1+t)$, we take $\alpha=1/4$, which is precisely the Ansatz \eqref{BL-Ansatz-linear}. Hence the equation for $\Psi_a^0$, $a\in\{\mtop,\mbot\}$ becomes
	\beq\label{eq:Psi_i}
	\begin{cases}
		\frac{1}{4} Z \p_Z^5 \Psi^0_a= \p_x^2 \Psi^0_a \quad \text{in }\T\times (0,+\infty),\\
		\Psi^0_{a|Z=0}=\p_Z \Psi^0_{a|Z=0}=0,\\
		\p_Z^4 \Psi^0_{a|Z=0}= \gamma^0_a(x),\quad \p_Z^5 \Psi^0_{a|Z=0}=0,\\
		
		\lim_{Z\to \infty} \Psi^0_a(x,Z)=0, 
	\end{cases}
	\eeq
	where $\gamma^0_{\mbot}(x)=\p_x  \theta_0(x,z=0)$, $\gamma^0_{\mtop}(x)=\p_x  \theta_0(x,z=1)$.
	Note that the above boundary conditions are redundant: indeed, if $\p_Z \Psi^0_{a|Z=0}=0$, then it follows from the equation (after one differentiation with respect to $Z$) that $\p_Z^5 \Psi^0_{a|Z=0}=0.$ Hence in the following subsection we will drop the condition $\p_Z^5 \Psi^0_{a|Z=0}=0.$
	
	In a similar fashion, the equation for  $\Psi_a^1$, $a\in\{\mtop,\mbot\}$ is
	\beq\label{eq:Phi_i}
	\begin{cases} 
		-\p_Z^4 \Psi^1_a+ Z \p_Z^5 \Psi^1_a = 4 \p_x^2 \Psi^1_a\quad \text{in }\T\times (0,+\infty),\\
		\Psi^1_{a|Z=0}=\p_Z \Psi^1_{a|Z=0}=0,\\
		\p_Z^4 \Psi^1_{a|Z=0}= 0,\quad \p_Z^5 \Psi^1_{a|Z=0}=\gamma^1_a(x),\\
		\lim_{Z\to \infty} \Psi^1_a(x,Z)=0,
	\end{cases}
	\eeq
	where $\gamma^1_{\mbot}(x)=\p_x \p_z \theta_0(x,z=0)$, $\gamma^1_{\mtop}(x)=-\p_x  \p_z\theta_0(x,z=1)$. Once again, the condition $\p_Z^4 \Psi^1_{a|Z=0}= 0$ is redundant and is automatically satisfied when one takes the trace of the equation at $Z=0$, using the other boundary conditions.
	We now turn towards the well-posedness of \eqref{eq:Psi_i} and \eqref{eq:Phi_i}.

	\subsection{Construction of the main profiles}
	The well-posedness of equations \eqref{eq:Psi_i} and \eqref{eq:Phi_i} stems from the following result:
	\begin{lemma}
		\label{lem:WP-ODE-BL}
		Let $m\geq m_0>0$ and let $S\in C([0,+\infty))$, $\delta>0$ such that 
		\[
		\|S\|^2:=\int_0^1 \frac{S(Z)^2}{Z^2}\ud Z + \int_0^\infty S(Z)^2 \exp(\delta  Z^{4/5}) \ud Z <+\infty.
		\]
		Consider the ODE
		\beq\label{ODE}
		Z \p_Z^5 \Psi(Z) =-m \Psi(Z) + S(Z)\quad \text{in }\quad (0,+\infty),\quad
		\lim_{Z\to \infty} \Psi(Z)=0,   
		\eeq
		endowed with one of the following four boundary conditions:
		\begin{enumerate}[label=(\roman*)]
			\item $\Psi(0)=\p_Z\Psi(0)=\p_Z^4 \Psi(0)=0$;
			\item $\Psi(0)=\p_Z^3 \Psi(0)=\p_Z^4 \Psi(0)=0$;
			\item $\Psi(0)=\p_Z^2\Psi(0)= \p_Z^3 \Psi(0)=0$;
			\item $\Psi(0)=\p_Z\Psi(0)=\p_Z^2\Psi(0)= 0$.
		\end{enumerate}
		Then there exists a constant $c>0$ depending only on $m_0$ and $\delta$ such that
		\eqref{ODE} endowed with one of the four previous conditions has a unique solution $\Psi\in H^5_{\text{loc}}(\R_+)$ such that for all $k\in \{0,\cdots,5\}$,
		\[
		\int_0^\infty  |\p_Z^k \Psi(Z)|^2 \exp(c  Z^{4/5})\: \ud Z\leq C\|S\|^2<+\infty .
		\]
		As a consequence, for $k\leq 4$, there exists a constant $C$ such that
		\[
		\left|\p_Z^k \Psi(Z)\right| \leq C \|S\| \exp\left(-\frac{c}{4}Z^{4/5}\right)\quad \forall Z>0.
		\]
	\end{lemma}
	The proof of \cref{lem:WP-ODE-BL} is postponed to \cref{ap:lemma32}, and relies on the use of the Lax-Milgram Lemma in weighted Sobolev spaces.
	As a corollary, we have the following result:
	\begin{corollary}
		For all $j\in \{0, 1\}$, there exists a unique solution $\chi_j\in C^\infty(0, +\infty)$ of the ODE
		\[
		Z \p_Z^5\chi_j - j \p_Z^4\chi_j+ 4 \chi_j=0\quad \text{on }(0, +\infty),
		\]
		endowed with the boundary conditions:
		\begin{itemize}
			\item $\chi_0(0)=\p_Z\chi_0(0)=0$, $\p_Z^4\chi_0(0)=1$;
			\item $\p_Z\chi_1(0)=\p_Z^4\chi_1(0)=0$, $\p_Z^5\chi_1(0)=1$,
		\end{itemize}
		and such that for $j=0,1$, $0\leq k\leq 5$,
		\[
		\int_0^\infty |\p_Z^K\chi_j (Z)| \exp(c  Z^{4/5})\: \ud Z<+\infty.
		\]
		Furthermore, $\p_Z^5\chi_0(0)=\chi_1(0)=0$.
		\label{def:chi_j}
	\end{corollary}
	\begin{proof}
		Let us start with $\chi_0$. Let $\eta\in C^\infty_c(\R)$ such that $\eta\equiv 1$ in a neighborhood of zero. Then $\chi_0-Z^4\eta/4!$ satisfies \eqref{ODE} with the boundary conditions (i) and with a $C^\infty$ and compactly supported source term. 
		Hence the result follows from \cref{lem:WP-ODE-BL}.
		The $C^\infty$ regularity of $\chi_0$ follows easily from the ODE \eqref{ODE} and from an induction argument.
		Differentiating the ODE and taking the trace at $Z=0$, we obtain $\p_Z^5\chi_0(0)=- 4 \p_Z\chi_0(0)=0$.

		Concerning $\chi_1$, we first consider the solution of the ODE
		\[
		\ba 
		Z \p_Z^5 \phi + 4 \phi=0\quad \text{on } (0, +\infty),\\
		\phi(0)=\p_Z^3\phi(0)=0, \ \p_Z^4 \phi(0)=1,\ \phi(+\infty)=0.
		\ea
		\]
		The existence, uniqueness, and exponential decay of $\phi$ follow from a lifting argument and \cref{lem:WP-ODE-BL} with boundary conditions (ii).
		We then set $\chi_1(Z)=-\int_Z^\infty \phi$, and we observe that $\p_Z(Z \p_Z^5\chi_1 - \p_Z^4\chi_1 + 4 \chi_1)=0$. As a consequence, $Z \p_Z^5\chi_1(Z) - \p_Z^4\chi_1(Z) + 4 \chi_1(Z)=\text{const.}=0$ on $(0, + \infty)$, thanks to the decay properties of $\phi$ at infinity. Hence the existence, uniqueness and decay of $\chi_1$ follow. Taking the trace of the equation at $Z=0$, we find that $\chi_1(0)=0$.

	\end{proof}

	Let us now explain how we construct the boundary layer profiles $\Psi_a^j$ for $a\in\{\mtop,\mbot\}$  and $j=0,1$, which satisfy \eqref{eq:Psi_i} and \eqref{eq:Phi_i}. Taking the Fourier transform of \eqref{eq:Psi_i} with respect to $x$ and dropping the index $a$, we infer that $\widehat{\Psi^0_{k}}$ satisfies 
	\[\ba
	\frac{1}{4}Z \p_Z^5 \widehat{\Psi^0_{k}}= - k^2 \widehat{\Psi^0_{k}},\\
	\p_Z^4 \widehat{\Psi^0_{k|Z=0}}= \widehat{\gamma^0_{k}},\quad \widehat{\Psi^0_{k|Z=0}}=\p_Z \widehat{\Psi^0_{k|Z=0}}=0.
	\ea
	\]
	Considering the function $\chi_0$ defined in \cref{def:chi_j}, it is then easily checked that 
	\[
	\widehat{\Psi^0_{k}}= |k|^{-2}\widehat{\gamma^0_k}\; \chi_0(|k|^{1/2} Z)
	\]
	is a solution of the problem.
	We infer that
	\beq\label{formula-Fourier-PsiO}
	\Psi^0_a(x,Z):=\sum_{k\in \Z\setminus\{0\}}  |k|^{-2}\widehat{\gamma^0_{a,k}}\; \chi_0(|k|^{1/2} Z)e^{ikx}
	\eeq
	is a solution of \eqref{eq:Psi_i}.
	In a similar way, 
	\beq \label{formula-Fourier-Psi1}
	\Psi^1_a(x,Z):=\sum_{k\in \Z\setminus\{0\}}  |k|^{-5/2}\widehat{\gamma^1_{a,k}}\; \chi_1(|k|^{1/2} Z)e^{ikx}
	\eeq
	satisfies \eqref{eq:Phi_i}.

	As a consequence, we have the following estimates, which follow easily from  formulas \eqref{formula-Fourier-PsiO} and \eqref{formula-Fourier-Psi1}: 
	\begin{corollary}
		Let $\gamma^0_a,\gamma^1_a \in L^2(\T)$.
    Then  \eqref{eq:Psi_i} (resp. \eqref{eq:Phi_i}) has a unique solution $\Psi^0_a\in H^{9/4}_xL^2_Z\cap L^2_x H^{9/2}_Z $ (resp. $\Psi^1_a\in H^{11/4}_xL^2_Z\cap L^2_x H^{11/2}_Z$). 
		Furthermore,  for all $m>9/2$,
		\begin{equation*}
		\ba
		\| \Psi^0_a\|_{H^m_x L^2_z} \lesssim \|\gamma^0_a\|_{H^{m-\frac{9}{4}}} \lesssim \|\p_x \theta_0\|_{H^{m-\frac{7}{4}}(\Om)},\quad \| \Psi^0_a\|_{L^2_x H^m_z} \lesssim \|\gamma^0_a\|_{H^{\frac{m}{2} - \frac{9}{4}}} \lesssim \|\p_x\theta_0\|_{H^{\frac{m}{2}-\frac{7}{4}}(\Om)},\\
		\| \Psi^1_a\|_{H^m_x L^2_z} \lesssim \|\gamma^1_a\|_{H^{m-\frac{11}{4}}} \lesssim\|\p_x\theta_0\|_{H^{m-\frac{5}{4}}(\Om)}  ,\quad \| \Psi^1_a\|_{L^2_x H^m_z} \lesssim \|\gamma^1_a\|_{H^{\frac{m}{2} - \frac{11}{4}}} \lesssim \|\p_x \theta_0\|_{H^{\frac{m}{2}-\frac{5}{4}}(\Om)}  .
		\ea
		\end{equation*}
		Additionally, the profiles $\Psi^0_a$ and $\Psi^1_a$ have exponential decay: there exists a universal constant $\bar c >0$ such that for any $Z_0\geq 1$ and any $m\in \N$,
		\[\ba 
		\|\Psi^0_a\|_{H^{m}(\T \times (Z_0, +\infty))}\lesssim \|\theta_0\|_{H^1(\Om)}\exp(-\bar c Z_0^{4/5}),\\
		\|\Psi^1_a\|_{H^{m}(\T \times (Z_0, +\infty))}\lesssim \|\theta_0\|_{H^2(\Om)}\exp(-\bar c Z_0^{4/5}).
		\ea
		\]

		\label{cor:est-Psia-Sobolev}
	\end{corollary}

	\subsection{Construction of an approximate solution}
	
	The idea is now to find a decomposition of $\theta$ as $\theta=\thbl + \thint$, where $\thbl$ is a solution of
	\[
	\p_t \thbl =\p_x^2 \Delta^{-2} \thbl + S_r,
	\]
	with a remainder term $S_r$ such that, for some $\delta>0$,
	\[\ba 
	S_r(t)=O((1+t)^{-2}) \text{ in } L^2(\Om),\quad S_r(t)=O((1+t)^{-1-\delta}) \text{ in } H^4(\Om),\\
	\p_t S_r(t)= O((1+t)^{-3}) \text{ in } L^2(\Om),
	\ea
	\]
	and a boundary layer profile $\thbl$ such that ${\thbl}\vert_{\p\Om}(t=0)={\theta_0}\vert_{\p\Om}$, ${\p_ n\thbl}\vert_{\p\Om}={\p_\bn\theta}\vert_{\p\Om}$.
	Recall that the operator $\Delta^{-2}$ is endowed with homogeneous conditions for the trace and the normal derivative on the boundary of $\p\Om$.

	As a consequence, the interior part $\thint$ satisfies
	\[
	\p_t \thint =\p_x^2 \Delta^{-2} \thint - S_r
	\]
	and the trace of $\thint$ vanishes on $\p\Om$, together with its normal derivative. Thus we may apply \cref{prop:bb-uniformbound} and \cref{prop:bb-decay}, and we obtain $\|\thint\|_{L^2}=O((1+t)^{-1})$, which will complete the proof of \cref{thm:BL-linear}.

	The main order part of $\thbl$ will be given by the profiles $\Theta^j_a$, $j=0,1$, $a\in \{\mtop,\mbot\}$ constructed in \cref{cor:est-Psia-Sobolev}.
	However, a few adjustments must be made in order to have a suitable decomposition:
 
$\bullet$ First, the profiles $\Theta^j_a$ must be truncated away from $z=0$ and $z=1$,  so that their (exponentially small) trace does not pollute the opposite boundary.
		Since $\Theta^j_a$ has exponential decay, this introduces a remainder of order $\exp(-ct^{1/5})$, which will be included in $S_r$. More precisely, the error terms generated by this truncation will be dealt with thanks to the following Lemma, whose proof is left to the reader:
		
		\begin{lemma}
			Let $\Psi \in L^2(\T \times (0, + \infty))$ such that there exist $c, C>0$ such that
			\[
			\int_{\T}\int_0^\infty |\Psi(x, Z)|^2 \exp(c Z^{4/5})\ud Z\ud x\leq C<+\infty.
			\]
			Let $\zeta \in L^\infty(0,1)$ such that $\supp \zeta \subset (1/4, 1).$ Then there exists a constant $c'>0$, depending only (and explicitly) on $c$, such that
			\[
			\left\|\Psi(x, (1+t)^{1/4} z)   \zeta(z) \right\|_{L^2(\T)} \lesssim C \| \zeta\|_\infty \exp(-c' (1+t)^{1/5}).
			\]
			
			\label{lem:reste-exponentiel}
		\end{lemma}

$\bullet$ More importantly, the main order profiles $(\Theta^j_a,\Psi^j_a)$ for $j=0,1$ do not satisfy exactly
		\[
		\Delta^2 \left((1+t)^{-1} \Psbot^j(x,(1+t)^{1/4} z)\right)= \Thbot^j (x,(1+t)^{1/4} z).
		\]
		Indeed, when constructing $\Psi^0_a$, we only kept the main order terms in $\Delta^2$, \ie the $z$ derivatives. It turns out that the term $2\p_x^2\p_z^2$ in the bilaplacian generates an error term in the equation which is not $O((1+t)^{-2})$. As a consequence, we introduce lower order correctors, whose purpose is precisely to cancel this error term.
		We emphasize that the construction of such additional correctors is quite classical in multiscale problems. In order to determine the order at which the expansion can be stopped, we will rely on the following Lemma, whose proof is postponed to the end of this section: 
		
		\begin{lemma}\,
		\begin{itemize}
		    \item 	Let $f\in H^4(\T, L^2(\R_+))$ such that there exist constants $c,C_1>0$ such that for all $k\in \{0,\cdots,4\}$,
			\beq\label{hyp-f-exp-decay}
			|\p_x^k f(x,Z)|\leq C_1 \exp(-c Z^{4/5})\quad  \forall (x,Z)\in \T\times\R_+.
			\eeq
			Then there exists a constant $C$ depending only on $c$ such that
			\[
			\left\|\Delta^{-2} \left(f(x, (1+t)^{1/4} z)\chi(z)\right)\right\|_{L^2}\leq  \frac{C C_1}{(1+t)^{3/4}}.
			\]
			
			Furthermore, if
			\[
			\int_0^\infty Z^2 f(x,Z)\ud Z= \int_0^\infty Z^3 f(x,Z)\ud Z=0 \quad \forall x\in \T,
			\] 
			this estimate becomes
			\[
			\left\|\Delta^{-2} \left(f(x, (1+t)^{1/4} z)\chi(z)\right)\right\|_{L^2}\leq \frac{C C_1}{1+t}.
			\]
			
			\item Let $f\in  H^2(\T, L^2(\R_+))$ such that \eqref{hyp-f-exp-decay} holds for all $k\in \{0,1,2\}$. 
			Then there exists a constant $C$ depending only on $c$ such that
			\[
			\left\|\Delta^{-2} \left(f(x, (1+t)^{1/4} z)\chi(z)\right)\right\|_{L^2}\leq  \frac{C C_1}{(1+t)^{1/2}}.
			\]
		\end{itemize}

			\label{lem:taille_H-2}
		\end{lemma}

	With the two above Lemmas in mind, we define $\thbl$ in the following way. 
	Let $\chi\in C^\infty_c(\R)$ be a cut-off function such that $\chi\equiv 1$ on $(-1/4,1/4)$, and $\supp \chi\subset(-1/2,1/2)$.
	We look for $\thbl$ in the form
	\begin{align*}
		\thbl(t,x,z) &:= \begin{aligned}[t]
            & \sum_{j=0}^4 (1+t)^{-j/4} \Thbot^j(x,(1+t)^{1/4} z)\chi(z)\\
		      &+\sum_{j=0}^4 (1+t)^{-j/4} \Thtop^j(x,(1+t)^{1/4} (1-z))\chi(z-1)
        \end{aligned}\\
		&=: \thbl_\mbot + \thbl_\mtop 
	\end{align*}
	and
	\begin{align*}
		\psbl(t,x,z) &:= \begin{aligned}[t]
            & \sum_{j=0}^4 (1+t)^{-1-\frac{j}{4}} \Psbot^j(x,(1+t)^{1/4} z)\chi(z)\\
		      &+\sum_{j=0}^4 (1+t)^{-1-\frac{j}{4}} \Pstop^j(x,(1+t)^{1/4} (1-z))\chi(z-1)
        \end{aligned}\\
		&=: \psbl_\mbot + \psbl_\mtop .
	\end{align*}
	The profiles $\Theta_a^j,\Psi_a^j$ for $j=0,1$ and $a\in \{\mbot,\mtop\}$ were defined in the previous subsection, and we now proceed to define $\Theta_a^j,\Psi_a^j$ for $j\geq 2$. The reason why we stop the expansion at $j=4$ follows from \cref{lem:taille_H-2}, as we will see shortly.
	
	We focus on the part near $z=0$, since the part near $z=1$ works identically. Setting $Z=(1+t)^{1/4} z$, we have
	\begin{equation*}
		\frac{\p}{\p t}\thbl_\mbot
		= (1+t)^{-1}\sum_{j=0}^4(1+t)^{-j/4} \left[-\frac{j}{4}\Thbot^j(x,Z) + \frac{1}{4}Z\p_Z \Thbot^j(x,Z)\right] \chi(z)\\
	\end{equation*}
	For $j=0,1$, the bracketed term in the right-hand side is simply $\p_x \Psbot^j(x,(1+t)^{1/4} z)$.
	Similarly, we choose $\Psi_a^j$ for $j=2,3, 4$ and $a\in \{\mbot,\mtop\}$ so that \beq\label{def:psiaj-thetaaj}
	\p_x \Psi_a^j= -\frac{j}{4}\Theta_a^j + \frac{1}{4}Z\p_Z \Theta_a^j.
	\eeq
	With this choice, we have
	\[
	\p_t \thbl=\p_x \psbl.
	\]
	There remains to choose $\Theta_a^j$ so that $\p_x\psbl=\Delta^{-2}\p_x^2 \thbl + O((1+t)^{-2})$ in $L^2$. To that end, we observe
	\begin{equation*} \ba
		\Delta^2\psbl_\mbot
		={}&\sum_{j=0}^4 (1+t)^{-\frac{j}{4}} \p_Z^4 \Psbot^j(x,Z) \chi(z)+2 \sum_{j=0}^4(1+t)^{-\frac{1}{2}-\frac{j}{4}} \p_x^2\p_Z^2 \Psbot^j(x,Z)\chi(z)\\
		&+ \sum_{j=0}^4 (1+t)^{-1-\frac{j}{4}} \p_x^4\Psbot^j(x,Z)\chi(z)\\
		&+\sum_{j=0}^4\sum_{k=0}^3\begin{pmatrix}k\\4\end{pmatrix}
		(1+t)^{-1+\frac{k-j}{4}}\p_Z^k\Psbot^j(x,Z) \chi^{(4-k)}(z)\\
		&+2\sum_{j=0}^4\sum_{k=0}^1 \begin{pmatrix}k\\2\end{pmatrix}(1+t)^{-1+\frac{k-j}{4}}\p_x^2\p_Z^k\Psbot^j(x,Z)\chi^{(2-k)}(z)
		.
	\ea \end{equation*}
	The last two terms are handled by \cref{lem:reste-exponentiel} (anticipating that $\Psi^j_a$ will have exponential decay for $j=2,3,4$).

	We obtain
	  \begin{equation} \label{eq:Delta2Psbl} \ba
        \Delta^2\psbl_\mbot
        ={}&\p_x \thbl_\mbot + O(\exp(-c'(1+t)^{1/5}))\\
        &+ (1+t)^{-1/2}\left[ -\p_x \Thbot^2 + \p_Z^4 \Psbot^2 + 2\p_x^2\p_Z^2 \Psbot^0\right](x,Z)\chi(z)\\
        &+(1+t)^{-3/4} \left[-\p_x \Thbot^3 + \p_Z^4\Psbot^3 + 2\p_x^2\p_Z^2 \Psbot^1\right](x,Z)\chi(z)\\
        &+ (1+t)^{-1} \left[-\p_x \Thbot^4 + \p_Z^4\Psbot^4 + 2\p_x^2\p_Z^2 \Psbot^2 + \p_x^4 \Psi^0\right](x,Z)\chi(z)\\
        &+\sum_{j\geq 5}(1+t)^{-j/4}\Phi^j_\mbot(x,Z)\chi(z),
    \ea \end{equation}
	for some functions $\Phi^j_\mbot$ depending on the profiles $\Psbot^j$ (for instance $\Phi^5= 2\p_x^2 \p_Z^2 \Psbot^3 + \p_x^4\Psbot^1$).
	Thanks to \cref{lem:taille_H-2}, the 
	inverse bilaplacian of the last term has a size of order $(1+t)^{-2}$ in $L^2$. Hence it 
	will be included in the remainder $S_r$. 
	Note that the reason why we need to stop the expansion in $\thbl$ at $j=4$ is dictated by the above formula and by \cref{lem:taille_H-2}. If we stop the expansion for a lower $j$, then the remainder may be greater than $(1+t)^{-2}$ in $L^2$.

	Therefore we focus on the terms of order $(1+t)^{-j/4}$ with $j=2,3,4$. We treat the cases $j=2$ and $j=3$ simultaneously, and we will focus on the case $j=4$ later.
	
	$\bullet$ \textbf{Construction of $\Psi^j_a$ for $j=2,3$:}\\
	Remembering \eqref{def:psiaj-thetaaj}, we choose $\Theta_a^j$ and $\Psi_a^j$ for $a\in \{\text{bot, top}\}$ and $j=2,3$ so that
	\begin{equation*}
		\ba
		\p_x \Psi_a^j&= -\frac{j}{4}\Theta_a^j + \frac{1}{4}Z\p_Z \Theta_a^j,\\
		-\p_x \Theta_a^j + \p_Z^4 \Psi_a^j + 2\p_x^2\p_Z^2 \Psi_a^{j-2}&=0,\ea
			\end{equation*}
		endowed with the boundary conditions
		\[\lim_{Z\to \infty}\Psi_a^j=0,\quad
		\Psi_a^j(Z=0)=\p_Z \Psi_a^j(Z=0)=
		\Theta_a^j(Z=0)=\p_Z\Theta_a^j(Z=0)=0.
		\]
	As before, we note that the boundary conditions at $Z=0$ are redundant.
	Eliminating $\Theta_a^j$ from the equation, we find that $\Psi_a^j$ satisfies
	\begin{equation}
		\label{eq:Psbl2}
		\ba     
		Z\p_Z^5\Psi_a^j - j \p_Z^4 \Psi_a^j &= 4 \p_x^2 \Psi_a^j + S_a^j,\\
		\Psi_a^j(Z=0)=\p_Z \Psi_a^j(Z=0)&=0,\\
		\p_Z^4 \Psi_a^j=-2\p_x^2 \p_Z^2\Psi_a^{j-2}&=-\frac{1}{j}S_a^j\quad \text{at }Z=0,\\
		\p_Z^5 \Psi_a^j=-2\p_x^2 \p_Z^3\Psi_a^{j-2}&= -\frac{1}{j-1}\p_Z S_a^j\quad \text{at }Z=0,\\
		\lim_{Z\to \infty}\Psi_a^j&=0,
		\ea
	\end{equation}
	where $S_a^j=-2(Z\p_Z-j)\p_x^2 \p_Z^2 \Psi^{j-2}_a$. 
	Therefore
  \begin{equation*}
    \p_Z^j S_a^j=-2Z\p_Z^{j+3}\p_x^2 \Psi^{j-2}_a= -8 \p_x^4\p_Z^{j-2} \Psi^{j-2}_a.
  \end{equation*}
	
	As a consequence, $\p_Z^j {\Psi^j_a}$ is a solution of
	\begin{equation*} 
	\begin{cases}
	    	Z\p_Z^5\p_Z^j {\Psi_a^j} = 4 \p_x^2 \p_Z^j {\Psi_a^j}  -8 \p_x^4\p_Z^{j-2} \Psi^{j-2}_a,\\
	\p_Z^j {\Psi_a^j}=0\quad \text{at }Z=0,\\
	\p_Z^4 \Psi_a^j=-2\p_x^2 \p_Z^2\Psi_a^{j-2},\quad 
	\p_Z^5 \Psi_a^j= -2\p_x^2 \p_Z^3\Psi_a^{j-2}\quad \text{at }Z=0,\\
	\lim_{Z\to \infty}\p_Z^j\Psi_a^j=0.
	\end{cases}
	\end{equation*}
	Note that the boundary condition $\p_Z^j {\Psi^j_a}(Z=0)=0$ follows from the identity
	\[
	\p_x \p_Z^j \Psi_a^j=\frac{1}{4} Z \p_Z^{j+1} \Theta_a^j.
	\]
	Taking the Fourier transform with respect to $x$, we observe that $\widehat{\p_Z^j{\Psi_a}}(k)$ satisfies \eqref{ODE} with nonhomogeneous boundary conditions of type (iii) (for $j=2$) or (iv) (for $j=3$). Using the Fourier representations \eqref{formula-Fourier-PsiO} and \eqref{formula-Fourier-Psi1} for $\Psi^0$ and $\Psi^1$, we anticipate that $\Psi^2_a$ and $\Psi^3_a$ can be written as
	\beq \label{Fourier-Psi2-Psi3}
	\ba 
	\Psi^2_a(x,Z)=\sum_{k\in \Z\setminus \{0\}} |k|^{-1}\widehat{\gamma^0_{a,k}}\; \chi_2(|k|^{1/2} Z)e^{ikx},\\
	\Psi^3_a(x,Z)=\sum_{k\in \Z\setminus \{0\}} |k|^{-3/2}\widehat{\gamma^1_{a,k}}\; \chi_3(|k|^{1/2} Z)e^{ikx}.
	\ea
	\eeq
	with $\chi_2,\chi_3 \in C^\infty((0, +\infty))$ decaying like $\exp(-\bar c Z^{4/5})$. The precise construction of $\chi_2$ and $\chi_3$ will be performed below.
	We obtain the following result:
	\begin{lemma}
		Let $a\in \{\mtop, \mbot\}$ and $\gamma^0_a,\gamma^1_a\in L^2(\T)$.
		Consider the solutions $\Psi^0_a,\Psi^1_a$ of \eqref{eq:Psi_i}, \eqref{eq:Phi_i} given by \cref{cor:est-Psia-Sobolev}.
		
		Then there exist unique solutions $\Psi^2_a\in H^{5/4}_xL^2_Z\cap L^2_x H^{5/2}_Z$, $\Psi^3_a \in H^{7/4}_x L^2_Z \cap L^2_x H^{7/2}_Z$ of \eqref{eq:Psbl2}. Furthermore, for any $m\in \N$,
		\begin{equation*} \ba
    		\| \Psi^2_a\|_{H^m_x L^2_Z} &\lesssim \|\p_x \theta_0\|_{H^{m-\frac{3}{4}}(\Om)},
            & \| \Psi^2_a\|_{L^2_x H^m_Z} &\lesssim \|\p_x \theta_0\|_{H^{\frac{m}{2}-\frac{3}{4}}(\Om)},\\
    		\| \Psi^3_a\|_{H^m_x L^2_Z} &\lesssim\|\p_x \theta_0\|_{H^{m-\frac{1}{4}}(\Om)},
            & \| \Psi^3_a\|_{L^2_x H^m_Z}  &\lesssim \|\p_x \theta_0\|_{H^{\frac{m}{2}-\frac{1}{4}}(\Om)}  .
		\ea \end{equation*}
		Additionally, the profiles $\Psi^2_a$ and $\Psi^3_a$ have exponential decay: for any $Z_0\geq 1$, for any $m\in \N$,
		\[\ba 
		\|\Psi^2_a\|_{H^{m}(\T \times (Z_0, +\infty))}\lesssim \|\theta_0\|_{H^1(\Om)}\exp(-\bar c Z_0^{4/5}),\\
		\|\Psi^3_a\|_{H^{m}(\T \times (Z_0, +\infty))}\lesssim \|\theta_0\|_{H^2(\Om)}\exp(-\bar c Z_0^{4/5}).
		\ea
		\]
		
		\label{lem:psi2-psi3}
	\end{lemma}

	\begin{proof}
		In view of \eqref{Fourier-Psi2-Psi3}, it is sufficient to construct $\chi_2$ and $\chi_3$.
		We first construct the solution $\phi_j$ of
		\[
		\begin{cases}
		    	Z \p_Z^5\phi_j(Z)= -4 \phi_j(Z) - 8 \p_Z^{j-2} \chi_{j-2},\\
		\phi_j(0)=0,\quad \p_Z^{4-j} \phi_j(0)= - 2 \p_Z^2\chi_{j-2}(0),\quad \p_Z^{5-j} \phi_j(0)&= -2 \p_Z^3\chi_{j-2} (0),\\
		\lim_{Z\to \infty} \phi_j(Z)=0.
		\end{cases}
		\]
		Note that after a suitable lifting, $\phi_j$ satisfies \eqref{ODE}
		with the boundary conditions  (iii) from \cref{lem:WP-ODE-BL} (for $j=2$) or (iv) (for $j=3$). Hence the existence and uniqueness of $\phi_j$ (and its exponential decay) follow from \cref{lem:WP-ODE-BL}.
		Now, define $\chi_j$ as
		\[
		\p_Z^j \chi_j= \phi_j, \quad \p_Z^k \chi_j(+\infty)=0\text{ for }0\leq k\leq j-1.
		\]
		It follows that $\chi_j$ decays like $\exp(-\bar c Z^{4/5})$.
		Furthermore, by construction
		\[
		\p_Z^j\left[Z\p_Z^5 \chi_j - j \p_Z^4 \chi_j + 4 \chi_j - 2 (Z\p_Z - j) \p_Z^2\chi_{j-2}\right]=0.
		\]
		We deduce that $Z\p_Z^5 \chi_j - j \p_Z^4 \chi_j + 4 \chi_j + 2 (Z\p_Z - j) \p_Z^2\chi_{j-2}$ is a polynomial of order at most $j-1$, which has exponential decay at infinity. Therefore, the following equality holds
		\[
		Z\p_Z^5 \chi_j - j \p_Z^4 \chi_j + 4 \chi_j - 2 (Z\p_Z - j) \p_Z^2\chi_{j-2}=0.
		\]
		Taking the trace of the above identity at $Z=0$, we infer that $\chi_j(0)=0$. In a similar way, we also find that $\chi_j'(0)=0$. Now, defining $\Psi^j_a$ by \eqref{Fourier-Psi2-Psi3}, we obtain that $\Psi^j_a$ satisfies \eqref{eq:Psbl2}. The Sobolev estimates are then a consequence of the Fourier representation formula.
		\end{proof}

	$\bullet$ \textbf{Construction of $\Psi^4_a$:}\\
	The definitions of $\Psi^4_a$ and $\Theta^4_a$ are similar. We choose $\Psi^4_a$ such that
	\[
	\begin{cases}
	    	Z\p_Z^5\Psi_a^4 - 4 \p_Z^4 \Psi_a^4 = 4 \p_x^2 \Psi_a^4 + S_a^4,\\
	\Psi_a^4(0)=\p_Z \Psi_a^4(0)=0,\quad
	\p_Z^4 \Psi_a^4=-\frac{1}{4}S_a^4,\quad 
	\p_Z^5 \Psi_a^4= -\frac{1}{3}\p_Z S_a^4\quad \text{at }Z=0,\\
	\lim_{Z\to \infty}\Psi_a^4=0,
	    	\end{cases}
	\]
	where 
	\[
	S_a^4=-(Z\p_Z - 4) (2 \p_x^2 \p_Z^2 \Psi^2_a + \p_x^4 \Psi^0_a).
	\]
	Therefore the Fourier transform of $\p_z^4 \Psi^4_a$, after a suitable lifting, is a solution of \eqref{ODE}.
	The main difference with the construction of $\Psi^j_a$ for $j\leq 3$ lies in the fact that $\p_Z^4\Psi^4_a$ is not fully determined. Indeed, we lack a boundary condition on $\p_Z^k \Psi^4_a$ for some $k\geq 6$.
	Once again, this phenomenon (a high order corrector is under-determined) is quite common in multiscale problems. \label{def:psi4}
	In fact it turns out that $\Psi^4_a$ could be determined in a unique fashion if we were looking for a higher order expansion (see \cref{rmk:higher-order-expansion}). In this case, we should choose $\Psi^4_\mbot$ so that $\p_Z^4 \Theta^4_{\mbot}\vert_{Z=0}$ lifts the trace of $\Delta^2 \theta\vert_{z=0}$.
	In the present case, since we merely wish to close the first order expansion, we simply further require that $\p_Z^8 \Psi^4_a\vert_{|Z=0}=0$, so that the lifted Fourier transform of $\p_Z^4 \Psi^4_a$ satisfies the boundary conditions (i) of \cref{lem:WP-ODE-BL}.
	We conclude that $\Psi^4_a$ is well-defined and satisfies the same estimates as $\Psi^j_a$ for $j\leq 3$. The details of the proof are left to the reader.

	\subsection{Estimate of the remainder and conclusion}
	At this stage, we have constructed $\thbl$ such that for all $t\geq 0$,
	\[\ba
	\thbl(t)\vert_{\p\Om}=\theta(t)\vert_{\p\Om}=\theta_0\vert_{\p\Om},\\
	\p_\bn\thbl(t)\vert_{\p\Om}=\p_\bn\theta(t)\vert_{\p\Om}=\p_\bn\theta_0\vert_{\p\Om},
	\ea
	\]
	and
	\[
	\p_t \thbl = \Delta^{-2}\p_x^2\thbl +\Delta^{-2}\p_x T_r + O(\exp(-c(1+t)^{1/5})\quad \text{in } L^2,
	\]
	where $T_r=T_\mtop+ T_\mbot$ and
	\[
	T_\mbot:=\left[\sum_{j\geq 5}(1+t)^{-j/4}\Phi^j_\mbot(x,(1+t)^{1/4}z)\chi(z)\right],
	\]
	with a similar expression for $T_\mtop$.
	According to \cref{lem:taille_H-2},
	\[
	\|\Delta^{-2} \p_x T_r\|_{L^2}\lesssim \|\theta_0\|_{H^s} (1+t)^{-2},\quad \|\p_t \Delta^{-2} \p_x T_r\|_{L^2}\lesssim \|\theta_0\|_{H^s} (1+t)^{-3}.
	\]
	Furthermore,
	\[
	\|\Delta^{-2} \p_x T_r\|_{H^4}\lesssim\|\p_x T_r\|_{L^2}\lesssim   \|\theta_0\|_{H^s} (1+t)^{-5/4}
	\]
 for some finite (and computable) index $s>0$.
	Therefore $\thint=\theta-\thbl$ solves
	\[
	\p_t \thint = \p_x^2 \Delta^{-2} \thint -\Delta^{-2}\p_x T_r + O(\exp(-c(1+t)^{1/5}),
	\]
	and $\thint=\p_\bn\thint=0$ on $\p\Om$. 
We first apply \cref{prop:bb-uniformbound} to $\Delta^2 \thint$ and find that $\|\Delta^2 \thint(t)\|_{L^2} \lesssim \|\theta_0\|_{H^s}$ for all $t\geq 0$, for some finite $s$. 
 From there, we apply \cref{prop:bb-decay} to $\p_x^2 \thint$ with $\alpha=0$, and we deduce that $\|\p_x^2 \thint(t)\|_{L^2} \lesssim  \|\theta_0\|_{H^s} (1+t)^{-1}$. 
 As in \cref{sec:stab}, estimates on $\psiint$ can be obtained by deriving bounds on $\p_t \thint$. More precisely, applying \cref{prop:bb-decay} to $\p_t \thint$, we find that $\| \p_t \thint(t)\|_{L^2} \lesssim \|\theta_0\|_{H^s}(1+t)^{-2} $, and therefore $\| \p_x^2 \Delta^{-2} \thint\|_{L^2} \lesssim \|\theta_0\|_{H^s}(1+t)^{-2}$.
 This completes the proof of \cref{thm:BL-linear}.

	\begin{remark}[Construction of an approximation at any order]
		\label{rmk:higher-order-expansion}
		
		Since $\Delta^2 \theta$ solves the same equation as $\theta$, one can easily iterate this construction. More precisely, if $\theta_0\in H^{4k}$, it can be proved that there  exist sequences of profiles $(\Thbot^j,\Thtop^j)_{0\leq j\leq 4k}$ such that the following result holds:
		\[
		\theta(t,x)=\sum_{j=1}^{4k}(1+t)^{-j/4} \left[ \Thbot^j(x,(1+t)^{1/4} z)\chi(z) + \Thtop^j(x,(1+t)^{1/4} (1-z))\chi(z-1)\right] + \theta_\rem^j(t),
		\]
		and
		\[
		\|\theta_\rem^j(t)\|_{L^2}\lesssim \frac{1}{(1+t)^k},\quad \|\theta_\rem^j(t)\|_{H^{4k}}\lesssim 1.
		\]
		For instance, the role of $\Thbot^{4j}$ is to lift the trace of $\Delta^{2j} \theta$ at $z=0$, the one of $\Thtop^{4j+1}$ is to lift the one of $\p_z\Delta^{2j} \theta$ at $z=1$, etc.
		
		The details of the construction are very similar to the ones of the profiles $\Theta_a^j$ for $0\leq j\leq 3$ above and are left to the reader.
	\end{remark}

	\subsection{Proof of \texorpdfstring{\cref{lem:taille_H-2}}{Lemma 3.6}}

		We first define a function $f_1$ such that 
		\[
		\p_Z^4 f_1=f,
		\]
		and $\p_Z^k f_1(+\infty)=0$ for $0\leq k\leq 3$. Note that the exponential decay assumption on $f$ ensures that $f_1 $ exists, and $f_1\in W^{4,\infty}\cap H^4$. Moreover, for $0\leq m_1,m_2\leq 4$,
		\[
		|\p_x^{m_1}\p_Z^{m_2} f_1(x,Z)| \leq C \exp(-c Z^{1/5}),
		\]
		with possibly different constants $C$ and $c$.
		Setting $Z=(1+t)^{1/4} z$, we infer that
		\begin{equation*} \ba
			\Delta^2\left( (1+t)^{-1} f_1(x,Z) \chi(z)\right) ={}&f(x,Z)\chi(z) \\
            &+ 2(1+t)^{-1/2} \p_x^2\p_Z^2 f_1(x,Z)\chi(z) \\
			&+ (1+t)^{-1} \p_x^4 f_1(x,Z) \chi(z) + O(e^{-ct^{1/5}})\text{ in } L^2,
		\ea \end{equation*}
		where the  term $O(e^{-ct^{1/5}})$ stems from the commutator involving derivatives of $\chi$ (see \cref{lem:reste-exponentiel}).
		Note that $\p_x^2\p_Z^2 f_1$ satisfies the same decay assumptions as $f$, and therefore we can lift it by another corrector $f_2$ such that
		\[
		\p_Z^4 f_2= -2\p_x^2\p_Z^2 f_1,
		\]
		\ie $\p_Z^2 f_2=-2\p_x^2 f_1$. Therefore
		\[
		\Delta^2\left( \left((1+t)^{-1} f_1(x,Z) + (1+t)^{-3/2} f_2(x,Z)\right) \chi(z)\right)= f(x,Z)\chi(z) + O((1+t)^{-1})\quad \text{in }H^{-2}.
		\]
		The only remaining issue lies in the fact that $f_1,f_2$ and their normal derivatives do not vanish on the boundary.
		Hence we set $a_i(x)=f_i(x,0)$, $b_i(x)=\p_Z f_i(x,0)$, and we add a corrector
		\[
		f_3(t,x,z):=-\sum_{i=1,2} (1+t)^{-\frac{i-1}{2}}(a_i(x) + z(1+t)^{1/4} b_i(x))\chi(z).
		\]
		Now
		\begin{multline*}
			\Delta^2\left( \left((1+t)^{-1} f_1(x,Z) + (1+t)^{-3/2} f_2(x,Z)\right) \chi(z) + (1+t)^{-1} f_3\right)\\= f(x,Z)\chi(z) + O((1+t)^{-3/4})\quad \text{in }H^{-2},  
		\end{multline*}
		and for $k=0,1$
		\[
		\p_z^k\left(\left((1+t)^{-1} f_1(x,Z) + (1+t)^{-3/2} f_2(x,Z)\right) \chi(z) + (1+t)^{-1} f_3\right)\vert_{\p\Om}=0.
		\]
		It follows that
		\begin{multline*}
			\left((1+t)^{-1} f_1(x,Z) + (1+t)^{-3/2} f_2(x,Z)\right) \chi(z) + (1+t)^{-1} f_3\\= \Delta^{-2} (f(x,Z)\chi(z)) +  O((1+t)^{-3/4})\quad \text{in } H^2.
		\end{multline*}
		
		Let us now prove that when $\int_0^\infty Z^2 f(\cdot, Z)\ud Z=\int_0^\infty Z^3 f(\cdot, Z)\ud Z=0$, we gain an additional factor  $(1+t)^{-1/4}$. It can be easily checked that
		\[
		f_1\vert_{Z=0}=\frac{1}{6}\int_0^\infty Z^3 f(\cdot, Z)\ud Z=0,\quad \p_Z f_1\vert_{Z=0}=-\frac{1}{2}\int_0^\infty Z^2 f(\cdot, Z)\ud Z=0.
		\]
		
		Hence, with the notation above, $a_1=b_1=0$ and therefore $f_3=O((1+t)^{-1/4})$.
		With the same arguments, we infer that
		\begin{multline*}
			\left((1+t)^{-1} f_1(x,Z) + (1+t)^{-3/2} f_2(x,Z)\right) \chi(z) + (1+t)^{-1} f_3\\= \Delta^{-2} (f(x,Z)\chi(z)) +  O((1+t)^{-1})\quad \text{in } H^2.
		\end{multline*}
	\qed
	
	\begin{remark}\label{rmk:taile_H-2-autosim}
	    Note that the first statement of \cref{lem:taille_H-2} provides a better decay of the $H^{-2}$ norm, but the second one requires less horizontal derivatives on $f$. 
	    In the next section, we will also use the following variant: assume that there exists a sequence $(\gamma_k)_{k\in \Z}$ such that
	    \[
	    f(x,Z)=\sum_{k\in \Z\setminus\{0\}}\gamma_k e^{ikx} \varphi(|k|^{1/2} Z),
	    \]
	    where $\varphi \in C^\infty(\R)$ decays like $C_1\exp(-c Z^{4/5})$. Then, following the previous computations,
	    \[
	    \ba 
	    f_1(x,Z)&=\sum_{k\in \Z\setminus\{0\}}|k|^{-2}\gamma_k e^{ikx} \varphi^{(-4)}(|k|^{1/2} Z),\\
	    f_2(x,Z)&=2\sum_{k\in \Z\setminus\{0\}}|k|^{-1}\gamma_k e^{ikx} \varphi^{(-6)}(|k|^{1/2} Z),
	    \ea
	    \]
	    where $\p_Z^m \varphi^{(-m)}=\varphi$, and $\varphi^{(-m)}(+\infty)=0.$ Hence
	    \[
	    \left\| \Delta^{-2} (f(x,(1+t)^{1/4} z)\chi(z))\right\|_{L^2} \lesssim C_1 \left(\sum_{k\in \Z\setminus\{0\}} |k|^2 |\gamma_k|^2\right)^{1/2} (1+t)^{-3/4}.
	    \]
	   \end{remark}

	\section{Long time boundary layers in the nonlinear setting\texorpdfstring{: proof of \cref{thm:BL-nonlinear}}{}}
	\label{sec:nonlinear-BL}
	
	We now go back to the long time analysis of 
	\eqref{eq:STi} when 
	$\theta'_0=\p_\bn\theta'_0=0$ on $\p\Om$.
	We recall (see \cref{thm:stab}) that in this case, $\theta'(t)$ converges towards zero in $H^s$ for all $s<4$ as $t\to \infty$.
	
	A natural question is to investigate whether the algebraic decay rate provided by \cref{thm:stab} can be improved, possibly at the cost of a stronger regularity requirement on the initial data. In other words, if we assume that $\theta_0\in H^s$ with $s $ large, can we prove a uniform $H^s$ bound on a solution, and thereby a higher decay estimate on $\theta'$?
	
	As explained in the Introduction and in \cref{sec:stab}, such a result does not follow immediately from an induction argument. Indeed, the traces of $\Delta^2 \theta'$ and of $\p_\bn \Delta^2 \theta'$ do not vanish on the boundary (even when the traces of $\Delta^2 \theta'_0$ and of $\p_\bn \Delta^2 \theta'_0$ do), and therefore we cannot apply \cref{prop:bb-decay} to $\Delta^2 \theta'$.
	
	However, it turns out that when $\p_z^2 \bar \theta_0\vert_{\p\Om}=0$, we can use (a variant of) the  linear analysis of \cref{sec:linear-BL} to analyze the long time behavior of $\Delta^2 \theta'$. In other words, in this case, there are boundary layers in the vicinity of the boundary, but they are driven by a linear mechanism.
	\cref{thm:BL-nonlinear} will follow.

	To that end, the strategy is to consider the equation satisfied by $\Delta^2 \theta'$. As we have seen previously, the structure of the equation is overall the same. The main difference lies in the fact that the traces of $\Delta^2 \theta'$ and $\p_\bn \Delta^2 \theta'$ do not vanish on the boundary, which makes the situation rather close to the one described in \cref{thm:BL-linear}. Following the methodology of the previous section, we may lift them thanks to a  corrector which remains linear at main order.
	Modifying slightly the bootstrap argument from \cref{sec:stab} in order to account for these boundary layers, we eventually prove \cref{thm:BL-nonlinear}, or rather the following more precise version:

\begin{proposition}
There exists a universal constant $\eps_0$ such that the following statement holds.
Let $\theta_0\in H^{14}(\Om)$ such that $\theta_0\vert_{\p\Om}= \p_\bn \theta_0\vert_{\p\Om}=0$, and $\p_z^2 \bar \theta_0\vert_{\p\Om}=0$.
Assume that $\| \theta_0\|_{H^{14}}\leq \eps_0.$

There exists a boundary layer profile $\thbl\in L^\infty_\mathrm{loc}(\R_+, H^9(\Om))$, given by
\[
\thbl=\sum_{j=0}^4(1+t)^{-1-\frac{j}{4}} \left(\Thbot^j(x, Z_\mbot) + \Thtop^j(x,Z_\mtop)\right) ,%
\]
where %
$Z_\mbot=z(1+t)^{1/4}$, $Z_\mtop=(1-z)(1+t)^{1/4}$, and $\Theta^j_a\in H^9(\T\times \R_+)$, such that the following estimates hold on $\threm:=\theta'-\thbl$ for all $t\geq 0$:
\begin{align*}  
\|\p_x^4 \threm(t)\|_{L^2} &\lesssim \|\theta_0\|_{H^{14}} (1+t)^{-2},\\
\|\p_x^2 \Delta^2 \threm (t)\|_{L^2} &\lesssim \|\theta_0\|_{H^{14}} (1+t)^{-1},\\
\|\Delta^4 \threm(t)\|_{L^2} &\lesssim \|\theta_0\|_{H^{14}} ,\\
\|\p_x^6 \Delta^{-2} \threm(t)\|_{L^2} &\lesssim\|\theta_0\|_{H^{14}} (1+t)^{-3}.
\end{align*}

\label{prop:BL-NL}
    
\end{proposition}

\begin{remark}
    Note that the assumptions of \cref{prop:BL-NL} are slightly weaker than the ones of \cref{thm:BL-nonlinear}. Indeed, we do not require that $\theta_0\in H^3_0$, but rather that $\theta_0\in H^2_0$ and $\p_z^2 \bar \theta_0=0$. According to \cref{lem:traces}, these properties are propagated by the equation. Using the notation of \cref{sec:stab} and setting $G=\p_z \bar \theta$, we infer that $G$ and $\p_z G$ vanish at $z=0$ and $z=1$.
\end{remark}

	\subsection{General strategy}
	\label{sub:strategy-NLBL}

	Following the same strategy as in \cref{sec:linear-BL}, we  look for an Ansatz for $\theta'$ as a sum of a boundary layer part $\thbl$, whose role is to lift the trace of $\Delta^2 \theta'$ and $\p_\bn \Delta^2 \theta'$ on the boundary, and an interior part $\thint$, which vanishes at a high order on the boundary, and for which we will therefore be able to prove better decay estimates.
	Let us give a few additional details on these two parts:
	\begin{itemize}
		\item As in the previous section, the boundary layer term will be defined as an asymptotic expansion in powers of $(1+t)^{-1/4}$, and the width of the boundary layers will also be $(1+t)^{-1/4}$.
		The different terms of the expansion will be constructed recursively: the main order terms will lift the traces of $\Delta^2 \theta'$ and $\p_\bn \Delta^2 \theta'$ (or rather, their limits as $t\to \infty$), and the next order terms will correct error terms generated by the first order ones.
		The precise construction of the boundary layer is the purpose of \cref{sub:NLBL-construction} below.

    \item In fact, $\Delta^2 \threm=\Delta^2(\theta'-\thbl)$ is not identically zero on the boundary, but it is of order $(1+t)^{-1}$. Hence we construct  additional small correctors $\theta_c$, $\sigma^\mathrm{NL}_\mathrm{lift}$, which handle the remaining traces and part of the error term.
  
		\item Thanks to the design of the boundary layer, the remaining interior part $\thint={\threm} -\theta_c- \sigma^\mathrm{NL}_\mathrm{lift}$  is such that
		\[
		\thint=\p_\bn\thint=\Delta^2 \thint=\p_\bn \Delta^2 \thint=0\quad \text{on } \p\Om.
		\]
		As a consequence, $\Delta^4 \thint$ satisfies assumptions that are similar to those of \cref{prop:bb-uniformbound}, and it is reasonable to expect that $\|\Delta^4 \thint\|_{L^2}$ remains uniformly bounded.
    Applying \cref{prop:bb-decay} first to $\p_x^2 \Delta^2 \thint$, and then to $\p_x^4\thint$, we infer that
    $\|\p_x^2 \Delta^2 \thint\|_{L^2}= O((1+t)^{-1})$, and  $\|\p_x^4 \thint\|_{L^2}=O((1+t)^{-2})$. 
		We will use a bootstrap argument to propagate these bounds; the corresponding argument is described in \cref{sub:bootstrap-thint}.

	\end{itemize}

	Before constructing $\thbl $ and proving the decay estimates on $\thint$, some preliminary (and somewhat technical) steps are in order.
	The  traces of $\Delta^2 \theta'$ and $\p_\bn \Delta^2 \theta'$ need to be decomposed as an asymptotic expansion in powers of  $(1+t)^{-1/4}$, in order to identify the relevant boundary conditions for the terms in the expansion of $\thbl$.
	This is performed in \cref{lem:decomp-gamma} below, whose proof  involves some high regularity bounds on $\theta$.
 This is the main reason for the requirement $\theta_0\in H^{14}$ from \cref{thm:BL-nonlinear}.
	As a consequence, the organisation of the rest of this section is the following. In \cref{sub:high-reg-NLBL}, we prove some quantitative $H^{s}$ bounds ($s\leq 14$)  on $\theta'$ under our bootstrap assumption.
	In \cref{sub:decomp-gamma}, we provide a decomposition of $\Delta^2 \theta'\vert_{\p\Om}$ and $\p_\bn \Delta^2 \theta'\vert_{\p\Om}$ under the bootstrap assumption.
	The main results of each section are given in the beginning of the corresponding section. The reader wishing to avoid the technicalities may jump to \cref{sub:NLBL-construction}, in which we construct the boundary layer, using the decomposition of \cref{sub:decomp-gamma} together with arguments from \cref{sec:linear-BL}.
	Eventually, we close the bootstrap argument in \cref{sub:bootstrap-thint}.

	\smallskip

		\smallskip
		
		Let us now introduce the bootstrap assumption that will be used throughout this section.
		We shall decompose $\theta'$ as $\theta'=\thbl + \threm$.
		As explained above, the remainder $\threm$ does not satisfy $\Delta^2 \threm \vert_{\p\Om}=\p_\bn \Delta^2 \threm \vert_{\p\Om}{=0}$, and therefore $\threm$ will be further decomposed into  a sum  of correctors and an interior term.
		
		The term $\thbl$ will take the form
		\begin{equation}\label{forme-thbl}
		\thbl= \frac{1}{1+t} \Theta_\mtop (x,(1+t)^{1/4}(1-z)) + \frac{1}{1+t} \Theta_\mbot (x,(1+t)^{1/4}z) +\text{l.o.t.}
		\end{equation}
		with boundary layer profiles $\Theta_\mtop,\Theta_\mbot $ such that
		\begin{equation*}\label{hyp:bootstrap-thbl}
		\|\Theta_a\|_{H^9(\T\times \R_+)} \leq B
		\end{equation*}
		for some constant $B>0$.
		Note that the amplitude of the boundary layer term $\thbl$ is $O((1+t)^{-1})$, whereas we recall that the amplitude of the boundary layer term in \cref{sec:linear-BL} was $O(1)$ (compare \eqref{forme-thbl} with 
   \eqref{BL-Ansatz-linear}). This is directly linked to the fact that $\theta'\vert_{\p \Om}= \p_\bn\theta'\vert_{\p \Om}=0$ in this Section, while these quantities were non-zero in \cref{sec:linear-BL}. However we keep the same notation for the sake of simplicity.
		
		The remainder term $\threm$ will satisfy the bootstrap assumptions
		\beq \label{hyp:bootstrap-thiint}\ba
		\sup_{t\in [0,T]} (1+t)^2\|\p_x^4 \threm (t)\|_{L^2} + \|\Delta^4 \threm(t)\|_{L^2}\leq B,\\
		\sup_{t\in [0,T]}  (1+t)^3\|\p_t\p_x^4 \threm (t)\|_{L^2} + (1+t)^3 \|\p_x^5 \psrem\|_{L^2}\leq B,
		\ea
		\eeq
		where $\psrem=\Delta^{-2} \p_x \threm$.
		
		As a consequence, our bootstrap assumptions on $\theta'$ read as follows:
		\beq
		\label{hyp:bootstrap-theta'-NLBL}
		\ba
		\forall t \in (0,T),\quad \forall k\in \{4,\cdots,8\} ,\quad \|\p_x^k \theta'\|_{L^2}& \leq B (1+t)^{-9/8} + B (1+t)^{ \frac{k-8}{2}},\\
		\forall t \in (0,T),\quad \forall k\in \{0,\cdots,8\},\quad  \|\p_z^k \theta'\|_{L^2} &\leq B (1+t)^{-\frac{9}{8} + \frac{k}{4}}, \quad \| \p_x^4 \Delta^2 \theta'\|_{L^2}\leq B,\\
		\forall t \in (0,T), \quad \|\p_x^5 \psi(t)\|_{L^2} &\leq B (1+t)^{-17/8}.
		\ea
		\eeq
		Note that these assumptions imply in particular that for $0\leq k\leq 3$
		\beq\label{est:psi-W3infty}
		\|\psi\|_{W^{k,\infty}} \lesssim B (1+t)^{-\frac{17}{8} + \frac{k+1}{4}}\quad \forall t\in [0,T].
		\eeq
		Let us prove inequality \eqref{est:psi-W3infty} in the case $k=3$ (the other cases are treated in a similar fashion and left to the reader.)
		By the Gagliardo-Nirenberg-Sobolev inequality,
		\[
		\|\psi\|_{W^{3,\infty}}\lesssim \|\psi\|_{L^2}^{1/5} \|\psi\|_{H^5}^{4/5} + \| \psi\|_{L^2}.
		\]
		By \eqref{hyp:bootstrap-theta'-NLBL}, $\|\psi\|_{L^2}\lesssim \|\p_x^5 \psi\|_{L^2} \lesssim B (1+t)^{-17/8}$, while
		\[
		\|\p_z^5 \psi\|_{L^2}\lesssim \|\p_z \p_x \theta'\|_{L^2}\lesssim \|\p_x^2 \theta'\|_{L^2}^{1/2} \|\p_z^2 \theta'\|_{L^2}^{1/2} \lesssim B (1+t)^{-7/8}.
		\]
		Estimate \eqref{est:psi-W3infty} follows.
		
  We also infer from \eqref{hyp:bootstrap-theta'-NLBL} some interpolated inequalities (which may be sub-optimal when compared to the bootstrap assumption on $\p_x^4 \Delta^2 \theta'$, depending on the values of $k,\ell$): for all $k,\ell\geq 0$ such that $k+ \ell \leq 8$, we have
		\beq\label{est:dxdztheta'-interpol}
		\ba
		\|\p_x^k \p_z^\ell \theta'\|_{L^2}&\leq \|\p_x^{k+\ell} \theta'\|_{L^2}^{\frac{k}{k+\ell}} \|\p_z^{k+\ell} \theta'\|_{L^2}^{\frac{\ell}{k+\ell}} \\
		&\lesssim B (1+t)^{-\frac{9}{8} + \frac{\ell}{4}} + B (1+t)^{\frac{k}{2} + \frac{\ell}{4} - 4 \frac{k}{k+\ell} - \frac{9}{8} \frac{\ell}{k+\ell}}.
		\ea
		\eeq

	\subsection{High regularity bounds under the bootstrap assumption}
	\label{sub:high-reg-NLBL}
		The purpose of this subsection is to prove the following estimates, which are the analogue of \cref{lem:H6} in higher regularity (see also \cref{rmk:higher-reg}):
		\begin{lemma}
			\label{lem-high-reg-NLBL}
			Let $\theta=\theta'+\bar \theta$ be a solution of \eqref{eq:STi}, and assume that $\theta_0\in H^{14}$ satisfies the assumptions of \cref{thm:BL-nonlinear}.
			Let $T>0$ be such that the bounds \eqref{hyp:bootstrap-theta'-NLBL} hold on $(0,T)$ for some constant $B\in (0,1)$. Assume furthermore that $\|\theta_0\|_{H^{14}}\leq B$.
			Then for all $t\in [0,T]$,
			\[
			\ba 
			\| \bar \theta(t)\|_{H^6}\lesssim B,\quad \| \bar \theta(t)\|_{H^9}\lesssim B (1+t)^{1/2}, \quad 
			\|\theta\|_{H^{14}} \lesssim B (1+t)^{5/2},\\
			\|\p_x^8 \Delta^2 \theta'\|_{L^2 }\lesssim B, \quad 
			\|\p_x^{10}\theta'\|_{L^2}\lesssim B (1+t)^{-1},\quad
			\|\p_x^{10} \psi\|_{L^2}\lesssim B (1+t)^{-2}.
			\ea
			\]

		\end{lemma}
		
		\begin{proof}
			First, recalling that
			\[
			\p_t \bar \theta= - \overline{\nabla^\bot \psi\cdot \nabla \theta'}
			\]
			and using the bootstrap assumptions \eqref{hyp:bootstrap-theta'-NLBL} and \eqref{est:psi-W3infty} together with the tame estimates \eqref{eq:tame}, we infer that
			\[
			\|\p_t \p_z^6 \bar \theta\|_{L^2}\lesssim B^2 (1+t)^{-5/4},
			\]
			and thus $\| \bar \theta(t)\|_{H^6} \lesssim \|\theta_0\|_{H^6} + B^2 \lesssim B$. A similar argument also shows that  $\| \bar \theta(t)\|_{H^7} \lesssim B + B^2 \ln(1+t)$.%
			
			Let us then compute the equation satisfied by $\partial^{10}\theta$, where $\p\in \{\p_x,\p_z\}$.  We have
			\[
			\p_t \partial^{10} \theta + \bu \cdot \nabla \partial^{10} \theta= \partial^{10}\p_x \psi - \left[\partial^{10}, \bu\cdot \nabla\right] \theta.
			\]
			Multiplying by $\partial^{10 } \theta$ and integrating by parts, we obtain
			\[
			\frac{\rd}{\rd t} \| \partial^{10 } \theta\|_{L^2} \leq 2 \| \partial^{10} \p_x \psi\|_{L^2} + 2 \left\| \left[\partial^{10}, \bu\cdot \nabla\right] \theta\right\|_{L^2}.
			\]
				Using the bootstrap assumptions \eqref{hyp:bootstrap-theta'-NLBL}, we have
			\[
			\| \partial^{10} \p_x \psi\|_{L^2} \lesssim B (1+t)^{3/8} \quad \forall t \in [0,T].
			\]
			As for the commutator term, using the tame estimates \eqref{eq:commu} together with the identity $\bu=\nabla^\bot \psi$, we obtain, for any $k\geq 5$,
\beq \label{est:commu-Dk}
\ba 
\left\| \left[  \p^k , \bu\cdot \nabla\right]\theta\right\|_{L^2} &\lesssim\|\psi\|_{W^{2,\infty}} \|\theta\|_{H^k} + \|\p_z \psi\|_{H^{k}} \|\p_x \theta\|_{\infty} + \|\p_x \psi\|_{H^{k}} \|\p_z \theta\|_{\infty}\\
&\lesssim B (1+t)^{-11/8} \|\theta\|_{H^k} + \|\p_x \theta\|_{H^{k-3}} \|\p_x \theta\|_{\infty}  + \|\p_x^2 \theta\|_{H^{k-4}} \|\p_z \theta\|_\infty.
\ea
\eeq			
		In particular, using the estimates \eqref{est:dxdztheta'-interpol}, we get, for $k=10$, 
\[
\left\| \left[  \p^{10} , \bu\cdot \nabla\right]\theta\right\|_{L^2} \lesssim  B (1+t)^{-11/8} \|\theta\|_{H^{10}} + B^2 (1+t)^{3/8}.
\]	
Assuming that $B<1$, we obtain
			\[
			\frac{\rd}{\rd t}  \| \theta\|_{H^{10}}\lesssim B (1+t)^{3/8} +  B (1+t)^{-11/8}  \| \theta\|_{H^{10}}.
			\]
			The Gronwall Lemma then ensures that
			\beq\label{est:H10}
			\| \theta(t)\|_{H^{10} } \lesssim \| \theta_0\|_{H^{10}} + B (1+t)^{11/8} \lesssim  B (1+t)^{11/8} .
			\eeq

			We then use the same strategy to estimate $\|\p_x^2 \Delta^4 \theta\|_{L^2}$. The linear term in the right-hand side is now
			\[
			\p_x^3 \Delta^4 \psi= \p_x^4 \Delta^2 \theta'=O(B) \quad \text{in }L^2.
			\]
			The only difference in the treatment of the commutator term lies in the bound of terms of the form $\p_z^2 \psi \p_x^3 \p_z^7 \theta'$. For those, we use our first estimate on $\| \theta\|_{H^{10}}$ \eqref{est:H10} together with the bootstrap assumptions \eqref{hyp:bootstrap-theta'-NLBL} (see in particular \eqref{est:psi-W3infty}, \eqref{est:dxdztheta'-interpol}), and we obtain
			\[
			\| \p_z^2 \psi \p_x^3 \p_z^7 \theta'\|_{L^2}\lesssim \| \p_z^2 \psi \|_\infty \|\theta'\|_{H^{10}} \lesssim B^2 (1+t)^{-\frac{11}{8}+ \frac{11}{8}} \lesssim B^2.
			\]
			It follows that
			\[
			\frac{\rd}{\rd t} \|\p_x^2 \Delta^4 \theta\|_{L^2}\lesssim B + B (1+t)^{-11/8} \|\p_x^2 \Delta^4 \theta\|_{L^2},
			\]
			and therefore $ \|\p_x^2 \Delta^4 \theta\|_{L^2}\lesssim B (1+t)$.
			The next step is to prove that $\sup_{t\in[0,T]} \|\p_x^6 \Delta^2 \theta'\|_{L^2}\lesssim B$.
			To that end, we check that $\p_x^6 \Delta^2 \theta'$ satisfies the assumptions of  \cref{prop:bb-uniformbound}.
			The source term  is $S=\bu\cdot \nabla \p_x^6\Delta^2 \theta'+ \left[ \p_x^6\Delta^2, \bu\cdot \nabla\right] \theta$.
			Classically, the first term is orthogonal to $\theta'$. It is therefore sufficient to bound the commutator. The terms involving $\bar \theta$ can be treated as perturbations of the dissipation term $\|\p_x^7 \Delta \theta'\|_{L^2}^2$, and therefore we focus on $\left[ \p_x^6\Delta^2, \bu\cdot \nabla\right] \theta'$. 
			First, note that
			\[
			\left\| ( \nabla^\bot\p_x^6 \Delta^2 \psi) \cdot \nabla \theta'\right\|_{L^2}\leq \|\nabla \theta'\|_\infty \| \nabla \p_x^7 \theta'\|_{L^2} \lesssim B (1+t)^{-3/4}\| \p_x^7 \Delta \theta'\|_{L^2}
			\]
			The other terms can be estimated thanks to the bootstrap assumptions together with the preliminary bounds on $\|\theta\|_{H^{10}}$ and  $ \|\p_x^2 \Delta^4 \theta\|_{L^2}$. We obtain
			\[
			\left\| \left[ \p_x^6\Delta^2, \bu\cdot \nabla\right] \theta \right\|_{L^2}\lesssim B (1+t)^{-1-\delta} \|\p_x^6 \Delta^2 \theta'\|_{L^2} + B^2(1+t)^{-1-\delta} +  B (1+t)^{-\frac{1}2 -\delta}\| \p_x^7 \Delta \theta'\|_{L^2}
			\]
			for some $\delta>0$. The details are left to the reader. Using a Cauchy-Schwarz inequality, it follows that
			\[
			\frac{\rd}{\rd t} \|\p_x^6 \Delta^2 \theta'\|_{L^2}^2 + c \| \p_x^7 \Delta \theta'\|_{L^2}^2 \lesssim B^2 (1+t)^{-1-\delta} +  B (1+t)^{-1-\delta} \|\p_x^6 \Delta^2 \theta'\|_{L^2}^2. 
			\]
			The Gronwall Lemma then implies that $\sup_{t\in [0,T]} \|\p_x^6 \Delta^2 \theta'\|_{L^2} \lesssim \| \theta_0\|_{H^{10}} + B^2 \lesssim B$.
			
			We then follow the same strategy to obtain bounds on $\|\theta\|_{H^{12}}$, $\|\p_x^4 \Delta^4 \theta\|_{L^2}$ and $\|\p_x^8 \Delta^2 \theta\|_{L^2}$.
			We have
			\[
			\frac{\rd}{\rd t}\|\theta\|_{H^{12}} \lesssim \|\p_x^2 \theta'\|_{H^8} + \left\|\left[\p^{12}, \bu \cdot \nabla\right] \theta\right\|_{L^2}.
			\]
			The first term in the right-hand side is  bounded by $B(1+t)$. 
			The commutator is estimated thanks to \eqref{est:commu-Dk} together with 
			the bootstrap assumptions and our preliminary bounds on derivatives up to order 10.
		We obtain $\|\theta(t)\|_{H^{12}}\lesssim B (1+t)^2$. We then write
			\[
			\p_t \p_x^4 \Delta^4 \theta' + \bu \cdot \nabla \p_x^4 \Delta^4 \theta'= \p_x^6 \Delta^2 \theta' - \left[\p_x^4 \Delta^4, \bu\cdot \nabla\right] \theta.
			\]
			The first term in the right-hand side is bounded by $CB$ in $L^2$. We then check that the nonlinear term can be treated perturbatively, using the bounds on $\theta'$ obtained so far, and we infer that $\| \p_x^4 \Delta^4 \theta'(t)\|_{L^2}\lesssim B(1+t)$. Once again, we then use \cref{prop:bb-uniformbound} in order to prove that $\|\p_x^8 \Delta^2 \theta'(t)\|_{L^2}\lesssim B $, and that
			\[
			\frac{\rd}{\rd t}\|\theta(t)\|_{H^{14}} \lesssim  B (1+t)^{-1-\delta}\|\theta(t)\|_{H^{14}} + B (1+t)^{3/2}.
			\]
			The computations are very similar to the ones above and left to the reader, and lead to the estimate of $\|\theta(t)\|_{H^{14}}.$
			
			The last step is to prove additional decay on $\|\p_x^{10} \theta'\|_{L^2}$ and $\| \p_x^{11} \psi\|_{L^2}$.
			Setting $S=-\p_x^{10} (\bu\cdot \nabla \theta')$, we can decompose $S$ into $S=S_{\para} + S_\bot + S_{\Delta}$, with $S_\bot=\bu\cdot \nabla \p_x^{10} \theta'$, and
            \[
            \ba
                S_\Delta&:=-\sum_{k\leq 6} \begin{pmatrix}
                    10\\k
                \end{pmatrix}\p_z \p_x^{10-k} \psi \p_x^{k+1} \theta + \sum_{k\leq 5} \begin{pmatrix}
                    10\\k
                \end{pmatrix} \p_x^{11-k}\psi \p_x^k\p_z\theta\\
                S_\para&:= -\sum_{7\leq k\leq 9} \begin{pmatrix}
                      10\\k
                \end{pmatrix}\p_z \p_x^{10-k} \psi \p_x^{k+1} \theta + \sum_{6\leq k\leq 9}  \begin{pmatrix}
                    10\\k
                \end{pmatrix} \p_x^{11-k}\psi \p_x^k\p_z\theta,
            \ea
            \]
			so that
            \[
			\|S_{\|}\|_{L^2}\lesssim B^2 (1+t)^{-2} + B (1+t)^{-1-\delta}\|\p_x^{10} \theta\|_{L^2},\quad \|S_{\Delta}\|_{L^2}\lesssim B (1+t)^{-1/2} \|\p_x^{10}\Delta \psi\|_{L^2}.
			\]
			Hence for $B$ sufficiently small, $S$ satisfies the assumptions of \cref{prop:bb-decay}, and we obtain
			\[
			\|\p_x^{10} \theta'(t)\|_{L^2} \lesssim B (1+t)^{-1}.
			\]
			Differentiating the equation on $\p_x^{9} \theta$ with respect to time, we get
			\[
			\p_t \p_t\p_x^9 \theta'
			=(1-G)\p_t \p_x^{10} \psi 
			- \p_t \p_x^9 (\bu\cdot \nabla) \theta' 
			- \p_t G \p_x^{10}\psi.
			\]
			Estimating the norm of each term in the right-hand side and using \cref{prop:bb-decay}, we obtain, for all $t\in [0, T]$,
			\[
			\|  \p_t\p_x^9 \theta'(t)\|_{L^2}\lesssim \frac{1}{(1+t)^2}\left( \sup_{s\in [0, T]} (1+s){\|\p_s \p_x^7 \Delta^2 \theta'(s)\|_{L^2}} + B^2 \right).
			\]
			Writing
			\[
			\p_t \p_x^7 \Delta^2 \theta'= \p_x^9 \theta' -\p_x^7 \Delta^2 (\bu\cdot \nabla \theta),
			\]
			we find that $\|\p_t \p_x^7 \Delta^2 \theta'\|_{L^2} \lesssim B (1+t)^{-1}$, and thus $\|\p_t \p_x^9 \theta'\|_{L^2}\lesssim B (1+t)^{-2}$. Going back to the equation on $\p_x^9 \theta'$, we find eventually that
			\[
			\p_x^{10} \psi= \p_t \p_x^9 \theta' + \p_x^9(\nabla^\bot \psi\cdot \nabla \theta)= O((1+t)^{-2})\quad \text{in } L^2.
			\]
			
    Finally, plugging these estimates into the equation on $\overline{\theta}$ leads to the desired bound on $\|\overline{\theta}\|_{H^9}$.
			
		\end{proof}
		
		Let us now prove a useful (albeit technical) result concerning the trace of $\p_z^3 \theta'$:
		\begin{corollary}\label{coro:trace-dz3theta'}
			Under the assumptions of \cref{lem-high-reg-NLBL}, for all $t\in [0,T]$,
			\[
			\|\p_z^3 \theta'\vert_{z=0}\|_{H^{33/4}(\T)}\lesssim B (1+t)^{-1/8}.
			\]
			The same estimate holds for the trace at $z=1$.
		\end{corollary}
		\begin{proof}
			Using Theorem 3.1 in Chapter 1 of \cite{LionsMagenes},
			\[
			\|\p_z^3 \theta'\vert_{z=0}\|_{H^s(\T)}\lesssim \|\theta'\|_{H^\beta_x L^2_z}^{1/8} \|\p_z^4 \theta'\|_{H^\gamma_x L^2_z}^{7/8},
			\]
			where $\frac{1}{8} \beta + \frac{7}{8}\gamma=s$. Taking $\beta=10$ and $\gamma=8$ and using the bounds of \cref{lem-high-reg-NLBL}, we obtain the desired result.
		\end{proof}

		\subsection{Decomposition of the traces of \texorpdfstring{$\Delta^2 \theta'$ and $\p_\bn \Delta^2 \theta'$}{Δ² θ' and ∂\_z Δ² θ'}}
		\label{sub:decomp-gamma}
		
		The role of the boundary layer is to lift the traces of $\Delta^2 \theta'$ and $\p_\bn \Delta^2 \theta'$ on the boundary.
		Therefore we first need to prove that these traces converge towards a (generically non trivial) limit.
	In fact, we will even need to have a rather precise asymptotic expansion of the traces in powers of $(1+t)^{-1/4}$.
	This is the main purpose of this section.

		The first result of this section concerns the long time behavior of $\Delta^2 \theta'\vert_{\p\Om}$ and $\p_\bn \Delta^2 \theta'\vert_{\p\Om}:$
		\begin{lemma}[Long-time behavior of $\p_z^k \Delta^2 \theta'\vert_{\p\Om}$ and of $\p_z^{2+k} G(t)\vert_{\p\Om}$, $k=0,1$]
			For $k=0,1$, let
			\begin{equation*}\label{def:gamma-j-a}
			\gamma^k_\mtop(t,x):=\p_z^k\Delta^2 \theta'(t,x, z=1),\quad \gamma^k_\mbot(t,x):=\p_z^k\Delta^2 \theta'(t,x, z=0).
			\end{equation*}
			Assume that $\theta_0\in H^{14}(\Om)$ and 
			$\theta_0=\p_\bn \theta_0=0$ on $\p\Om$, $\p_z^2 \bar\theta_0=0$ on $\p\Om$.
			Let $T>0$, $B \in(0,1)$ such that the bootstrap assumptions \eqref{hyp:bootstrap-theta'-NLBL} hold on $[0,T]$. Assume furthermore that $\|\theta_0\|_{H^{14}}\leq B$.
			Then there exist universal constants $B_0,\delta>0$ and functions $\gamma^0_{a,T}\in H^9(\T)$, $\gamma^1_{a,T} \in H^8(\T)$,%
			such that if $B\leq B_0$,
			\[
			\ba
			\| \gamma^0_{a,T} \|_{H^9(\T)}\lesssim \|\theta_0\|_{H^{14}} + B^2 \quad \text{and} \quad \| \gamma^0_{a}(t) - \gamma^0_{a,T} \|_{H^9(\T)}\lesssim B^2 \frac{1}{(1+t)^\delta}\quad \forall t\in [0,T],\\
			\| \gamma^1_{a,T} \|_{H^8(\T)}\lesssim \|\theta_0\|_{H^{14}} + B^2 \quad \text{and} \quad \| \gamma^{{1}}_{a}(t) - \gamma^{{1}}_{a,T} \|_{H^8(\T)}\lesssim B^2 \frac{1}{(1+t)^\delta}\quad \forall t\in [0,T].
			\ea
			\]

			In a similar fashion, for $k=2,3$, $a\in \{\mtop,\mbot\}$, there exists $g^k_{a,T}\in \R$ 
				such that
				\[
				\ba 
					|g^k_{a,T}| \lesssim \|\theta_0\|_{H^{14}} + B^2\quad \text{for }k\in \{2, 3\},\\
					\left| g^2_{\mbot,T} - \p_z^2 G(t,0)\right|\lesssim \frac{B^2}{(1+t)^{3/4}}\quad\forall t\in [0,T),\\
						\left| g^3_{\mbot,T} - \p_z^3 G(t,0)\right|\lesssim \frac{B^2}{(1+t)^{1/2}}\quad\forall t\in [0,T).
				\ea
				\]
				The same estimates hold for $g^k_{\mtop, T} - \p_z^k G(t,1)$.
			\label{lem:CV-gamma-G}
		\end{lemma}
		The proof of \cref{lem:CV-gamma-G} is postponed to the end of this section.
		
		The second intermediate result of this section pushes further the decomposition of $\gamma^k_a(t)$. It holds under additional structural assumptions on $\thbl$ and $\threm=\theta'-\thbl$. More precisely, let us assume that there exist profiles $\Theta^j_a,\Psi^j_a$, %
  such that
		\beq\label{def:thbl}
		\ba 
		\thbl(t,x,z)=\sum_{j=0}^{4}(1+t)^{-1-\frac{j}{4}} \left( \Theta^j_\mbot(x,(1+t)^{1/4} z)  + \Theta^j_\mtop(x,(1+t)^{1/4} (1-z))\right) %
  \\
		\psbl(t,x,z)=\sum_{j=0}^{4}(1+t)^{-2-\frac{j}{4}} \left( \Psi^j_\mbot(x,(1+t)^{1/4} z)  + \Psi^j_\mtop(x,(1+t)^{1/4} (1-z))\right) %
		\ea
		\eeq
		where there exists a universal constant $c>0$ such that for all $Z_0\geq 0$,
		\beq
		\label{bound-Thetaj-thetac}
		\ba 
		\|\Theta^j_a\|_{H^9(\T\times (Z_0, +\infty))} +\|\Psi^j_a\|_{H^{11}(\T\times (Z_0, +\infty))}\lesssim \left(\|\theta_0\|_{H^{14}} + B^2\right) \exp(- c Z_0^{4/5}).
		\ea
		\eeq
		In the course of the proof, we shall also need the following assumption:
		\beq\label{hyp:traces-Thetaj}
		\p_Z^2 \Theta^j_a\vert_{Z=0}\in H^7(\T),\quad \p_Z^3 \Theta^j_a\vert_{Z=0}\in H^{15/2}(\T).
		\eeq
		\begin{remark}
The profiles $\Theta^j_a$, $\Psi^j_a$ are not the same as the ones of \cref{sec:linear-BL}. However we kept the same notation for convenience.
		\end{remark}
    \begin{definition}[Definition of $\gamma^{j,k}_a$]\label{def:gamma-jak}
		Let $\Theta^j_a, \Psi^j_a$ be the boundary layer profiles from \eqref{def:thbl}, with $a\in \{\mtop,\mbot\}$, $j\in \{0,\cdots, 4\}$.
		
		 Let $\eta_\mbot=1$ and $\eta_\mtop=-1$.
		We then define the following coefficients:
		\beq\label{def:gamma-j-al}
		\ba 
		\gamma^{0,2}_a= \;& 12  g^2_{a, T} \p_x \p_Z^2 \Psi^0_{a}|_{Z=0},\\
		\gamma^{0,3}_a=\;&8   g^2_{a, T} \p_x \p_Z^2 \Psi^1_{a}|_{Z=0}  - \frac{4\eta_a}{3} \left[\p_Z^4 \left\{ \Psi^0_a, \Theta^0_a\right\}_{x,Z}\right]'_{\vert Z=0},\\
		\gamma^{1,1}_a=\;&40 \eta_a g^2_{a, T} \p_x \p_Z^3 \Psi^0_{a}|_{Z=0} \\
		\gamma^{1,2}_a=\;&20 \left(  \eta_a g^2_{a, T} \p_x \p_Z^3 \Psi^1_{a}|_{Z=0} + g^{3}_{a,T} \p_x \p_Z^2 \Psi^0_{a}|_{Z=0}\right)\\
		&+2 \left[\p_Z^5 \left\{ \Psi^0_a, \Theta^0_a\right\}_{x,Z}\right]'_{\vert Z=0}.
		\ea
		\eeq
		where $\{\cdot,\cdot\}_{x,Z}$ denotes the Poisson bracket
		\begin{equation*}
			\{f,g\}_{x,Z} = \p_xf\p_ZG - \p_Zf\p_xg.
		\end{equation*}
		
		\end{definition}
\begin{remark}
The bounds \eqref{bound-Thetaj-thetac} ensure that
$\| \gamma^{j,k}_a\|_{H^5(\T)} \lesssim B^2 $.
We shall actually derive stronger regularity estimates in the course of the proof, as we construct explicitly the profiles $\Theta^j_a$ and $\Psi^j_a$.
\end{remark}

		\begin{lemma}[Decomposition of $\gamma^k_a$]
			Assume that $\theta_0\in H^{14}(\Om)$ and 
			$\theta_0=\p_\bn \theta_0=0$ on $\p\Om$, $\p_z^2 \bar{\theta}_0=0$ on $\p\Om$.
			Let $T>0$, $B\in (0,1)$ such that
				$ \| \theta_0\|_{H^{14}(\Om)}\leq B$ and  such that the bootstrap assumptions \eqref{hyp:bootstrap-theta'-NLBL} hold on $[0,T]$. 
    
    Assume furthermore that there exist profiles $\Theta^j_a,\Psi^j_a$ satisfying \eqref{bound-Thetaj-thetac} and \eqref{hyp:traces-Thetaj} such that $\threm=\theta'-\thbl$ satisfies 
    \eqref{hyp:bootstrap-thiint}, where $\thbl$ is defined by \eqref{def:thbl}.
			Define the coefficients $\gamma^{j,k}_a$ by \eqref{def:gamma-j-al}.
			
			Then for $j=0,1$, $a\in\{ \mtop,\mbot\} $, there exists %
			$\Gamma^j_{a,T}\in W^{1,\infty}((0,T); L^2(\T))$ such that for all $t\in [0,T]$,
			\[\ba
			\gamma^0_a(t)={}& \gamma^0_{a,T} + \gamma^{0,2}_a(1+t)^{-1/2} + \gamma^{0,3}_a(1+t)^{-3/4}\\&+ \Gamma^0_{a,T}(t) -\gamma^{0,2}_a(1+T)^{-1/2} - \gamma^{0,3}_a(1+T)^{-3/4} ,\\
			\gamma^1_a(t)={}&\gamma^1_{a,T} + \gamma^{1,1}_a(1+ t)^{-1/4} + \gamma^{1,2}_a(1+t)^{-1/2}\\&+ \Gamma^1_{a,T}(t) - \gamma^{1,1}_a(1+ T)^{-1/4} + \gamma^{1,2}_a(1+T)^{-1/2}.
			\ea
			\]
			where 
			for all $t\in [0,T]$, for $j=0,1$, $\ell=0,1,2$,
			\[\ba
			\left\| \p_t^\ell\Gamma^j_{a,T}(t)\right\|_{L^2(\T)} \lesssim B^2 (1+t)^{-1-\ell+ \frac{j}{4}},\quad
			\left\| \Gamma^j_{a,T}(t)\right\|_{H^4(\T)} \lesssim B^2 (1+t)^{-\frac{23}{24}+ \frac{j}{4}}.
			\ea
			\]
		
			\label{lem:decomp-gamma}
		\end{lemma}

		Let us now prove \cref{lem:CV-gamma-G,lem:decomp-gamma}.
		
		\begin{proof}[Proof of \cref{lem:CV-gamma-G}]
			We have
			\begin{equation}
			\label{eq:Delta2-theta}
            \ba
                \frac{\p}{\p t} \Delta^2 \theta'={}& (1-  G) \p_x^2 \theta' -4 \p_z G \p_z \p_x^3 \psi -2 \p_z^2 G \p_x^3 \psi - \sum_{k=1}^4 \begin{pmatrix}
                   4\\ k
                \end{pmatrix} \p_z^k G \p_x \p_z^{4-k} \psi \\
                &+\Delta^2 (\p_z \psi \p_x \theta')' -\Delta^2 (\p_x \psi \p_z \theta')'.
            \ea
			\end{equation}
			We now take the trace of the above equation at $z=0$, recalling that $G\vert_{z=0}=\p_z G\vert_{z=0}=0$ (see \cref{lem:traces}), $\psi\vert_{z=0}= \p_z\psi\vert_{z=0} =0$, and $\theta'\vert_{z=0}=\p_z \theta'\vert_{z=0}=0$.
			We obtain
			\beq 
			\ba 
			\label{eq:gamma_0}
			\frac{\rd}{\rd t}\gamma^0_\mbot={}&-6 \p_z^2 G\vert_{z=0} \p_x \p_z^2 \psi\vert_{z=0} \\
			&+ 6 \left(\p_z^3 \psi\vert_{z=0} \p_x \p_z^2 \theta'\vert_{z=0}\right)' + 4 \left(\p_z^2 \psi\vert_{z=0} \p_x \p_z^3 \theta'\vert_{z=0}\right)'\\
			&-6 \left( \p_x \p_z^2 \psi\vert_{z=0} \p_z^3 \theta'\vert_{z=0}\right)' - 4 \left( \p_x \p_z^3 \psi\vert_{z=0} \p_z^2 \theta'\vert_{z=0}\right)'.
			\ea
			\eeq
			We then estimate each term in the right-hand side using \cref{lem-high-reg-NLBL}. Note that $\p_z^2 G|_{z=0}$ is bounded in $L^\infty(\R_+ \times (0,1))$.
			We focus on the first term, which has the smallest decay. Using Theorem 3.1 in Chapter 1 of \cite{LionsMagenes}, we infer that for any $s>0$,
			\[
			\|\p_x \p_z^2 \psi\vert_{z=0}\|_{H^s}\leq  \|\p_z^2  \psi\vert_{z=0}\|_{H^{s+1}} \lesssim \|\psi\|_{H^\beta_x L^2_z}^{3/8} \|\p_z^4 \psi \|_{H^\gamma_xL^2_z }^{5/8},
			\]
			where $\beta,\gamma$ are such that $\frac{3}{8}\beta + \frac{5}{8}\gamma=s+1$. 
			Let us choose $\beta=\gamma=10$, $s=9$. According to \cref{lem-high-reg-NLBL}, $\|\psi\|_{H^{10}_x L^2_z}\lesssim B (1+t)^{-2}$.  
			As for the other term, using the short-hand notation from \cref{sec:stab},
				\[
					\|\p_z^4 \psi \|_{H^{10}_x L^2_z}\lesssim \|\theta\|_{H^{11}_xL^2_z} \lesssim \underbrace{\|\theta\|_{H^{10}_x L^2_z}^{1/2} \|\theta\|_{H^{12}_x L^2_z}^{1/2} }_{\frac{1}{2}\times 1 + \frac{1}{2}\times 0}\lesssim B (1+t)^{-1/2}.
			\]
		Hence
		\[
		\|\p_x \p_z^2 \psi\vert_{z=0}\|_{H^9}\lesssim B (1+t)^{-17/16}.
		\]

			The quadratic terms, involving traces of derivatives of $\psi$ and of $\theta'$, have a higher decay. Let us estimate for instance $\p_z^3 \psi \p_x \p_z^2 \theta'$ at $z=0$. We have, for any $s>1/2$,
			\[\ba
			\|\p_z^3 \psi\vert_{z=0} \p_x \p_z^2 \theta'\vert_{z=0}\|_{H^s(\T)}\lesssim{}
            &\|\p_z^3 \psi\vert_{z=0}\|_{L^\infty(\T)} \|\p_x \p_z^2 \theta'\vert_{z=0}\|_{H^s(\T)} \\
            &+ \|\p_z^3 \psi\vert_{z=0}\|_{H^s(\T)} \|\p_x \p_z^2 \theta'\vert_{z=0}\|_{L^\infty(\T)}  .
			\ea\]
			Using once again Theorem 3.1 in Chapter 1 of \cite{LionsMagenes}, we find that
			\[
			\|\p_x \p_z^2 \theta'\vert_{z=0}\|_{L^\infty(\T)} \lesssim \|\theta'\|_{H^4_x L^2_z}^{11/16} \|\p_z^8 \theta'\|_{L^2_x L^2_z}^{5/16}\lesssim B (1+t)^{-1/2},
			\]
			and we recall that $\|\psi\|_{W^{3,\infty}}\lesssim B (1+t)^{-9/8}$ (see \eqref{est:psi-W3infty}). The same arguments together with \cref{lem-high-reg-NLBL} also imply
			\begin{align*}
				\|\p_x \p_z^2 \theta'\vert_{z=0}\|_{H^{19/2}(\T)} &\lesssim \|\theta'\|_{H^{12}_x L^2_z}^{3/8} \|\p_z^4 \theta'\|_{H^8_x L^2_z}^{5/8}\\
				&\lesssim \|\p_x^8 \Delta^2 \theta'\|_{L^2} \lesssim B,
			\end{align*}
			while
			\beq\label{est:trace-dz3psi}
			\|\p_z^3 \psi\vert_{z=0}\|_{H^s(\T)}\lesssim \|\psi\|_{L^2_z H^\beta_x}^{1/8} \| \p_z^4 \psi \|_{L^2_z H^\gamma_x }^{7/8} \lesssim \|\psi\|_{L^2_z H^\beta_x}^{1/8} \|  \theta' \|_{L^2_z H^{\gamma+1}_x }^{7/8},
			\eeq
			with $\frac{1}{8}\beta + \frac{7}{8}\gamma=s$. Taking $\beta=10$ and $\gamma=9$, we obtain, for some $s>9$,
			\begin{equation*}
			\|\p_z^3 \psi\vert_{z=0} \p_x \p_z^2 \theta'\vert_{z=0}\|_{H^s(\T)}\lesssim B^2 (1+t)^{-9/8}.
			\end{equation*}
			The other terms are treated in a similar fashion.
			We infer that there exists $\delta>0$ such that
			\[
		\left|	\frac{\ud}{\ud t} \|\gamma^0_a(t)\|_{H^9(\T)}\right| \lesssim \frac{B^2}{(1+t)^{1+\delta}}\quad \forall t \in (0,T).
			\]
		
			This completes the proof of the estimate on $\gamma^0_a$.
			
			The estimate for $\gamma^1_a$ follows from a similar argument. 
			Taking the vertical derivative of \eqref{eq:Delta2-theta}, we have
			\begin{equation*}
			\label{eq:d_z-Delta2-theta}
            \ba
                \frac{\p}{\p t} \p_z \Delta^2 \theta'={}& (1-  G) \p_z \p_x^2 \theta' -\sum_{k=1}^3  \begin{pmatrix}
    					3\\ k
				\end{pmatrix}\p_z^k G \p_z^{3-k} \p_x^3 \psi - \sum_{k=1}^5 \begin{pmatrix}
					5\\ k
				\end{pmatrix} \p_z^k G \p_x \p_z^{5-k} \psi \\
				&+\p_z\Delta^2 (\p_z \psi \p_x \theta')' -\Delta^2 (\p_x \psi \p_z \theta')'.
			\ea \end{equation*}
			Taking the trace of the above equation at $z=0$, we obtain
			\begin{equation*} \label{eq:gamma_1}
			\ba 
			\frac{\rd}{\rd t}\gamma^1_\mbot={}&-10 \p_z^2 G\vert_{z=0} \p_x \p_z^3 \psi\vert_{z=0} -10 \p_z^3 G\vert_{z=0} \p_x \p_z^2 \psi\vert_{z=0} \\
			&+ 6 \p_x^2 \left(\p_z^2 \psi\vert_{z=0} \p_x \p_z^2 \theta'\vert_{z=0} - \p_x \p_z^2 \psi\vert_{z=0} \p_z^2 \theta'\vert_{z=0} \right)\\
			&+ 10 \left(\p_z^3 \psi\vert_{z=0} \p_x \p_z^3 \theta'_{z=0} + \p_z^4 \psi_{z=0} \p_x \p_z^2 \theta'\vert_{z=0}\right)' + 5 \left(\p_z^2 \psi\vert_{z=0} \p_x \p_z^4 \theta'_{z=0}\right)'\\
			&-10 \left( \p_x \p_z^3 \psi\vert_{z=0} \p_z^3 \theta'\vert_{z=0} + \p_x \p_z^4 \psi_{z=0}  \p_z^2 \theta'\vert_{z=0}\right)' - 5 \left( \p_x \p_z^2 \psi\vert_{z=0} \p_z^4 \theta'\vert_{z=0}\right)'.
			\ea
			\end{equation*}
			The highest order term is the first one. We recall that $\p_z^2G$ and $\p_z^3 G$ are uniformly bounded by $C B$ in $L^\infty$, and that the trace of $\p_z^3 \psi$ in $H^9$ is evaluated thanks to \eqref{est:trace-dz3psi}. 
Once again, the quadratic terms have a higher decay and can be handled as perturbations.
			The trace of $\p_z^4 \theta'\vert_{z=0}$ can estimated thanks to $\gamma^0_\mbot.$ We then obtain, for some $\delta>0$,
				\[
				\left|	\frac{\ud}{\ud t} \|\gamma^1_a(t)\|_{H^8(\T)}\right| \lesssim \frac{B^2}{(1+t)^{1+\delta}}\quad \forall t \in (0,T).
				\]
			and the desired estimate for $\gamma^1_a$ follows.

			Let us now address the convergence of $\p_z^k G(t)\vert_{\p\Om}$ as $t\to \infty$. We recall that
			\[
			\ba
			\p_t \p_z^k G(t, z=0)= &- \p_z^{k+1} \overline{\bu \cdot \nabla \theta'|_{z=0}}\\
			=& \p_z^{k+1} \overline{\p_z\psi\p_x \theta' - \p_x \psi \p_z \theta'\vert_{z=0}}.
			\ea
			\]
			Since  $\psi\vert_{z=0}= \p_\bn\psi\vert_{z=0} =0$, and $\theta'\vert_{z=0}=\p_\bn \theta'\vert_{z=0}=0$, we have
			\[
			\p_t \p_z^2 G(t,z=0)= \overline{3\p_z^2 \psi\vert_{z=0} \p_x\p_z^2 \theta'\vert_{z=0} - \p_x\p_z^2 \psi\vert_{z=0} \p_z^2 \theta'\vert_{z=0}}.
			\]
			As above,
			\[
			\ba
			\| \p_z^2 \psi\vert_{z=0} \p_x\p_z^2 \theta'\vert_{z=0}\|_{L^1(\T)}\leq &\| \p_z^2 \psi\vert_{z=0}\|_{L^2(\T)} \|  \p_x\p_z^2 \theta'\vert_{z=0}\|_{L^2(\T)}\\
			\lesssim& \| \p_z^2 \psi\|_{H^{1/2}(\Om)} \|  \p_x\p_z^2 \theta'\|_{H^{1/2}(\Om)}\\
			\lesssim & B^2 (1+t)^{-7/4}.
			\ea
			\]
			The estimate on $\p_t \p_z^3 G(t,z=0)$ is similar and left to the reader.
			
		\end{proof}
		
		We now turn towards the decomposition of $\gamma^0_a$ and $\gamma^1_a$ for $a\in \{\mtop,\mbot\}$:
		
		\begin{proof}[Proof of \cref{lem:decomp-gamma}]
			We focus on $a=\mbot$ by symmetry, and we start with the decomposition of $\gamma^0_\mbot$.
			Let us go back to \eqref{eq:gamma_0}. The main term in the right-hand side is $-6\p_z^2 G\vert_{z=0} \p_x \p_z^2 \psi\vert_{z=0}$. Following \cref{lem:CV-gamma-G} and using the decomposition $\theta'=\thbl + \threm$, we write
			\begin{equation*} \ba
				\p_z^2 G\vert_{z=0} \p_x \p_z^2 \psi\vert_{z=0}={}&(1+t)^{-3/2} g^2_{\mbot, T} \p_x \p_Z^2 \Psi^0_{\mbot|Z=0} + (1+t)^{-7/4} g^2_{\mbot, T} \p_x \p_Z^2 \Psi^1_{\mbot|Z=0}\\
				&+ \sum_{j\geq 2} (1+t)^{-\frac{3}{2} -\frac{j}{4}} \p_z^2 G\vert_{z=0} \p_x \p_Z^2 \Psi^j_{\mbot|Z=0}\\
				&+  \sum_{j=0,1} (1+t)^{-\frac{3}{2} -\frac{j}{4}} \left(\p_z^2 G\vert_{z=0} -g^2_{\mbot, T}\right) \p_x \p_Z^2 \Psi^j_{\mbot|Z=0}\\
				&+ \p_z^2 G\vert_{z=0} \p_x \p_z^2 \psrem\vert_{z=0} + O(B^2 \exp(-c (1+t)^{1/5}),
			\ea \end{equation*}
   where the exponentially small term comes from the traces of derivatives of $\Psi^j_\mtop$ evaluated at $Z=(1+t)^{1/4}$.
			The assumptions of the Lemma and the bootstrap inequalities  \eqref{hyp:bootstrap-thiint} ensure that for all $t\in [0,T]$,
			\[
			\ba 
			\left\| \sum_{j\geq 2} (1+t)^{-\frac{3}{2} -\frac{j}{4}} \p_z^2 G\vert_{z=0} \p_x \p_Z^2 \Psi^j_{\mbot|Z=0}\right\|_{L^2(\T)}\lesssim B^2 (1+t)^{-2} ,\\
			\left\|  \p_z^2 G\vert_{z=0} \p_x \p_z^2 \psrem\vert_{z=0}\right\|_{L^2(\T)} \lesssim \| \p_z^2 G\|_\infty \|\p_x \threm\|_{L^2} \lesssim B^2 (1+t)^{-2}.
			\ea
			\]
			Furthermore, \cref{lem:CV-gamma-G} ensures that for all $t\in (0,T)$
			\[
			\left| \p_z^2 G\vert_{z=0} -g^2_{\mbot, T}\right| \lesssim B^2 (1+t)^{-3/4},
			\]
			and therefore
			\[
			\left\|  \sum_{j=0,1} (1+t)^{-\frac{3}{2} -\frac{j}{4}} \left(\p_z^2 G\vert_{z=0} -g^2_{\mbot, T}\right) \p_x \p_Z^2 \Psi^j_{\mbot|Z=0}\right\|_{L^2(\T)}\lesssim B^3 (1+t)^{-9/4}.
			\]
			We now address the  quadratic terms in \eqref{eq:gamma_0}, namely
			\[
			\mathcal B(\psi, \theta'):=6 \left\{\p_z^2 \psi, \p_z^2 \theta'\right\}'\vert_{z=0} + 4(\p_z^2 \psi \p_x \p_z^3 \theta' - \p_x \p_z^3 \psi \p_z^2 \theta')'\vert_{z=0}.
			\]
			Decomposing $\psi$ and $\theta'$ into their boundary layer and their remainder part, we find that  the main order quadratic term is
			\[\ba
			&(1+t)^{-\frac{7}{4}} \left[6\p_Z^3 \Psi^0_\mbot \p_x \p_Z^2 \Theta^0_\mbot + 4 \p_Z^2\Psi^0_\mbot \p_x \p_Z^3 \Theta^0_\mbot\right. \\
   &\qquad\qquad\qquad\left.- 6 \p_x \p_Z^2\Psi^0_\mbot  \p_Z^3 \Theta^0_\mbot - 4  \p_x \p_Z^3 \Psi^0_\mbot \p_Z^2 \Theta^0_\mbot\right]'\vert_{z=0}\\
			=:{}&(1+t)^{-7/4} \gamma^0_{\mbot, \mathrm{NL}},\ea
			\]
			while all the other terms are bounded in $L^2(\T)$ by $
			C B^2 (1+t)^{-2}.$
		      
        Defining $\gamma^{0,j}_a$ by \eqref{def:gamma-j-al}, we find that
				\[\ba 
				\p_t \left( \gamma^{0,2}_\mbot (1+t)^{-1/2}\right) &= -6 (1+t)^{-3/2} g^2_{\mbot, T} \p_x \p_Z^2 \Psi^0_{\mbot|Z=0} ,\\
				\p_t  \left( \gamma^{0,3}_\mbot (1+t)^{-3/4}\right) &= -6 (1+t)^{-7/4} g^2_{\mbot, T} \p_x \p_Z^2 \Psi^1_{\mbot|Z=0}+ (1+t)^{-7/4} \gamma^0_{\mbot, \mathrm{NL}},
				\ea
				\]
	where we recognize the main terms in $-6\p_z^2 G\vert_{z=0} \p_x \p_z^2 \psi\vert_{z=0}$.		
			Now, define $\Gamma^0_{\mbot, T}$ by
			\begin{equation*} \ba
				\Gamma^0_{\mbot, T}(t)={}&6\int_t^T \sum_{j\geq 2} (1+s)^{-\frac{3}{2} -\frac{j}{4}} \p_z^2G(s)\vert_{z=0} \p_x \p_Z^2 \Psi^j_{\mbot|Z=0}\ud s\\
				&+6 \int_t^T  \sum_{j=0,1} (1+s)^{-\frac{3}{2} -\frac{j}{4}} \left(\p_z^2 G(s)\vert_{z=0} -g^2_{\mbot, T}\right) \p_x \p_Z^2 \Psi^j_{\mbot|Z=0}\ud s\\
				&+  \int_t^T \mathcal B \left(\sum_{j=1}^4 (1+s)^{-2 - \frac{j}{4}} \Psi^j_{\mbot}(x, Z_\mbot) + \psrem, \theta'(s)\right)  \ud s\\
				&+  \int_t^T \mathcal B \left((1+s)^{-2 } \Psi^0_{\mbot}(x, Z_\mbot) , \sum_{j=1}^4 (1+s)^{-1 - \frac{j}{4}} \Theta^j_{\mbot}(x, Z_\mbot) + \threm\right)  \ud s\\
    &+ O(B^2\exp(-c (1+t)^{1/4})).
			\ea \end{equation*}
   The last --- exponentially small --- term comes once again from the trace of $\Psi^j_\mtop$ at the lower boundary, \ie at $Z=(1+t)^{1/4}$. We do not write its full expression  for the sake of readability.

			Note that the assumptions \eqref{hyp:bootstrap-thiint} on $\threm$ ensure that
			\[
			\|\p_z^3 \threm\vert_{z=0}\|_{L^2} \lesssim \|\threm\|_{L^2}^{9/16} \|\p_z^8 \threm\|_{L^2}^{7/16} \lesssim B (1+t)^{-9/8}.
			\]
			Recalling \cref{coro:trace-dz3theta'} and using the assumption $\p_Z^3 \Theta^j_{a}\vert_{Z=0} \in H^{15/2}(\T)$ (see \eqref{hyp:traces-Thetaj}), we also infer that
			\[
			\|\p_z^3 \threm\vert_{z=0}\|_{H^{15/2}}\lesssim B (1+t)^{-1/8}.
			\]
			Interpolating between these two estimates, we find in particular that
			\[
			\|\p_z^3 \threm\vert_{z=0}\|_{H^5} \lesssim B (1+t)^{-11/24}.
			\]
			The above estimates ensure that for $k=0,1$
			\[\ba 
			\left\| \p_t^k \Gamma^0_{\mbot, T}(t)\right\|_{L^2(\T)}\lesssim B^2(1+t)^{-k-1},\\
			\left\| \Gamma^0_{\mbot, T}(t)\right\|_{H^4(\T)}\lesssim B^2 (1+t)^{-23/24}.
			\ea
			\]
			Therefore we obtain the decomposition announced in the Lemma for $\gamma^0_a$.

			The decomposition of $\gamma^1_a$ follows from similar arguments and is left to the reader.
		\end{proof}

	\subsection{Iterative construction of the boundary layer profile}
	\label{sub:NLBL-construction}
		Let us now turn towards the construction of the boundary layer profile, and more generally, of an approximate solution.
		The purpose of this subsection is to prove the two following Lemmas. Our first result, which is truly the core of the construction,
		is valid under the bootstrap assumption \eqref{hyp:bootstrap-theta'-NLBL} on $\theta'$:
		
		\begin{lemma}
			Let $\theta_0\in H^{14}(\Om)$ such that $\|\theta_0\|_{H^{14}}\leq B <1$, and 
			$\theta_0\vert_{\p \Om}=\p_\bn \theta_0\vert_{\p \Om}=0$, $\p_z^2 \bar\theta_0\vert_{\p \Om}=0$.
			
			 Let $\theta=\bar \theta+\theta' $ be a solution of \eqref{eq:STi}, and assume that the bounds \eqref{hyp:bootstrap-theta'-NLBL} hold on $(0,T)$.
			Let $\gamma^0_{a,T}$, $\gamma^1_{a,T}$ be defined by \cref{lem:CV-gamma-G}.

			Then there exist profiles $\Theta^j_a\in H^8(\T\times \R_+)$, $\Psi^j_a\in H^9(\T\times \R_+)$, $j\in \{0,\cdots, 4\}$ and  a corrector $\theta_c \in H^9(\Om)$, depending only on $\gamma^0_{a,T}, \gamma^1_{a,T}, g^2_{a,T}$ and $g^3_{a,T}$, such that, defining $\thbl$ by \eqref{def:thbl}, $\gamma^j_{a,T}$ by \cref{lem:CV-gamma-G} and $\gamma^{j,k}_a$ by \eqref{def:gamma-j-al}, the following properties hold:
			\begin{enumerate}
				\item Bounds on the profiles: $\Theta^j_a$, $\Psi^j_a$ satisfy \eqref{bound-Thetaj-thetac};

\item Bound on the corrector: setting $\psi_c=\Delta^{-2}\p_x \theta_c$,
\[\ba 
\sup_{t\in [0,T]}\left(\|\p_x \Delta^4 \theta_c(t)\|_{L^2} + (1+t)^2 \|\p_x^5 \theta_c(t)\|_{L^2} \right)\lesssim \|\theta_0\|_{H^{14}} + B^2,\\
\sup_{t\in [0,T]}\left((1+t)^3\left(\|\p_t \p_x^4 \theta_c (t)\| + \|\p_x^5 \psi_c(t)\|_{L^2}\right) \right)\lesssim \|\theta_0\|_{H^{14}} + B^2
.\ea\]
    
				\item Traces at the top and bottom: at $z=0$,
				\begin{align}
                  \begin{split}
                    \Delta^2 (\thbl+ \theta_c)\vert_{z=0} ={}
                    & \gamma^0_{\mbot, T} + \gamma^{0,2}_{\mbot} (1+t)^{-1/2} + \gamma^{0,3}_{\mbot} (1+t)^{-3/4}\\
                    & - \gamma^{0,2}_{\mbot} (1+T)^{-1/2} - \gamma^{0,3}_{\mbot} (1+T)^{-3/4},
                  \end{split} \label{trace-Delta2-thbl}\\
                  \begin{split}
                  \p_z \Delta^2 (\thbl+\theta_c)\vert_{z=0} ={}
                    & \gamma^1_{\mbot, T} + \gamma^{1,1}_{\mbot,} (1+t)^{-1/4} + \gamma^{1,2}_{\mbot} (1+t)^{-1/2}\\
                    & - {\gamma^{1,1}_\mbot}(1+T)^{-1/4} - {\gamma^{1,2}_{\mbot}} (1+T)^{-1/2} .
                  \end{split} \label{trace-dzDelta2-thbl}
				\end{align}
				Similar formulas hold at $z=1$.
				
				\item Evolution equation: $\thbl+\theta_c$ satisfies
				\[
    \ba
				\p_t (\thbl+\theta_c) =& (1-G) \p_x^2 \Delta^{-2} (\thbl+\theta_c) \\
   & - \left(\nabla^\bot \Delta^{-2} \p_x (\thbl+\theta_c) \cdot \nabla (\thbl+\theta_c)\right)' + R^\BL,\ea
				\]
				and the remainder $R^\BL$ is such that for  $\ell=0,1$, for all $t\in [0,T]$,
				\[
				\| \p_t^\ell \p_x^4 R^\BL \|_{L^2}\lesssim B^2 (1+t)^{-3 -\ell},\quad
    \|\p_x^2 \Delta^2 R^\BL\|\lesssim B^2 (1+t)^{-2},\quad \|\Delta^4 R^\BL\|_{L^2}\lesssim B^2 (1+t)^{-\frac{9}{8}}.
				\]
				
			\end{enumerate}

			\label{lem:BL-main}
		\end{lemma}

\begin{remark}
	The reader may compare formulas \eqref{trace-Delta2-thbl}, \eqref{trace-dzDelta2-thbl} with the ones from \cref{lem:decomp-gamma}. The terms $\Gamma^j_{a,T}$ are  lifted neither by the boundary layer term $\thbl$ nor by the corrector $\theta_c$, and an additional corrector will be built to handle them, see \cref{lem:BL-NLcorrector} below.
\end{remark}	

		\begin{remark}
			Actually, all profiles $\Theta^j_a,\Psi^j_a$, and therefore $\thbl,\psbl$ depend on $T$ through $\gamma^0_{a,T},\gamma^1_{a,T}$.
			However, in order not to burden unnecessarily the notation, we will omit this dependency in the present Section. 
			The dependency will be restored in \cref{sub:bootstrap-thint} when we perform the final bootstrap argument.
			\label{rem:thbl_depends_on_T}
			
		\end{remark}

		Once the boundary layer part is constructed, under an additional bootstrap assumption on the remainder, we can define a nonlinear corrector:
		\begin{lemma}
				Let $\theta_0\in H^{14}(\Om)$ such that $\|\theta_0\|_{H^{14}}\leq B <1$, and 
		$\theta_0\vert_{\p \Om}=\p_\bn \theta_0\vert_{\p \Om}=0$, $\p_z^2 \bar\theta_0\vert_{\p \Om}=0$.
		
		Let $\theta= \bar \theta+\theta'$ be a solution of \eqref{eq:STi}, and assume that the bounds \eqref{hyp:bootstrap-theta'-NLBL} hold on $(0,T)$.
			
			Let $\thbl,\psbl$ be given by \cref{lem:BL-main}, and let $\threm=\theta'-\thbl$. Assume  that \eqref{hyp:bootstrap-thiint} holds on $(0,T)$, and define $\Gamma^j_{a,T}$ as in \cref{lem:decomp-gamma}.
			
			Then there exists $\sigma_\mathrm{lift}^\mathrm{NL}\in H^8(\Om)$ such that
			\[\ba 
			\Delta^2 \sigma_\mathrm{lift}^\mathrm{NL}\vert_{z=0}= \Gamma^0_{\mbot,T},\quad \p_z \Delta^2 \sigma_\mathrm{lift}^\mathrm{NL}\vert_{z=0}= \Gamma^1_{\mbot,T},\\
			\Delta^2 \sigma_\mathrm{lift}^\mathrm{NL}\vert_{z=1}= \Gamma^0_{\mtop,T},\quad \p_z \Delta^2\sigma_\mathrm{lift}^\mathrm{NL}\vert_{z=1}= \Gamma^1_{\mtop,T}\ea
			\]
			and for all $t\in [0,T]$,
			\[
			\ba 
			\| \Delta^4 \sigma_\mathrm{lift}^\mathrm{NL}\|_{L^2}\lesssim B^2 (1+t)^{-1/12} ,\quad \|\p_x^2 \Delta^2 \sigma_\mathrm{lift}^\mathrm{NL}\|_{L^2}\lesssim B^2 (1+t)^{-1+ \frac{1}{64}},\quad \|\p_x^4 \sigma_\mathrm{lift}^\mathrm{NL}\|_{L^2}\lesssim B^2 (1+t)^{-2}.
			\ea
			\]
			As a consequence, setting $\thapp:=\thbl + \theta_c+ \sigma_\mathrm{lift}^\mathrm{NL}$, we have
			\[
			\ba 
			\Delta^2 \thapp = \Delta^2 \theta', \quad \p_\bn \Delta^2 \thapp = \p_\bn \Delta^2 \theta'\quad \text{on }\p\Om.
			\ea
			\]
			Furthermore $\thapp$ is a solution of
			\[
			\p_t \thapp = (1-G) \p_x^2 \Delta^{-2} \thapp - \left(\nabla^\bot \Delta^{-2} \p_x \thapp \cdot \nabla \thapp\right)' + S_\rem,
			\]
			and the remainder $S_\rem$ is such that for all $k,m\geq 0$ with $k+m\leq 8$,
			\[
				\| \p_t^\ell \p_x^4 S_\rem \|_{L^2}\lesssim B^2 (1+t)^{-3 -\ell},\quad
    \|\p_x^2 \Delta^2 S_\rem\|\lesssim B^2 (1+t)^{-2},\quad \|\Delta^4 S_\rem\|_{L^2}\lesssim B^2 (1+t)^{-\frac{9}{8}}.
			\]

			\label{lem:BL-NLcorrector}
		\end{lemma}

		The main part of this section will be devoted to the proof of \cref{lem:BL-main}.
		The strategy will be very similar to the one of \cref{sec:linear-BL}, and we will often refer the reader to the computations therein.
		We begin with the construction of the profiles $\Theta^j_a$, $\Psi^j_a$\footnote{We recall that these profiles are different from the ones constructed in \cref{sec:linear-BL}, in spite of a similar notation.}.
		To that end,
		we plug the Ansatz \eqref{def:thbl} into \eqref{eq:STi} and identify the powers of $1+t$ in the vicinity of $z=0$ or $z=1$. Note that for $z\ll 1$ and $t\in [0,T]$, setting $Z=(1+t)^{1/4}z$ and using \cref{lem:CV-gamma-G},
		\beq\label{expansion-G}\ba 
		G(t,z)&=\frac{1}{2} \p_z^2 G(t,0) z^2 + \frac{1}{6} \p_z^3 G(t,0) z^3 + O(z^4)\\
		&= (1+t)^{-1/2}\frac{g^2_{\mbot,T}}{2} Z^2 +  (1+t)^{-3/4}\frac{g^3_{\mbot,T}}{6} Z^3 \\
		&\qquad + O((1+t)^{-1} Z^4 + (1+t)^{-5/4} (Z^2 + Z^3)) .
		\ea
		\eeq
		A similar expansion holds in the vicinity of $z=1$.
		Furthermore, in the vicinity of $z=0$, setting $S=-(\nabla^\bot \psi \cdot \nabla \theta')'$ and assuming that \eqref{hyp:bootstrap-thiint} holds,
		\begin{equation*}\label{expansion-S_BL}
		S=\sum_{0\leq i,j\leq 4} (1+t)^{-3 - \frac{i+j-1}{4}} \left( \p_Z \Psbot^i \p_x \Thbot^j - \p_x \Psbot^i \p_Z \Thbot^j\right)' + O((1+t)^{-15/4})\quad \text{in } L^2.
		\end{equation*}
		Following the computations of the previous section and identifying the coefficient of $(1+t)^{-2-\frac{j}{4}}$, we 
		obtain, for $j\in \{0,\cdots, 3\}$ (compare with \eqref{def:psiaj-thetaaj})
		\beq\label{eq:Thetaj-BL-EDP}
		- \left(1 + \frac{j}{4}\right) \Theta^j_a + \frac{1}{4}Z \p_Z \Theta^j_a = \p_x \Psi^j_a + S^j_a,
		\eeq
		where the source terms $S^j_a$ are defined by
		\beq\label{def:Sja}
		\ba
		S^0_a &=S^1_a=0,\\
		S^2_a &=-\frac{g^2_{a,T}}{2} Z^2\p_x \Psi^0_a,\\
		S^3_a &=-\frac{g^2_{a,T}}{2} Z^2\p_x \Psi^1_a- \eta_a \frac{g^3_{a,T}}{6} Z^3 \p_x \Psi^0_a + \eta_a \left( \p_Z \Psi_a^0 \p_x \Theta_a^0 - \p_x \Psi_a^0 \p_Z \Theta_a^0\right)',
		\ea
		\eeq
		with $\eta_\mbot=1,\eta_{\mtop}=-1$.

		Let us now proceed to define recursively the profiles $\Theta^j_a, \Psi^j_a$.

		\paragraph{Main order boundary layer terms: $\Theta^0_a$ and $\Theta^1_a$.}
		
		The role of the boundary layer profiles $\Theta^j_a$ for $j=0,1$ is to correct the traces of $\Delta^2 \theta'$ and $\p_z\Delta^2 \theta'$ on $\p \Om$ \emph{at main order}, \ie $\gamma^j_{a,T}$ (see \cref{lem:decomp-gamma}).
		Choosing $\Psi^j_a$  such that  $\p_Z^4 \Psi^j_a=\p_x\Theta^j_a$ for $j=0,1$ and recalling \eqref{eq:Thetaj-BL-EDP}, we are led to
		\begin{equation*}
		\begin{cases}
			Z \p_Z^5 \Theta^0_a=4\p_x^2 \Theta^0_a\quad \text{in } \T\times (0,+\infty)\\
			\p_Z^4\Theta^0_{a|Z=0}=\gamma^0_{a,T}, \quad \p_Z^5 \Theta^0_{a|Z=0}=0,\\
			\Theta^0_{a|Z=0}=0,\quad \p_Z  \Theta^0_{a|Z=0}=0,\quad \lim_{Z\to \infty} \Theta^0_a=0,
		\end{cases}
		\end{equation*}
		and
		\begin{equation*}\label{def:Theta1}
		\begin{cases}
			Z \p_Z^6 \Theta^1_a=4\p_x^2\p_Z \Theta^1_a\quad \text{in } \T\times (0,+\infty)\\
			\p_Z^4\Theta^1_{a|Z=0}=0, \quad \p_Z^5 \Theta^1_{a|Z=0}=\eta_a\gamma^1_{a,T},\\
			\Theta^1_{a|Z=0}=0,\quad \p_Z  \Theta^1_{a|Z=0}=0,\quad \lim_{Z\to \infty} \Theta^1_a=0.
		\end{cases}
		\end{equation*}
		Note that these systems are identical to \eqref{eq:Psi_i} and \eqref{eq:Phi_i} respectively. 
		As a consequence, as in the previous section (see \eqref{formula-Fourier-PsiO}), we find that
		\beq\label{formula-Fourier-0-NL}\ba 
		\Theta^0_a(x,Z)&=\sum_{k\in \Z\setminus\{0\}}|k|^{-2}\widehat{\gamma^0_{a,T}}(k) \chi_0(|k|^{1/2} Z)  e^{ikx},\\
		\Psi^0_a(x,Z)&=\sum_{k\in \Z\setminus\{0\}}\frac{1}{ik|k|^2}\widehat{\gamma^0_{a,T}}(k) \left[ \frac{1}{4} |k|^{1/2} Z \chi_0'(|k|^{1/2} Z) - \chi_0(|k|^{1/2} Z)  \right]e^{ikx},
		\ea
		\eeq
		where $\chi_0$ is defined in \cref{def:chi_j}. 
		Since $\|\gamma^0_{a,T}\|_{H^9}\lesssim \|\theta_0\|_{H^{14}}+ B^2$ according to \cref{lem:CV-gamma-G}, it follows that
		\[\ba
		\|\Theta^0_a\|_{H^{11}_x L^2_Z} + \|\Theta^0_a\|_{L^2_xH^{22}_Z} \lesssim \|\theta_0\|_{H^{14}}+ B^2,\\
		\|\Psi^0_a\|_{H^{12}_x L^2_Z} + \|\Psi^0_a\|_{L^2_xH^{24}_Z} \lesssim \|\theta_0\|_{H^{14}}+ B^2.\ea
		\]
		In a similar fashion, recalling the definition of $\chi_1$ from \cref{def:chi_j} (see also \eqref{formula-Fourier-Psi1}),
		\beq \label{formula-Fourier-1-NL}\ba 
		\Theta^1_a(x,Z)&=\eta_a\sum_{k\in \Z\setminus\{0\}}|k|^{-5/2}\widehat{\gamma^1_{a,T}}(k) \chi_1(|k|^{1/2} Z)  e^{ikx},\\
		\Psi^1_a(x,Z)&=\eta_a\sum_{k\in \Z\setminus\{0\}}\frac{1}{ik|k|^{5/2}}\widehat{\gamma^1_{a,T}}(k) \left[ \frac{1}{4} |k|^{1/2} Z \chi_1'(|k|^{1/2} Z) - \frac{5}{4}\chi_1(|k|^{1/2} Z)  \right]e^{ikx}.
		\ea
		\eeq
		Since $\|\gamma^1_{a,T}\|_{H^8}\lesssim \|\theta_0\|_{H^{14}}+ B^2$, we also have
		\[\ba
		\|\Theta^1_a\|_{H^{21/2}_x L^2_Z} + \|\Theta^1_a\|_{L^2_xH^{21}_Z} &\lesssim \|\theta_0\|_{H^{14}}+ B^2,\\
		\|\Psi^1_a\|_{H^{23/2}_x L^2_Z} + \|\Psi^1_a\|_{L^2_xH^{23}_Z} &\lesssim \|\theta_0\|_{H^{14}}+ B^2.\ea
		\]
		Let us now define the boundary terms {$\gamma^{0,j}_a$ and $\gamma^{1,j}_a$ by \cref{def:gamma-j-al}. It follows from the above expressions for $\Psi^j_a$ and $\Theta^j_a$  and from the boundary conditions $\chi_0(0)=\chi_0'(0)=0$ that
		\beq\label{est:gamma-j-al}
		\ba
		\|\gamma^{0,2}_a\|_{H^{10}(\T)} &\lesssim B^2,& \|\gamma^{0,3}_a\|_{H^{17/2}(\T)} &\lesssim B^2,\\
		\|\gamma^{1,1}_a\|_{H^{19/2}(T)} &\lesssim B^2,& \|\gamma^{1,2}_a\|_{H^{8}(\T)} &\lesssim B^2.
		\ea
		\eeq
		In a similar fashion, defining the source terms $S^2_a$, $S^3_a$ by \eqref{def:Sja}, we have
		\beq \label{est:Sja}
		\ba
		\|S^2_a\|_{H^{12}_x L^2_Z}  + \|S^2_a\|_{L^2_xH^{24}_Z} \lesssim B^2,\\
		\| S^3_a\|_{H^{10}_x L^2_Z} + \| S^3_a\|_{L^2_x H^{20}_Z} \lesssim B^2.
		\ea
		\eeq
		Note however that because of the quadratic term $\{\Psi^0_a,\Theta^0_a\}_{x,Z}$, $S^3_a$ does not have the same self-similar structure as $\Psi^j_a,\Theta^j_a$ for $j=0,1$, which is also shared by $S^2_a$.

		\paragraph{Correctors $\Theta^0_{c,a}$ and $\Theta^1_{c,a}$.}
		
		We recall that the coefficients $\gamma^{j,k}_a$ are defined  by   \eqref{def:gamma-j-al}, and are estimated in \eqref{est:gamma-j-al} above.
		The terms $\gamma^{0,2}_a(1+T)^{-1/2}$ and  $\gamma^{0,3}_a(1+T)^{-3/4}$ in \cref{lem:decomp-gamma} are constant in time, but smaller (for $T\gg 1$) than $\gamma^0_{a,T}$. Hence they give rise to a profile $\Theta^0_{c,a}$ whose construction is very similar to the one of $\Theta^0_a$, but whose size is much smaller.
		More precisely, we set
		\[\ba 
		\Theta^0_{c,a}(x,Z)&=\sum_{k\in \Z \setminus\{0\}}|k|^{-2} \left[ \widehat{\gamma^{0,2}_a} (k)(1+T)^{-1/2} 
		+ \widehat{\gamma^{0,3}_a}(k) (1+T)^{-3/4}\right]\chi_0(|k|^{1/2} Z) e^{ikx},\\
		\Theta^1_{c,a}(x,Z)&=\eta_a \sum_{k\in \Z \setminus\{0\}}|k|^{-5/2} \left[ \widehat{\gamma^{1,1}_a}(k) (1+T)^{-1/4} +\widehat{\gamma^{1,2}_a} (1+T)^{-1/2} \right]\chi_1(|k|^{1/2} Z) e^{ikx}.
		\ea
		\]
		Remembering \eqref{est:gamma-j-al}, we have, for $j=0,1$
		\[
    \ba
		\|\Theta^j_{c,a}\|_{H^{21/2}_x L^2_Z} +\|\Theta^j_{c,a}\|_{L^2_x H^{21}_Z} &\lesssim B^2 (1+T)^{-\frac{1}{2} + \frac{j}{4}},\\
		\|\p_Z^2 \Theta^j_{c,a}\vert_{Z=0}\|_{H^{19/2}(\T)} &\lesssim B^2 (1+T)^{-\frac{1}{2} + \frac{j}{4}}.
		\ea
    \]
		Analogously to $\Psi^0_a$ and $\Psi^1_a$, we also define
		\[
		\ba 
		\Psi_{c,a}^0&=\sum_{k\in \Z^*} \frac{1}{ik|k|^2} \left[ \widehat{\gamma^{0,2}_a} (k)(1+T)^{-\frac{1}{2}} 
		+ \widehat{\gamma^{0,3}_a}(k) (1+T)^{-\frac{3}{4}}\right] \left( \frac{1}{4} \xi \chi_0'(\xi) - \chi_0(\xi)\right)\vert_{\xi=|k|^{1/2} Z} e^{ikx},\\
		\Psi_{c,a}^1&=\eta_a\sum_{k\in \Z^*} \frac{1}{ik|k|^{5/2}}\left[ \widehat{\gamma^{1,1}_a}(k) (1+T)^{-\frac{1}{4}} +\widehat{\gamma^{1,2}_a} (1+T)^{-\frac{1}{2}} \right]\left( \frac{1}{4} \xi \chi_1'(\xi) - \chi_1(\xi)\right)\vert_{\xi=|k|^{1/2} Z} e^{ikx},
		\ea
		\]
		so that $\p_Z^4 \Psi_{c,a}^j=\p_x \Theta_{c,a}^j$, and we have
		\[
		\|\Psi^j_{c,a}\|_{H^{23/2}_x L^2_Z} +\|\Psi^j_{c,a}\|_{L^2_x H^{23}_Z}\lesssim B^2 (1+T)^{-\frac{1}{2} + \frac{j}{4}}.
		\]

		\paragraph{Lower order boundary layer terms: $\Theta^2_a$, $\Theta^3_a$ and $\Theta^2_{c,a}$.}
		
		We recall that $\Theta^j_a$, $\Psi^j_a$ must satisfy \eqref{eq:Thetaj-BL-EDP}, where the source term $S^j_a$ is given by \eqref{def:Sja}.
		Note that since $\Psi^0_a$, $\Psi^1_a$ and $\Theta^0_a$ have been constructed in the previous step, the source terms $S^2_a$ and $S^3_a$ are defined unequivocally and have exponential decay.
		Moreover, following \cref{lem:decomp-gamma} and noting that
		\[
		\Delta^2 \thbl\vert_{z=0}= \sum_{j=0}^3 (1+t)^{-j/4} \p_Z^4 \Theta^j_{\mbot|Z=0} + 2 \sum_{j=0}^3 (1+t)^{-\frac{1}{2}-\frac{j}{4}} \p_x^2 \p_Z^2  \Theta^j_{\mbot|Z=0} + O((1+t)^{-1}),
		\]
		we enforce the following boundary conditions:
		\beq\label{BC-Theta2-Theta3}
		\ba 
		\p_Z^4 \Theta^2_{a|Z=0}&=\gamma^{0,2}_a - 2 \p_x^2 \p_Z^2 \Theta^0_{a|Z=0},& \p_Z^5 \Theta^2_{a|Z=0}&=\eta_a \gamma^{1,1}_a - 2 \p_x^2\p_Z^3\Theta^0_{a|Z=0} ,\\
		\p_Z^4 \Theta^3_{a|Z=0}&=\gamma^{0,3}_a- 2 \p_x^2 \p_Z^2 \Theta^1_{a|Z=0},& \p_Z^5 \Theta^3_{a|Z=0}&=\eta_a \gamma^{1,2}_a- 2 \p_x^2\p_Z^3\Theta^1_{a|Z=0} ,
		\ea
		\eeq
		where the coefficients $\gamma^{j,k}_a$ are defined in \eqref{def:gamma-j-al} and estimated in \eqref{est:gamma-j-al}.
		There remains to specify the relationship between $\Psi^j_a$ and $\Theta^j_a$. In order that $\Delta^2 \psbl=\p_x \thbl$ at main order, following the computations of the previous section (see in particular \eqref{eq:Delta2Psbl}), we take, for $j=2,3$
		\[
		\p_Z^4 \Psi^j_a + 2 \p_x^2 \p_Z^2 \Psi^{j-2}_a=\p_x \Theta^j_a.
		\]
		Eliminating $\Psi^2_a$ from the equation on $\Theta^2_a$, we find that the system satisfied by $\Theta^2_a$ is
		\beq\label{eq:Theta2}
		\begin{cases}
			Z \p_Z^5 \Theta^2_a -2 \p_Z^4 \Theta^2_a= 4 \p_x^2 \Theta^2_a -2{g^2_{\mbot,T}}\p_Z^4(Z^2\p_x \Psi^0_a) - 8 \p_x^3 \p_Z^2 \Psi^0_a,\\
			\Theta^2_{a|Z=0}=\p_Z \Theta^2_{a|Z=0}=0,\\
			\Theta^2_a(Z)\to 0\quad \text{as } Z\to \infty,
		\end{cases}
		\eeq
		together with \eqref{BC-Theta2-Theta3}.
		Note that $2 \p_x \p_Z^2 \Psi^0_{a|Z=0}= - \p_Z^2 \Theta^0_{a|Z=0}$ and $4 \p_x \p_Z^3 \Psi^0_{a|Z=0}= \p_Z^3 \Theta^0_{a|Z=0} $, so that the boundary conditions are (once again) redundant. In other words, taking the trace of \eqref{eq:Theta2} at $Z=0$, we find $\p_Z^4 \Theta^2_{a|Z=0}=\gamma^{0,2}_a - 2 \p_x^2 \p_Z^2 \Theta^0_{a|Z=0}$. Differentiating twice more with respect to $Z$, we find that the Fourier transform of $\p_Z^2 \Theta^2_{a}$, after a suitable lifting, satisfies an equation of the form \eqref{ODE} with boundary conditions of the type (iii) from \cref{lem:WP-ODE-BL}.
		Using the explicit Fourier representation of $\Psi^0_a$ and $\Theta^0_a$ \eqref{formula-Fourier-0-NL}, we find that
		\begin{equation*}
		\|\Theta^2_a\|_{H^{10}_x L^2_Z} + \|\Theta^2_a\|_{L^2_x H^{20}_Z}\lesssim \|\theta_0\|_{H^{14}}+ B^2,\quad
		\| \Psi^2_a \|_{H^{11}_x L^2_Z} + \|\Psi^2_a \|_{L^2_x H^{22}_Z}\lesssim \|\theta_0\|_{H^{14}}+ B^2 .
		\end{equation*}
		In a similar fashion, $\Theta^3_a$ satisfies the system
		\begin{equation*} \label{eq:Theta3}
		\begin{cases}
			Z \p_Z^5 \Theta^3_a - 3 \p_Z^4 \Theta^3_a = 4\p_x^2 \Theta^3_a + 4\p_Z^4 S^3_a + 6 \p_x^2 \p_Z^2 \Theta^1_a - 2 Z \p_x^2 \p_Z^3 \Theta^1_a,\\
			\Theta^3_{a|Z=0}=\p_Z \Theta^3_{a|Z=0}=0,\\
			\Theta^3_a(Z)\to 0\quad \text{as } Z\to \infty,
		\end{cases}
		\end{equation*}
		together with \eqref{BC-Theta2-Theta3}. Once again, we find that the lifted Fourier transform of $\p_Z^3 \Theta^3_a$ satisfies an equation of the form \eqref{ODE} with boundary conditions of the type (iv) from \cref{lem:WP-ODE-BL}. Using the explicit Fourier representation of $\Theta^1_a$ \eqref{formula-Fourier-1-NL} together with the estimates on $S^3_a$ \eqref{est:Sja}, we find that
		\begin{equation*} 
		\|\Theta^3_a\|_{H^{19/2}_x L^2_Z} + \|\Theta^3_a\|_{L^2_x H^{19 }_Z}\lesssim \|\theta_0\|_{H^{14}}+ B^2,\quad
		\| \Psi^3_a \|_{H^{21/2}_x L^2_Z} + \|\Psi^3_a \|_{L^2_x H^{21}_Z}\lesssim \|\theta_0\|_{H^{14}}+ B^2 .
		\end{equation*}
		Note that the Fourier representation of $\Theta^2_a$ and of the linear part of $\Theta^3_a$ also ensure that for $j=2,3$,
		\beq \label{est-trace-Theta23}
		\| \p_Z^2 \Theta^j_a\vert_{Z=0}\|_{H^8(\T)} + \| \p_Z^3 \Theta^j_a\vert_{Z=0}\|_{H^{15/2}(\T)}\lesssim \|\theta_0\|_{H^{14}}+ B^2.
		\eeq
		Eventually, we define $\Theta^2_{c,a}$ analogously to $\Theta^2_a$ so that
		\[
		\begin{cases}
			Z \p_Z^5 \Theta^2_{c,a} - 2\p_Z^4 \Theta^2_{c,a} = 4\p_x^2 \Theta^2_{c,a} - 8 \p_x^3 \p_Z^2 \Psi^0_{c,a},\\
			\Theta^2_{c,a}\vert_{Z=0}=\p_Z \Theta^2_{c,a}\vert_{Z=0}=0,\\
			\p_Z^j\Theta^2_{c,a}\vert_{Z=0}= -2 \p_x^2 \p_Z^{j-2} \Theta^0_{c,a}\vert_{Z=0}\quad \forall j \in \{4,5\},\\
			\Theta^2_{c,a}(x,Z)\to 0\quad \text{as } Z\to \infty.
		\end{cases}
		\]
		Once again, note that the boundary conditions are redundant. We also define $\Psi^2_{c,a}$ by $\p_Z^4 \Psi^2_{c,a}= \p_x \Theta^2_{c,a} - 2 \p_x^2 \p_Z^2 \Psi^2_{c,a}$, with homogeneous boundary conditions at $Z=0$. We obtain
		\[
		\ba 
		\|\Theta^2_{c,a}\|_{H^{9}_x L^2_Z} +\|\Theta^2_{c,a}\|_{L^2_x H^{18}_Z}%
		&\lesssim B^2 (1+T)^{-\frac{1}{2} },\\
		\| \Psi^2_{c,a}\|_{H^{10}_x L^2_Z}+ \|\Psi^2_{c,a}\|_{L^2_x H^{20}_Z}&\lesssim B^2 (1+T)^{-\frac{1}{2} }.
		\ea
		\]

		\paragraph{Boundary layer corrector $\Theta^4_a$.}
		As in the previous section, we need to define a higher order boundary layer corrector $\Theta^4_a$, whose role is to ensure that
		\[
		\left\|\p_x^2 \Delta^{-2} \thbl - \p_x \psbl \right\|_{L^2}\lesssim B (1+t)^{-3}.
		\]
		To that end, we choose $\Theta^4_a$, $\Psi^4_a$ so that
		\[
		\ba 
		\p_Z^4 \Psi^4_a + 2 \p_x^2 \p_Z^2 \Psi^2_a + \p_x^4 \Psi^0_a&=\p_x \Theta^4_a,\\
		Z \p_Z \Theta^4_a - 8 \Theta^4_a&=4\p_x \Psi^4_a.
		\ea
		\]
		Eliminating $\Psi^4_a$ from the equation, we find
		\[
		Z \p_Z^5 \Theta^4_a - 4 \p_Z^4 \Theta^4_a= 4\p_x^2 \Theta^4_a - 8 \p_x^3 \p_Z^2 \Psi^2_a - 4 \p_x^5 \Psi^{0}_a.
		\]
		We enforce the boundary conditions (which are redundant):
		\[
		\Theta^4_a\vert_{Z=0}=\p_Z\Theta^4_a\vert_{Z=0}=0,\ \p_Z^4 \Theta^4_a\vert_{Z=0}=2\p_x^3 \p_Z^2 \Psi^2_a\vert_{Z=0},\quad \p_Z^5 \Theta^4_a\vert_{Z=0}=\frac{8}{3}\p_x^3 \p_Z^3 \Psi^2_a\vert_{Z=0},
		\]
		together with a decay assumption at infinity. Looking at the equation satisfied by the Fourier transform and applying \cref{lem:WP-ODE-BL}, we infer that there exists a (non unique) solution $\Theta^4_a$ of this equation such that
		\[
		\| \Theta^4_a\|_{H^{9}_x L^2_Z} + \| \Theta^4_a\|_{L^2_x H^{18}_Z} 
		\lesssim \|\theta_0\|_{H^{14}}+ B^2.
		\]
		As in the previous section (see the discussion on page \pageref{def:psi4}),  non-uniqueness comes from the fact that the Fourier transform of $\p_Z^4 \Theta^4_a$ satisfies an ODE of the form \eqref{ODE}, with boundary conditions at $Z=0$ for $\p_Z^4 \Theta^4_a$ and $\p_Z^5 \Theta^4_a$. However, the boundary conditions above do not prescribe any condition on $\p_Z^k \Theta^4_a$ for any $k\geq 6$. We lift this indetermination by requiring (somewhat arbitrarily) that $\p_Z^8 \Theta^4_a\vert_{Z=0}=0$. The solution thus obtained satisfies the previous Sobolev estimates, and its trace satisfies
		\beq \label{est-trace-Theta4}
		\| \p_Z^2 \Theta^4_a\vert_{Z=0}\|_{H^{8}(\T)} + \| \p_Z^3 \Theta^4_a\vert_{Z=0}\|_{H^{15/2}(\T)}\lesssim \|\theta_0\|_{H^{14}}+ B^2.
		\eeq
		
		\paragraph{Lift of the remaining traces of order $B$.}
		At this stage, we have defined $\Theta^j_a$, $\Psi^j_a$ for $0\leq j \leq 4$ together with $\Theta^j_{c,a}$, $\Psi^j_{c,a}$ for $0\leq j \leq 2$. 
		Let $\chi\in C^\infty_c(\R)$ be a cut-off function such that $\chi\equiv 1$ on $(-1/4,1/4)$ and $\supp \chi \subset (-1/2,1/2)$.
		Setting $\Theta^j_{c,a}=\Psi^j_{c,a}=0$ for $j\geq 3$ and  $Z_\mbot=(1+t)^{1/4} z$, $Z_\mtop=(1+t)^{1/4} (1-z)$, the main boundary layer term is given by
		\begin{align*}
			\thbl_\text{main}:={}&\sum_{j=0}^4 (1+t)^{-1-\frac{j}{4}} \left(\Theta^j_\mbot + \Theta^j_{c,\mbot}\right) (x,Z_\mbot) \chi(z)\\
			&+ \sum_{j=0}^4 (1+t)^{-1-\frac{j}{4}} \left(\Theta^j_\mtop + \Theta^j_{c,\mtop}\right) (x,Z_\mtop) \chi(1-z),\\
			\psbl_\text{main}:={}& \sum_{j=0}^4 (1+t)^{-2-\frac{j}{4}} \left(\Psi^j_\mbot + \Psi^j_{c,\mbot}\right) (x,Z_\mbot) \chi(z)\\
			&+ \sum_{j=0}^4 (1+t)^{-2-\frac{j}{4}} \left(\Psi^j_\mtop + \Psi^j_{c,\mtop}\right)  (x,Z_\mtop) \chi(1-z),
		\end{align*}
		
		By construction, we have
		\begin{align*}
			\Delta^2 \thbl_\text{main}\vert_{z=0}={}& \gamma^0_{\mbot, T} + \gamma^0_{\mbot, 2} (1+t)^{-1/2} + \gamma^0_{\mbot, 3} (1+t)^{-3/4} \\
            & - \gamma^0_{\mbot, 2} (1+T)^{-1/2} - \gamma^0_{\mbot, 3} (1+T)^{-3/4} \\
            &+2 (1+t)^{-5/4} \p_x^2 \p_Z^2 \Theta^3_{\mbot}\vert_{Z=0} + 2 (1+t)^{-3/2} \p_x^2 \p_Z^2 \Theta^4_{\mbot}\vert_{Z=0},\\
			\p_z \Delta^2 \thbl_\text{main}\vert_{z=0}={}& \gamma^1_{\mbot, T} + \gamma^1_{\mbot, 1} (1+t)^{-1/4} + \gamma^1_{\mbot, 2} (1+t)^{-1/2} \\
            &- \gamma^1_{\mbot, 1} (1+T)^{-1/4} - \gamma^1_{\mbot, 2} (1+t)^{-1/2} \\
			&+2 (1+t)^{-1} \p_x^2 \p_Z^3 \Theta^3_{\mbot}\vert_{Z=0} + 2 (1+t)^{-5/4} \p_x^2 \p_Z^3 \Theta^4_{\mbot}\vert_{Z=0}.
		\end{align*}
		Similar formulas hold at $z=1$. Comparing with \cref{lem:decomp-gamma,lem:BL-main}, we see that we need to lift the traces of $\p_x^2 \p_Z^k \Theta^j_a$ for $k=2, 3$ and $j\geq 3$.
		We  lift these remaining traces thanks to a corrector $\sigma_{\mathrm{lift}}^\mathrm{lin}$ which we define in Fourier in the following way. Let $\zeta_4,\zeta_5\in C^\infty_c(\R)$ such that $\zeta_j (Z)=Z^j/j!$ in a neighborhood of zero and such that $\supp \zeta_j \subset (-1/4, 1/4)$.
		In order to apply the last estimate of \cref{lem:taille_H-2}, we further choose $\zeta_j$ so that
		\beq\label{cancel-zeta}
		\int_0^\infty Z^k \zeta_j(Z)\ud Z=0\quad \forall k \in \{2,3\}.
		\eeq
		We then take
		\begin{align*}
			\widehat{\sigma_{\mathrm{lift}}^\mathrm{lin}}(t,k,z)={}& 2
			\sum_{l\geq 3, j=0,1} (1+t)^{-\frac{3}{2}-\frac{j+l}{4}}|k|^{-2-j}\widehat{\p_Z^{2+j} \Theta^l_\mbot}(k)\vert_{Z=0}\; \zeta_{4+j}(|k| z (1+t)^{1/4})\\
			&+ 
			2
			\sum_{l\geq 3, j=0,1} (1+t)^{-\frac{3}{2}-\frac{j+l}{4}}|k|^{-2-j}\widehat{\p_Z^{2+j} \Theta^l_\mtop}(k)\vert_{Z=0}\; \zeta_{4+j}(|k| (1-z) (1+t)^{1/4}),\\
		\end{align*}
	so that
\[
\ba 
\Delta^2 \sigma_{\mathrm{lift}}^\mathrm{lin}\vert_{z=0}&=-2 (1+t)^{-5/4} \p_x^2 \p_Z^2 \Theta^3_{\mbot}\vert_{Z=0} - 2 (1+t)^{-3/2} \p_x^2 \p_Z^2 \Theta^4_{\mbot}\vert_{Z=0},\\
\p_\bn \Delta^2 \sigma_{\mathrm{lift}}^\mathrm{lin}\vert_{z=0}&=-2 (1+t)^{-1} \p_x^2 \p_Z^3 \Theta^3_{\mbot}\vert_{Z=0} - 2 (1+t)^{-5/4} \p_x^2 \p_Z^3 \Theta^4_{\mbot}\vert_{Z=0}.
\ea
\]	

		The estimates on the traces $\Theta^j_a$ for $j\geq 2$ (see \eqref{est-trace-Theta23}, \eqref{est-trace-Theta4}) %
		ensure that for all $k,m\geq 0$ such that $k+m\leq 10$,
		\beq\label{est:theta-c-3}\ba
		\|\sigma_{\mathrm{lift}}^\mathrm{lin}\|_{H^m_x H^k_z} &\lesssim (\|\theta_0\|_{H^{14}}+ B^2) (1+t)^{-2-\frac{1}{8} + \frac{k}{4}},\\
		\|\p_t \sigma_{\mathrm{lift}}^\mathrm{lin}\|_{H^m_x H^k_z} &\lesssim (\|\theta_0\|_{H^{14}}+ B^2) (1+t)^{-3-\frac{1}{8} + \frac{k}{4}}. \ea
		\eeq
		We define an associated corrector $\phi_{\mathrm{lift}}^\mathrm{lin}=\Delta^{-2} \p_x \sigma_{\mathrm{lift}}^\mathrm{lin}$.
		According to \cref{lem:taille_H-2} and using \eqref{cancel-zeta}, 
		we have, for all $k,m\geq 0$ such that $k+m\leq 13$,
		\beq\label{est:psi-c-3}
		\ba 
		\| \phi_{\mathrm{lift}}^\mathrm{lin}\|_{H^m_x H^k_z} &\lesssim (\|\theta_0\|_{H^{14}}+ B^2) (1+t)^{-3 - \frac{1}{8} + \frac{k}{4}} \\
		\| \p_t \phi_{\mathrm{lift}}^\mathrm{lin}\|_{H^m_x H^k_z} &\lesssim (\|\theta_0\|_{H^{14}}+ B^2) (1+t)^{-4 - \frac{1}{8} + \frac{k}{4}}.
		\ea
		\eeq
		
		\paragraph{Evaluation of the remainder.}
		Let us now focus on the different remainder terms in the equation satisfied by $\thbl_\text{main}$, in view of defining one last linear corrector.
		\begin{itemize}
			\item \textit{Remainder stemming from the nonlinear term:} Using \cref{lem:reste-exponentiel} together with the estimates on $\Theta^j_a$, we have, 
			\begin{align*}
		&	\nabla^\bot \psbl_\text{main} \cdot \nabla \thbl_\text{main} \\
             ={}&   \sum_{0\leq j,k\leq 4} (1+t)^{-3-\frac{k+j-1}{4}} \Big\{\Psi^j_\mbot + \Psi^j_{c,\mbot}, \Theta^k_\mbot + \Theta^k_{c,\mbot}\Big\}_{x,Z}(x, Z_\mbot) \chi(z) \\
			&	- \sum_{0\leq j,k\leq 4} (1+t)^{-3-\frac{k+j-1}{4}} \{\Psi^j_\mtop + \Psi^j_{c,\mtop}, \Theta^k_\mtop + \Theta^k_{c,\mtop}\}_{x,Z}(x, Z_\mtop) \chi(1-z)\\
			&	+ O(\exp(-c(1+t)^{1/5})) \quad \text{in } H^9(\Om).
			\end{align*}
			In the above expansion, we put aside the terms corresponding to $k=j=0$, which are part of $S^3_a$ and are lifted by $\Theta^3_a$. If $j+k\geq 1$, we have,  when $0\leq s+r\leq 8$,
			\[
			\left\| \{\Psi^j_\mbot+\Psi^j_{c,\mbot}, \Theta^k_\mbot+ \Theta^k_{c,\mtop}\}_{x,Z}(x, (1+t)^{1/4} z) \chi(z) \right\|_{H^r_x H^{s}_z}\lesssim B^2 (1+t)^{\frac{s}{4}- \frac{1}{8}},
			\]
			and the same estimate holds for the top boundary layer. 
			We infer that
			\begin{equation*} \label{NL-BL}
			\ba
			\nabla^\bot \psbl_\text{main}\cdot \nabla \thbl_\text{main}={}& (1+t)^{-11/4} \{\Psi^0_\mbot, \Theta^0_\mbot\}_{x,Z}(x, Z_\mbot) \chi(z)\\& - (1+t)^{-11/4} \{\Psi^0_\mtop, \Theta^0_\mtop\}_{x,Z}(x, Z_\mtop) \chi(1-z)\\& + R_{\mathrm{NL}}, \ea
			\end{equation*}
			where for all $r,s\geq 0$, $r+s\leq 8$,
			\begin{equation*}\label{est:RNL}
			\|R_{\mathrm{NL}}\|_{H^r_x H^{s}_z}\lesssim B^2  (1+t)^{-3 +\frac{s}{4}- \frac{1}{8}}.
			\end{equation*}
			Note in particular that $\|R_{\mathrm{NL}}\|_{H^8}\lesssim B^2 (1+t)^{-1-\delta}$ with $\delta=1/8$.
			
			\item \textit{Remainder stemming from the Taylor expansion of $G$:} As explained in the construction of $\Theta^2_a$, $\Theta^3_a$, when defining the boundary layer term, we replaced $G$ by its Taylor expansion in the vicinity of $z=0$ and $z=1$. 
			Recalling \eqref{expansion-G}, we have, in the vicinity of $z=0$, setting $Z=(1+t)^{1/4} z$,
			\begin{align*}
				G \p_x \psbl_\text{main}={}&\frac{1}{2(1+t)^{1/2}}g^2_{\mbot,T} Z^2 \p_x \psbl_\text{main} + \frac{1}{6(1+t)^{3/4}} g^3_{\mbot,T} Z^3 \p_x \psbl_\text{main} \\
				&+ O((1+t)^{-1} (Z^2 + Z^4) \p_x \psbl_\text{main})\\
				={}&\frac{1}{2(1+t)^{5/2}}g^2_{\mbot,T} Z^2 \p_x \Psi^0_\mbot (x, Z) \chi(z) \\
				&+ (1+t)^{-11/4} \left(\frac{g^2_{\mbot,T} }{2} Z^2 \p_x \Psi^1_\mbot(x,Z) + \frac{g^3_{\mbot,T}}{6} Z^3\p_x \Psi^0_\mbot (x, Z) \right)\chi(z)\\
				&+ R_G,
			\end{align*}
			where the first two terms enter the definition of $\Theta^2_\mbot$ and $\Theta^3_\mbot$ respectively, and the remainder term $R_G$ satisfies
			\begin{equation*} \label{est-RG}
			\|R_G\|_{H^r_x H^{s}_z}\lesssim B^2  (1+t)^{-3 +\frac{s}{4}- \frac{1}{8}}\quad \text{if } 0\leq r+s\leq 8.
			\end{equation*}

			\item \textit{Remainder stemming from $\psbl - \Delta^{-2} \p_x \thbl$:} We now address the fact that $\Delta^2 \psbl_\text{main}$ is not equal to $\p_x \thbl_\text{main}$. More precisely, using the definition of $\Psi^j_a$, %
			we have, in $\Om\cap \{z\leq 1/2\}$,
			\begin{align*}
				\Delta^2 \psbl_\text{main} - \p_x \thbl_\text{main} ={}& 2\sum_{j=3,4} (1+t)^{-2-\frac{j-2}{4}} \p_x^2 \p_Z^2 \Psi^j_\mbot(x,(1+t)^{1/4} z) \chi(z)\\
				&+ \sum_{j\geq 1}(1+t)^{-2-\frac{j}{4}} \p_x^4 \Psi^j_\mbot(x,(1+t)^{1/4} z) \chi(z)\\
				&+2\sum_{j=1,2} (1+t)^{-2-\frac{j-2}{4}} \p_x^2 \p_Z^2 \Psi^j_{c,\mbot}(x,(1+t)^{1/4} z) \chi(z)\\
				&+ \sum_{j\geq 0}(1+t)^{-2-\frac{j}{4}} \p_x^4 \Psi^j_{c,\mbot}(x,(1+t)^{1/4} z) \chi(z)\\
				&+ O(\exp(-c(1+t)^{1/5}))\quad \text{in } H^8(\Om).
			\end{align*}
			A similar expression holds in $\Om\cap \{z\geq 1/2\}$, replacing bot with top and $z$ with $1-z$. The exponentially small remainder comes from the commutator of the bilaplacian with multiplication by $\chi$ (see \cref{lem:reste-exponentiel}), and from the estimates on $\Psi^j_a$, $\Psi^j_{c,a}$.
			We now apply \cref{lem:taille_H-2} and its variant \cref{rmk:taile_H-2-autosim}: more precisely, in order to avoid a high loss of horizontal derivatives, we apply the ``self-similar version'' from \cref{rmk:taile_H-2-autosim} to the term involving $\p_x^4 \Psi^1_\bot$, and the second statement from \cref{lem:taille_H-2} to all other terms.
			We obtain
			\begin{equation*}
			 \p_x \psbl_\text{main}- \Delta^{-2} \p_x^2 \thbl_\text{main} =: R_{\Delta^2},
      \end{equation*}
      with
      \begin{equation*}
        \sup_{t\in [0,T]}\left((1+t)^3 \|\p_x^5R_{\Delta^2}\|_{L^2} + (1+t)^{2+ \frac{3}{8}} \| \p_x^3 \Delta^2 R_{\Delta^2}\|_{L^2} +  (1+t)^{1+ \frac{3}{8}}\| \p_x \Delta^4 R_{\Delta^2}\|_{L^2} \right)\lesssim B.
			\end{equation*}
			Note that the decay of this remainder is similar to the one of $R_{\mathrm{NL}}$ and $R_G$, but its order of magnitude is $B$. Hence we call it a ``linear'' remainder. In order to simplify the forthcoming bootstrap argument, we will lift it thanks to another (linear) corrector.

			\item \textit{Remainder stemming from $\sigma_{\mathrm{lift}}^\mathrm{lin}$:} Recalling \eqref{est:theta-c-3}, \eqref{est:psi-c-3} and using a variant of \cref{rmk:taile_H-2-autosim}, we have, setting $R_{c,\mathrm{lin}}=\p_t \sigma_{\mathrm{lift}}^\mathrm{lin} - \Delta^{-2} \p_x^2 \sigma_{\mathrm{lift}}^\mathrm{lin}$, for $k+m\leq 10$,
			\begin{equation*}\label{est:reste-lineaire} 
    		\sup_{t\in [0,T]}\left((1+t)^3 \|\p_x^5R_{c,\mathrm{lin}}\|_{L^2} + (1+t)^2 \| \p_x^3 \Delta^2 R_{c,\mathrm{lin}}\|_{L^2} +  (1+t)^{9/8}\| \p_x \Delta^4 R_{c,\mathrm{lin}}\|_{L^2} \right)\lesssim B.
			\end{equation*}
			Once again, $R_{c,\mathrm{lin}}$ is a linear remainder, and shall be lifted before the bootstrap argument of the next subsection. We also have
			\[
			\| G \p_x^7 \Delta^{-2} \sigma_{\mathrm{lift}}^\mathrm{lin}\|_{L^2}\lesssim B^2 (1+t)^{-3}. 
			\]
			
		\end{itemize}
		
		In the remainders above, all terms of order $B^2(1+t)^{-3}$ in $L^2$ will be included in the remainder for the interior part (see \cref{sub:bootstrap-thint}), while the terms of order $B(1+t)^{-3}$ will be lifted   thanks to a linear corrector $\sigma^R$, which we now construct.
		
		\paragraph{Definition of $\sigma^R$.}
		Let $\sigma^R$ be the solution of
		\begin{equation*} 
		\ba
		\p_t \sigma^R &= \p_x^2 \Delta^{-2} \sigma^R - R_{\Delta^2} - R_{c,\mathrm{lin}},\\
		\sigma^R(t=0) &= 0.
		\ea
		\end{equation*}
		Note that $\p_t \sigma^R \vert_{\p\Om}= \p_t \p_\bn \sigma^R \vert_{\p\Om}=0$, and therefore $\sigma^R \vert_{\p\Om}= \p_\bn \sigma^R \vert_{\p\Om}=0$ for all $t>0$.
		Applying $\Delta^2$ to the above equation and taking the trace at $z=0$, we have, using the identity \eqref{eq:Thetaj-BL-EDP}
		\begin{align*}
			\p_t \Delta^2 \sigma^R\vert_{z=0}&=- \Delta^2 \left(R_{\Delta^2} + R_{c,\mathrm{lin}}\right)\vert_{z=0}\\
			&=-\p_x \Delta^2 \psbl_\text{main}\vert_{z=0} - \p_t \p_z^4 \sigma_{\mathrm{lift}}^\mathrm{lin}\vert_{z=0}\\
			&=-2 \sum_{j\geq 3} (1+t)^{-2-\frac{j-2}{4}} \p_x^3 \p_Z^2 \Psi^j_\mbot\vert_{Z=0} + 2 \sum_{j\geq 3} \frac{2+j}{4}(1+t)^{-2-\frac{j-2}{4}}\p_x^2 \p_Z^2\Theta^j_\mbot\vert_{Z=0}\\
			&=0.
		\end{align*}
		Hence $\Delta^2 \sigma^R\vert_{z=0}=0 $ for all $t\in (0,T)$. In a similar way, $\p_z\Delta^2 \sigma^R\vert_{z=0}=0 $ for all $t\in (0,T)$, and the same properties hold at $z=1$. 
		Applying first \cref{prop:bb-uniformbound} to $\p_x\Delta^4 \sigma^R$, and then \cref{prop:bb-decay} to $\p_x^3 \Delta^2 \sigma^R$, $\p_x^5 \sigma^R$, and $\p_t \p_x^4 \sigma^R$, we infer
		\begin{align*}
    		\|\p_x \Delta^4   \sigma^R\|_{L^2} & \lesssim \|\theta_0\|_{H^{14}}+ B^2,
            & \|\p_x^3 \Delta^2  \sigma^R\|_{L^2} &\lesssim (\|\theta_0\|_{H^{14}}+ B^2) (1+t)^{-1},\\
            && \| \p_x^5  \sigma^R\|_{L^2} &\lesssim (\|\theta_0\|_{H^{14}}+ B^2) (1+t)^{-2},\\
    		\| \p_t  \p_x^4\sigma^R\|_{L^2} &\lesssim (\|\theta_0\|_{H^{14}}+ B^2) (1+t)^{-3},
            & \| \p_x^6 \Delta^{-2} \sigma^R\|_{L^2 } &\lesssim (\|\theta_0\|_{H^{14}}+ B^2) (1+t)^{-3}.
		\end{align*}
		Furthermore, looking at the expressions of $R_{\Delta^2}$ and $R_{c,\mathrm{lin}}$ and recalling the estimates on $\Psi^j_a$, we can perform similar estimates for $z\chi(z)\p_z \sigma^R$ and $(1-z)\chi(1-z)\p_z \sigma^R$. For instance, estimating the commutators, we find that
		\begin{align*}
			\p_t \Delta^4 (z\chi(z) \p_z \sigma^R)={}& \Delta^4 (z\chi(z) \p_z\p_x^{2} \Delta^{-2} \sigma^R) -  \Delta^4 (z\chi(z) \p_z (R_{\Delta^2} + R_{c,\mathrm{lin}}))\\
			={}&\p_x^2 \Delta^2(z\chi(z) \p_z \sigma^R)\\
			&+ \left[ \Delta^4, z\chi(z) \p_z\right]  \p_x^{2} \Delta^{-2} \sigma^R + \p_x^2\left[z\chi(z)\p_z, \Delta^2\right]\sigma^R\\
			&-  \Delta^4 (z\chi(z) \p_z (R_{\Delta^2} + R_{c,\mathrm{lin}})),
		\end{align*}
		where
	\begin{align*}
        \left\|  \left[ \Delta^4, z\chi(z) \p_z\right]  \p_x^{2} \Delta^{-2} \sigma^R\right\|_{L^2} + \left\|\p_x^2\left[z\chi(z)\p_z, \Delta^2\right]\sigma^R\right\|_{L^2}
        &\lesssim \| \p_x^2 \Delta^2 \sigma^R\|_{L^2} \\
        &\lesssim   (\|\theta_0\|_{H^{14}}+ B^2) (1+t)^{-1},\\
        \| \Delta^4 (z\chi(z) \p_z (R_{\Delta^2} + R_{c,\mathrm{lin}}))\|_{L^2}
        &\lesssim (\|\theta_0\|_{H^{14}}+ B^2) (1+t)^{-9/8}.
	\end{align*}
	It follows that for all $t\in [0,T]$,
	\[
	\| \Delta^4 (z\chi(z) \p_z \sigma^R)\|_{L^2}\lesssim (\|\theta_0\|_{H^{14}}+ B^2)\ln (2+t).
	\]

		\paragraph{Conclusion.}
		Let
\[
\theta_c:=\sigma^\mathrm{lin}_\mathrm{lift} +   \sigma^R + \thbl_\mbot (\chi(z)-1) + \thbl_\mtop (\chi(1-z)-1) ,\quad \thbl=\thbl_\mbot + \thbl_\mtop,
\]
where
\[
\thbl_a=\sum_{j=0}^4 (1+t)^{-1-\frac{j}{4}} \left(\Theta^j_a + \Theta^j_{c,a}\right)(x, Z_a),\quad a\in \{\mtop,\mbot\}. 
\]
Then (up to a redefinition of $\Theta^j_a + \Theta^j_{c,a}$ as $\Theta^j_a$),
the bounds on the profiles $\Theta^j_a$ and  the corrector $\theta_c$, together with the boundary conditions on 
$\thbl+\theta_c$ announced in the statement of \cref{lem:BL-main} are all satisfied.

    Most of the remainder terms have already been evaluated. There only remains to evaluate the quadratic terms involving $\sigma_{\mathrm{lin}}^\mathrm{lift} $ and $\sigma^R$. We have for instance
		\[
		\left\|\p_x^2 \Delta^2( \nabla^\bot \psbl_\text{main}) \cdot \nabla \sigma^R\right\|_{H^r_xH^s_z} \lesssim  B^2 (1+t)^{-3/4} (1+t)^{-2+ \frac{1}{4}} \lesssim B^2 (1+t)^{-\frac{5}{2} }.
		\]
		For the $H^8$ estimate, we write,  for $z\leq 1/2$,
		\[
		\p_x \psbl_\text{main} \p_z \sigma^R= \frac{\p_x \psbl_\text{main}}{z} z\p_z \sigma^R.
		\]
		Both terms in the right-hand side belong to $H^8$, and we infer
		\[
		\left\| \nabla^\bot \psbl_\text{main} \cdot \nabla \sigma^R\right\|_{H^8} \lesssim  B^2 (1+t)^{-5/4}{\ln (2+t)}.
		\]
		The statement of  \cref{lem:BL-main} follows.
		\qed

		\begin{proof}[Proof of \cref{lem:BL-NLcorrector}]
			Assume that $\threm=\theta'-\thbl$ satisfies \eqref{hyp:bootstrap-thiint}, and define $\Gamma^j_{a,T}$ as in \cref{lem:decomp-gamma}. 
			According to  \cref{lem:decomp-gamma},
			\[
			\ba 
			\|\Gamma_{a, T}^j (t)\|_{L^2(\T)} &\lesssim B^2 (1+t)^{-1 + \frac{j}{4}},\\
			\|\p_t \Gamma_{a, T}^j (t)\|_{L^2(\T)} &\lesssim B^2 (1+t)^{-2 + \frac{j}{4}},\\
			\|\Gamma_{a, T}^j (t)\|_{H^4(\T)} &\lesssim B^2 (1+t)^{-\frac{23}{24} +  \frac{j}{4}}.
			\ea
			\]
			We now lift these traces thanks to a corrector $\sigma_\mathrm{lift}^\mathrm{NL}$, whose definition is similar to the one of $\sigma_{\mathrm{lift}}^\mathrm{lin}$, namely
			\begin{align*}
				\widehat{\sigma_\mathrm{lift}^\mathrm{NL}}(t,k,z)={}& 
				\sum_{j=0,1} (1+t)^{-1-\frac{j}{4}}|k|^{-4-j}\widehat{\Gamma^j_{\mbot,T}}(t,k)\zeta_{4+j}(|k| z (1+t)^{1/4})\\
				&+ \sum_{j=0,1} (1+t)^{-1-\frac{j}{4}}|k|^{-4-j}(-1)^j\widehat{\Gamma^j_{\mtop,T}}(t,k)\zeta_{4+j}(|k| (1-z) (1+t)^{1/4})
			\end{align*}
			where we recall that $\zeta_j\in C^\infty_c(\R)$, $\zeta(Z)=Z^j/j!\; $ in a neighborhood of zero, and $\zeta_j$ satisfies \eqref{cancel-zeta}.
			
			It follows from the  estimates on $\Gamma^j_{a,T}$ and from the formula defining $\sigma_\mathrm{lift}^\mathrm{NL}$  that for $\ell=0,1$,
			\begin{equation*}\label{est:theta-c-1}
			\ba 
			\|\p_t^\ell  \sigma_\mathrm{lift}^\mathrm{NL}\|_{H^m_x H^k_z} &\lesssim B^2 (1+t)^{-2 -\ell  + \frac{k}{4}- \frac{1}{8}}\quad \text{if }k+m\leq 9/2,\\
			\|\sigma_\mathrm{lift}^\mathrm{NL}\|_{H^{8}(\Om)} &\lesssim B^2 (1+t)^{-1/12},\\
			\| z\sigma_\mathrm{lift}^\mathrm{NL}\|_{H^{9}(\Om\cap \{z\leq 1/2\})} + \| (1-z)\sigma_\mathrm{lift}^\mathrm{NL}\|_{H^{9}(\Om\cap \{z\geq 1/2\})} &\lesssim B^2 (1+t)^{-1/12}.
			\ea
			\end{equation*}
			The function $\sigma_\mathrm{lift}^\mathrm{NL}$ has been designed so that
			\begin{align*}
    			\Delta^2\sigma_\mathrm{lift}^\mathrm{NL}\vert_{z=0} &= \Gamma^0_\mbot(t), &
                \p_\bn \Delta^2\sigma_\mathrm{lift}^\mathrm{NL}\vert_{z=0} &= \Gamma^1_\mbot(t),\\
    			\Delta^2\sigma_\mathrm{lift}^\mathrm{NL}\vert_{z=1} &= \Gamma^0_\mtop(t),&
                \p_\bn \Delta^2\sigma_\mathrm{lift}^\mathrm{NL}\vert_{z=1} &= \Gamma^1_\mtop(t).
			\end{align*}
			Furthermore, according to \cref{rmk:taile_H-2-autosim} and using \eqref{cancel-zeta}, 
			we have, for all $k,m\geq 0$ such that $k+m\leq 8$,
			\begin{equation*}\label{est:psi-c-1}
			\ba 
			\| \Delta^{-2} \sigma_\mathrm{lift}^\mathrm{NL}\|_{H^m_x H^k_z} &\lesssim B^2 (1+t)^{-3 - \frac{1}{8} + \frac{k}{4}} \\
			\| \p_t \Delta^{-2} \sigma_\mathrm{lift}^\mathrm{NL} \|_{H^m_x H^k_z} &\lesssim B^2 (1+t)^{-4 - \frac{1}{8} + \frac{k}{4}}.
			\ea
			\end{equation*}
			The statement of \cref{lem:BL-NLcorrector} follows immediately from these estimates and \cref{lem:decomp-gamma,lem:BL-main}.
			
		\end{proof}

\subsection{Bootstrap argument for \texorpdfstring{$\thint$}{θ\^{}int}}
\label{sub:bootstrap-thint}

	In this subsection, we complete the proof of \cref{thm:BL-nonlinear} thanks to a bootstrap argument (or rather, two nested bootstrap arguments).
	We start with an initial data $\theta_0\in H^{14} (\Om)$, with $\|\theta_0\|_{H^{14}(\Om)}\leq B $ and $\theta_0=\p_\bn \theta_0=0$ on $\p \Om$, $\p_z^2 \bar \theta_0=0$ on $\p \Om$.
	We assume that $B\leq B_0<1$, where $B_0$ is a small universal constant,
	so that \cref{thm:stab} holds.
	
	Let $\bar C\geq 2$ be a universal constant to be determined.
	We define
	\begin{equation*}\label{def:T1}
	T_1=\sup\left\{ T>0, \ \eqref{hyp:bootstrap-theta'-NLBL}  \text{ holds on }(0,T)\text{ with } B_1=\bar C \|\theta_0\|_{H^{14}}\right\}.
	\end{equation*}
	By continuity, $T_1>0$. For any $T\in (0,T_1)$, we define an associated boundary layer profile $\thbl_T$ (see \cref{lem:BL-main} and \cref{rem:thbl_depends_on_T}) together with a corrector $\theta_c$. 
	We recall that there exists a universal constant $C_1$ such that for all $m,k\geq 0$ with $k+m\leq 8$, for all $T\in (0, T_1)$, $t\in [0,T]$,
	\[
	\|\thbl_T(t)\|_{H^m_xH^k_z} \leq C_1 (\|\theta_0\|_{H^{14}} + B_1^2)(1+t)^{-1+\frac{k}{4} - \frac{1}{8}}\leq 2C_1 \|\theta_0\|_{H^{14}}(1+t)^{-1+\frac{k}{4} - \frac{1}{8}},
	\]
  provided $\bar C^2 B_0\leq 1$.
Similarly, for all $t\in [0, T_1]$,
 \[
(1+t)^2\|\p_x^5 \theta_c\|_{L^2} +  \|\p_x\Delta^4 \theta_c\|_{L^2}+ (1+t)^3\|\p_t \p_x^4 \theta_c\|_{L^2}+ (1+t)^3\|\p_x^5 \psi_c\|_{L^2}\leq 2C_1 \|\theta_0\|_{H^{14}}.
\]

	We then introduce a new time
	\beq 
	\label{def:T2}
	T_2=\sup\left\{ T\in (0, T_1), \ \threm=\theta'-\thbl_T \text{ satisfies }\eqref{hyp:bootstrap-thiint}\text{ on }(0,T)\text{ with } B_2=(2 \bar C + 3 C_1)\|\theta_0\|_{H^{14}}\right\}.
	\eeq
	On $(0,T_2)$, relying on \cref{lem:BL-NLcorrector}, we construct an approximate solution $\thapp$. We now set $\thint=\theta'-\thapp=\threm - \theta_c- \sigma_\mathrm{lift}^\mathrm{NL}$. Note that we can always choose $\| \theta_0\|_{H^{14}}$ small enough so that for all $t\in (0,T_2)$, for $0\leq k+m\leq 8$,
	\[
	\|\sigma_\mathrm{lift}^\mathrm{NL}\|_{H^m_x H^k_z}\leq \bar C  \|\theta_0\|_{H^{14}} (1+t)^{-2 + \frac{k}{4} - \frac{1}{8}}.
	\]
	Consequently, $\thint$ satisfies \eqref{hyp:bootstrap-thiint} with $B_3= (3 \bar C + 5 C_1) \|\theta_0\|_{H^{14}}$ on $(0, T_2)$. Note that $B_j\lesssim B$ for $j=1,2,3$.
	
	Our goal is now to prove that $T_1=T_2=+\infty$ for a suitable choice of $\bar C$, provided $\|\theta_0\|_{H^{14}}$ is sufficiently small.
	To that end, we check that $\Delta^2 \thint$ satisfies the assumptions  of \cref{prop:bb-decay}.

	By construction (see \cref{lem:BL-NLcorrector}),
	\[
	\thint=\p_\bn \thint=\Delta^2 \thint=\p_\bn \Delta^2 \thint=0\quad \text{on }\p\Om.
	\]
	
	Furthermore, defining the quadratic form
	\[
	Q(f,g)=-\left( \nabla^\bot \Delta^{-2} \p_x f \cdot \nabla g\right)', 
	\]
	we have, recalling \cref{lem:BL-NLcorrector},
	\beq\label{eq:threm}
	\p_t \thint= (1-G)\p_x^2 \Delta^{-2} \thint+ S^1_\rem,
	\eeq
	where, recalling the definition of $S_\rem $ from \cref{lem:BL-NLcorrector},
	\[
	S^1_\rem= - S_\rem +  Q(\thapp + \thint, \thint) + Q(\thint, \thapp).
	\]
	
	We claim that we have the following estimates on $S^1_\rem$:
	
	\begin{lemma}[Estimates on $S^1_\rem$]
		Let $T_2$ be defined by \eqref{def:T2}.
		
		\begin{itemize}
			\item $L^2$ and $H^4$ estimates:  for all $t\in [0, T_2)$,
		      \begin{align*}
			\| \p_x^4 S^1_\rem (t)\|_{L^2} &\lesssim B^2 \frac{1}{(1+t)^{3}},& \|\p_x^2 \Delta^2 S^2_\rem (t)\|_{L^2} &\lesssim B^2 \frac{1}{(1+t)^{2}};
			\end{align*}
		
			\item $H^8$ estimate: there exist $S^2_\para, S^2_\bot\in L^\infty([0, T_2), L^2 (\Om))$ such that $\Delta^4 S^1_\rem (t) =S^2_\para + S^2_\bot$, with
			\[
			\| S^2_\para (t)\|_{L^2} \lesssim B^2 \frac{1}{(1+t)^{9/8}}\quad \forall t\in [0, T_2)\quad \text{and}\quad \int_\Om S^2_\bot(t) \Delta^4 \thint(t)=0. 
			\]
			
			\item Estimates on the time derivative: for all $t\in [0, T_2)$,
			\[
			\| \p_t \p_x^4 S^1_\rem (t)\|_{L^2}\lesssim B^2 \frac{1}{(1+t)^{4}}.
			\]
		\end{itemize}
		
		\label{lem:est-Srem}

	\end{lemma}
	
	\begin{proof}
		We estimate each term separately. The estimates on $S_\rem$ have already been proved in the previous subsection (see \cref{lem:BL-NLcorrector}).
		Therefore we focus on the quadratic terms. It follows from the estimates of \cref{lem:BL-main,lem:BL-NLcorrector}, and from the bootstrap estimates \eqref{hyp:bootstrap-thiint} on $\threm$   that 
		\[
		\ba
		\left\| \p_x^4 Q(\thapp + \thint, \thint)\right\|_{L^2} &\lesssim B^2 (1+t)^{-3},\\
			\left\| \p_x^2 \Delta^2 Q(\thapp + \thint, \thint)\right\|_{L^2} &\lesssim B^2 (1+t)^{-2}.
		\ea
		\]
		For the $H^8$ estimate, the situation is slightly different, because $\Delta^4 Q(\thapp + \thint, \thint)$ involves derivatives of order 9 of $\thint$, for which we have no estimate.
		Therefore, as in \cref{sec:stab}, we decompose $\Delta^4Q(\thapp + \thint, \thint)$ into two parts, writing
		\[
		\ba
    		\Delta^4 Q(\thapp + \thint, \thint)={}& -\left( \nabla^\bot \Delta^{-2} \p_x (\thapp + \thint)\right) \cdot \nabla \Delta^4 \thint \\
            &- \p_z^8\;  \overline{\nabla^\bot \Delta^{-2} \p_x (\thapp + \thint) \cdot  \thint} \\
            &- \left[ \Delta^4 , \nabla^\bot \Delta^{-2} \p_x (\thapp + \thint) \cdot \nabla\right] \thint.
        \ea
		\]
		It can be easily checked that the term $\left[ \Delta^4 , \nabla^\bot \Delta^{-2} \p_x (\thapp + \thint) \cdot \nabla\right] \thint$ can be evaluated as above, and we have
		\[
		\left\| \left[ \Delta^4 , \nabla^\bot \Delta^{-2} \p_x (\thapp + \thint) \cdot \nabla\right] \thint\right\|_{L^2} \lesssim B^2 (1+t)^{-2+\frac{9}{4}} (1+t)^{-2 + \frac{1}{4}} \lesssim B^2 (1+t)^{-3/2}.
		\]
		Furthermore, since $\langle \Delta^4 \thint(t,\cdot, z)\rangle=0$ for all $t,z$,
		\[
		\int_\Om \p_z^8 \overline{\nabla^\bot \Delta^{-2} \p_x (\thapp + \thint) \cdot  \thint}\Delta^4 \thint=0.
		\]
		Eventually, integrating by parts the remaining term,
		\[\ba
		&-\int_\Om \left(\left( \nabla^\bot \Delta^{-2} \p_x (\thapp + \thint)\right) \cdot \nabla \Delta^4 \thint\right) \Delta^4 \thint\\
  &= \frac{1}{2}\int_\Om \nabla \cdot\left( \nabla^\bot \Delta^{-2} \p_x (\thapp + \thint)\right) |\Delta^4 \thint|^2 =0. 
\ea		
  \]
		Therefore, setting
		\[
		\ba 
		S^2_\bot &= -\nabla^\bot \Delta^{-2} \p_x (\thapp + \thint) \cdot \nabla \Delta^4 \thint - \p_z^8 \overline{\nabla^\bot \Delta^{-2} \p_x (\thapp + \thint) \cdot  \thint} ,\\
		S^2_\para &= \Delta^4 Q(\thint, \thapp)- \left[ \Delta^4 , \nabla^\bot \Delta^{-2} \p_x (\thapp + \thint) \cdot \nabla\right]  \thint,
		\ea 
		\]
		we obtain the desired $H^8$ estimates.
		
		We now need to estimate the time derivative of $\p_x^4 S^1_\rem$ in $L^2$.
		Note that the definition of  time $T_2$ (see \eqref{def:T2}) ensures that
		\[
		\| \p_t \p_x^4 \thint(t)\|_{L^2}\lesssim B (1+t)^{-3}\quad \forall t\in [0, T_2].%
		\]
		Setting $\psiint=\Delta^{-2} \p_x \thint$, it follows that
		\[
		\|\p_t \p_x^3 \psiint\|_{H^4}\lesssim B (1+t)^{-3}\quad \forall t\in [0, T_2].%
		\]
		From there, differentiating with respect to time $\p_x^4 S^1_\rem$, we obtain the desired estimate in $L^2$. The only problematic term is $\p_t \p_x^5 \psiint\p_z \thapp $, which we decompose as
		\[
		\p_t \p_x^5 \psiint\p_z \thapp \chi(z) + \p_t \p_x^5 \psiint\p_z \thapp (1-\chi(z) ),
		\]
		with $\chi\in C^\infty_c(\R)$ such that $\chi\equiv 1$ in a neighborhood of zero and $\chi(z)=0$ for $|z|\geq 1/2$. Let us consider the first term. Recalling that $\psiint(z=0)=0$, we write, using the Hardy inequality,
		\begin{align*}
			\| \p_t \p_x^5 \psiint\p_z \thapp \chi(z)\|_{L^2}&\leq \left\| \frac{1}{z} \p_t \p_x^5 \psiint\right\|_{L^2} \left\|z \p_z \thapp \chi(z)\right\|_{L^\infty}\\
			&\lesssim \left\|\p_t \p_x^5 \p_z\psiint\right\|_{L^2} \left\|z \p_z \thapp \chi(z)\right\|_{L^\infty}\\
			&\lesssim B (1+t)^{-3} \times B (1+t)^{-1} \lesssim B^2 (1+t)^{-4}.
		\end{align*}
		The term involving $ (1-\chi(z) )$ is treated similarly, exchanging the roles of $z=0$ and $z=1$.

	\end{proof}

	\paragraph{Conclusion.}

	We apply the operator $\Delta^4$ to \eqref{eq:threm}. We recall that by construction, $\Delta^2 \thint=\p_\bn \Delta^2 \thint=0$ on $\p\Om$. We obtain
	\[
	\p_t \Delta^4 \thint = (1-G) \p_x \Delta^2\thint + \Delta^4 S^1_\rem - [\Delta^4, G]\p_x \psiint .
	\]
	Let us now check that the assumptions of \cref{prop:bb-uniformbound} are satisfied. The decay assumptions on $\Delta^4 S^1_\rem$ follow from \cref{lem:est-Srem}. 
	Therefore it suffices to check that the decay of the  commutator term satisfies the desired bounds. Using \eqref{eq:commu} together with the bounds on $G$ (see \cref{lem-high-reg-NLBL}), we have, for all $t\in (0, T_2)$, 
    \begin{align*}
        \left\|  [\Delta^4, G]\p_x \psiint \right\|_{L^2} & \lesssim \| G\|_{W^{1,\infty}} \|\p_x \psiint\|_{H^7} + \|G\|_{H^8}\|\p_x \psiint \|_{\infty}\\
        &\lesssim B^2 (1+t)^{-5/4} + B^2 (1+t)^{1/2} (1+t)^{-11/4}  \lesssim B^2 (1+t)^{-5/4}.
    \end{align*}

	Therefore, according to \cref{prop:bb-uniformbound},  there exists a universal constant $C_2$ such that for all $t\in (0, T_2)$, setting $B= (3 + 2C_1) \bar C \|\theta_0\|_{H^{14}}$ (see \eqref{def:T2}),
	\[
	\|\Delta^4 \thint(t)\|_{L^2} \leq C_2 \left(\| \thint(t=0)\|_{H^8} + B^2\right) .
	\]
	From there, we apply \cref{prop:bb-decay} twice (first to $\p_x^2 \Delta^2 \thint$ and then to $\p_x^4\thint$), and we obtain, up to a change in the constant $C_2$, for all $t\in [0,T_2]$
	\begin{equation*}\label{est:Delta2-threm}
	\ba 
	\|\p_x^2 \Delta^2 \thint(t)\|_{L^2} &\leq  C_2 \left(\| \thint(t=0)\|_{H^8} + B^2\right) (1+t)^{-1},\\
	\| \p_x^4 \thint(t)\|_{L^2} &\leq C_2 \left( \| \thint(t=0)\|_{H^8} + B^2\right) (1+t)^{-2}.
	\ea
	\end{equation*}
	There remains to bound $\p_t \p_x^4 \thint$ and $\p_x^5\psiint$ in $L^2$. To that end, we differentiate \eqref{eq:threm} with respect to time, and we obtain
	\[
	\p_t \p_t \p_x^4\thint= (1-G) \p_x^5 \p_t \psiint + \p_t \p_x^4 S^1_\rem  - \p_t G \p_x^5 \psiint.
	\]
	The source term $\p_t \p_x^4 S^1_\rem$ is evaluated in \cref{lem:est-Srem}.
	As for the commutator term, we have
	\[
	\left\| \p_t G \p_x^5 \psiint\right\|_{L^2} \leq \|\p_t G\|_{L^\infty} \|\p_x^5 \psiint\|_{L^2} \lesssim B^3 (1+t)^{-3 + \frac12 -3} \lesssim B^3 (1+t)^{-4}.
	\]
	Using \cref{prop:bb-decay}, 
	we find that for any $t\in (0, T_2)$,
	\[
	\|\p_t \p_x^4 \thint\|_{L^2}\leq  C_2\left( \| \thint(t=0)\|_{H^8} + B^2\right) \frac{1}{(1+t)^3}.
	\]
	Using \eqref{eq:threm}, 
	\[
	\| \p_x^5 \psiint(t)\|_{L^2}\leq  C_2\left( \| \thint(t=0)\|_{H^8} + B^2\right) \frac{1}{(1+t)^3} \quad \forall t\in (0, T_2).
	\]
	Grouping these estimates with the ones on $\sigma_\mathrm{lift}^\mathrm{NL}$ from \cref{lem:BL-NLcorrector}, we infer that up to a change of the constant $C_2$, for any $t\in (0,T_2)$
	\[
	\ba 
(1+t)^2	\| \p_x^4 \threm(t)\|_{L^2} + \|\Delta^4 \threm(t)\|_{L^2} &\leq C_2
\left( \| \theta_0\|_{H^8} + B^2\right) ,\\
(1+t)^3\left(\| \p_t \p_x^4 \threm(t)\|_{L^2} + \|\p_x^5 \psrem(t)\|_{L^2}\right) &\leq C_2
\left( \| \theta_0\|_{H^8} + B^2\right) .
	\ea
	\]
	We now recall that $B_3= (3 \bar C + 5 C_1)  \|\theta_0\|_{H^{14}}$ for some constant $\bar C$ that remains to be chosen. We want to pick $\bar C$ so that
	\[
	C_2\left( \|\theta_0\|_{H^8} + (3 \bar C + 5 C_1)^2  \|\theta_0\|_{H^{14}}^2\right) \leq (\bar C + C_1)\|\theta_0\|_{H^{14}} .
	\]
	It is sufficient to take $\bar C$ such that $2 C_2\leq (\bar C + C_1)$, and $\|\theta_0\|_{H^{14}}$ sufficiently small.
	We then infer that the bounds within \eqref{def:T2} are satisfied with $B_2$ replaced by $B_2/2$. It follows that $T_2=T_1$.
	From there, recalling the estimates on $\thbl$, we deduce that there exists a universal constant $C_3 $ such that for all $t\in (0, T_1)$, for all $k\in \{4,\cdots, 8\}$,
	\begin{align*}
		\| \p_x^k \theta'\|_{L^2} &\leq \|\p_x^k \thbl\|_{L^2} + \|\p_x^k \threm\|_{L^2}\\
		&\leq C_3 \left( \|\theta_0\|_{H^{14}} +B^2\right) \left( (1+t)^{-9/8} + (1+t)^\frac{k-8}{2}\right).
	\end{align*}
	Similar estimates hold for $\p_z^k \theta'$ and $\p_x^5 \psi$ in $L^2$.
	Hence we further choose the constant $\bar C$ so that
	\[
	2 C_3  \left( \|\theta_0\|_{H^{14}} +B^2\right) \leq \bar C \|\theta_0\|_{H^{14}} 
	\]
	provided $\|\theta_0\|_{H^{14}} $ is sufficiently small. We conclude that $T_1=+\infty$. \cref{thm:BL-nonlinear} follows.

\appendix

\renewcommand{\be}{\mathbf{e}}

\section{Well-posedness of the Stokes-transport equation in Sobolev spaces} \label{app:wp}

The aim of this section is to prove the well-posedness of the Stokes-transport on the domain of interest of the present paper, namely $\Omega = \T\times(0,1)$. The proof is also valid on any regular enough bounded domain of $\R^d$ with $d=2$ or $3$. 
\begin{theorem} \label{thm:wp}
	Let $\Omega$ satisfy either
	\begin{enumerate}
		\item $\Omega = \T\times(0,1)$ or
		\item $\Omega$ is a simply connected compact subdomain of $\R^d, d=2,3$ regular enough.
	\end{enumerate}
	Let $m\geq3,\rho_0\in H^m(\Omega)$ (and $\Omega$ of regularity $\mcC^{m+2}$). The system
	\begin{equation} \label{eq:wp-st}
		\left\{
		\ba
		\p_t\rho + \bu\cdot\nabla\rho &= 0 \\
		-\Delta\bu+\nabla p &= -\rho\be_z  \\
		\div\,\bu &= 0 \\
		\bu|_{\p\Omega} &= \bzero \\
		\rho|_{t=0} &= 0,
		\ea
		\right.
	\end{equation}
	has a unique global solution for the present regularity
	\begin{equation*}
		(\rho,\bu) \in \mcC(\R_+;H^m(\Omega))\times \mcC(\R_+;H^{m+2}(\Omega)).
	\end{equation*}
	Moreover, the solution obeys the following energy estimate
	\begin{equation} \label{eq:wp-energy}
		\|\rho(t)\|_{H^m} \leq \|\rho_0\|_{H^m}\exp\left(C\int_0^t\|\nabla\bu(s)\|_{L^\infty}+\|\nabla\rho(s)\|_{L^\infty} \ud s\right).
	\end{equation}
\end{theorem}
The proof of this result follows rather classical techniques. It also relies on a previous work of one of the authors \cite{leblond}, including in particular the well-posedness in a weak sense of the system \eqref{eq:wp-st}.

\begin{remark}
    The well-posedness in the weak sense of \eqref{eq:wp-st} in $\T\times(0,1)$ is a direct consequence its the well-posedness in $\R\times(0,1)$ stated in \cite[Theorem 1.2]{leblond}. In this latter unbounded domain, the Poiseuille flows are avoided thanks to a zero flux condition on the velocity field. In the periodic case, this condition is no longer required as the periodicity of the solution prevents the existence of Poiseuille flows.
\end{remark}

\paragraph{A priori estimate.}
Formally, the energy estimate for any derivative of order $m$ can be written as
\begin{equation*}
	\frac12\frac\rd{\rd t}\|\p^m\rho\|_{L^2}^2 = - \int_\Omega [\p^m,\bu\cdot\nabla]\rho \p^m\rho,
\end{equation*}
due to the divergence free condition satisfied by $\bu$. We apply the tame estimate \eqref{eq:commu}, together with the continuous Sobolev embedding of $H^m(\Omega)$ in $L^\infty(\Omega)$ and the Stokes equation regularization estimate $\|\bu\|_{H^m}\lesssim\|\rho\|_{H^{m-2}}$ to get
\begin{equation*}
	\frac\rd{\rd t}\|\p^m\rho\|_{L^2}^2 \lesssim (\|\nabla\bu\|_{L^\infty} +\|\nabla\rho\|_{L^\infty})\|\rho\|_{H^m}^2.
\end{equation*}
One therefore obtains the same inequality with the complete $H^m$ norm on the left-hand-side, and the estimate \eqref{eq:wp-energy} follows. This energy estimate tells us that $\rho$ remains in $H^m(\Omega)$ as long as $\|\nabla\bu\|_{L^\infty}$ and $\|\nabla\rho\|_{L^\infty}$ are integrable in time. Regarding the properties we know from \cite{leblond} about the solutions of this equation it is enough to prove that the solution exists globally, and is unique. Let us recall from \cite[Theorem IV.6.1]{galdi} and \cite[Section 2.1]{leblond} that the source term and the solution of the Stokes equation satisfy for all times
\begin{equation} \label{eq:wp-stokes}
	\|\bu\|_{H^m}\lesssim\|\rho\|_{H^{m-2}}, \qquad \|\bu\|_{W^{1,\infty}}\lesssim\|\rho\|_{L^\infty}.
\end{equation}
Also, the uniform norm of $\rho$ is constant since $\rho $ is transported by an incompressible vector field. We also observe
\begin{equation} \label{eq:wp-grad}
	\|\nabla\rho\|_{L^\infty}\leq\|\nabla\rho_0\|_{L^\infty}\exp\left(C\int_0^t\|\nabla\bu(s)\|_{L^\infty}\ud s\right)\leq\|\nabla\rho_0\|_{L^\infty}\exp(C\|\rho_0\|_{L^\infty}t).
\end{equation}
Putting these considerations together leads to
\begin{equation*}
	\|\rho\|_{H^m}\leq\|\rho_0\|_{H^m}\exp\left(C\|\rho_0\|_{L^\infty}t + \frac{\|\nabla\rho_0\|_{L^\infty}}{\|\rho_0\|_{L^\infty}}\left(\exp(C\|\rho_0\|_{L^\infty}t)-1\right)\right).
\end{equation*}
This suggests that if $\rho_0\in H^m\cap W^{1,\infty}$, the solution exist globally in time in $H^m$. In particular, if $m$ is large enough so that $H^m(\Omega)\hookrightarrow W^{1,\infty}(\Omega)$, the Stokes-transport system is well-posed in $H^m$.

\paragraph{Proof.}
An iterative scheme allows us to formalize the previous considerations. Let $\rho^0 : t \mapsto \rho_0$ which  belongs to $C(\R_+,H^m(\Omega))$. Now if $\rho^N$ belongs to $C(\R_+,H^m(\Omega))$, which is true for $N=0$, we know that the Stokes system
\begin{equation*}
	\left\{
	\ba
	-\Delta\bu^N+\nabla p^N &= -\rho^N\be_z \\
	\div\,\bu^N &= 0 \\
	\bu^N|_{\p\Omega} &= \bzero,
	\ea
	\right.
\end{equation*}
admits for any time a unique solution $\bu^N(t) \in H^{m+2}(\Omega)$ obeying inequalities \eqref{eq:wp-stokes}.
By linearity of the problem, $\bu^N$ in $H^{m+2}(\Omega)$ inherits the continuity of $\rho^N$ in $H^m(\Omega)$. Then since $\bu^N$ belongs in particular to $C(\R_+,H^{m+2}(\Omega))$, the transport equation
\begin{equation*}
	\left\{
	\ba
	\p_t\rho^{N+1}+\bu^N\cdot\nabla\rho^{N+1} &= 0 \\
	\rho^{N+1}|_{t=0} &= \rho_0,
	\ea
	\right.
\end{equation*}
has a unique strong solution $\rho^{N+1} \in C(\R_+,H^m(\Omega))$. This concludes the definition of the sequences $(\rho^N)_N$ and $(\bu^N)_N$. We thereafter show that for any $T>0$ the sequence $(\rho^N)_N$ is bounded in $L^\infty((0,T),H^m(\Omega))$ and equicontinuous in $C((0,T),H^{m-1}(\Omega))$, so it converges in $C((0,T),H^{m-1}(\Omega))$ to a solution of the original system up to an extraction. Since this is true for any $T>0$ and by uniqueness of the weak solution ensured by \cite[Theorem 1.1 \& 1.2]{leblond}, we get the well-posedness of the system and the proposition is proven.
\paragraph{Boundedness.}
Let us show that we have for any $N\in\N$,
\begin{equation} \label{eq:wp-rec}
	\|\rho^N\|_{H^m} \leq \|\rho_0\|_{H^m}\exp\left(C\|\rho_0\|_{L^\infty}t+\frac{\|\nabla\rho_0\|_{L^\infty}}{\|\rho_0\|_{L^\infty}}\left(\exp(C\|\rho_0\|_{L^\infty}t)-1\right)\right) =:B_{\rho_0}(t).
\end{equation}
This inequality is immediately satisfied for $N=0$ since $\rho^0$ is constant in time and equal to $\rho_0$. Let $N\in\N$ such that \eqref{eq:wp-rec} is satisfied. Then the tame estimate \eqref{eq:tame} provides here
\begin{equation*}
	\frac\rd{\rd t}\|\rho^{N+1}(t)\|_{H^m}^2 
	\lesssim \|\nabla\bu^N\|_{L^\infty}\|\rho^{N+1}\|_{H^m}^2 + \|\nabla\rho^{N+1}\|_{L^\infty}\|\bu^N\|_{H^m}\|\rho^{N+1}\|_{H^m}.
\end{equation*}
The considerations \eqref{eq:wp-stokes} and \eqref{eq:wp-grad} applied to $\rho^N,\rho^{N+1}$ and $\bu^N$ lead here to
\begin{equation*}
	\frac\rd{\rd t}\|\rho^{N+1}\|_{H^m} \lesssim \|\rho_0\|_{L^\infty}\|\rho^{N+1}\|_{H^m} + \|\nabla\rho_0\|_{L^\infty}\exp(C\|\rho_0\|_{L^\infty}t)\|\rho^{N+1}\|_{H^m}.
\end{equation*}
From here, we use the Grönwall lemma to estimate $\|\rho^{N+1}\|_{H^m}$,
\begin{equation} \label{eq:wp-int}
	\|\rho^{N+1}\|_{H^m}\leq \exp(C\|\rho_0\|_{L^\infty}t)\left(\|\rho_0\|_{H^m} + C\|\nabla\rho_0\|_{L^\infty}\int_0^t\|\rho^N(s)\|_{H^m} \ud s\right).
\end{equation}
Then, according to the assumption on $\rho^N$, we observe that
\begin{equation*}
    \begin{aligned}
	C\|\nabla\rho_0\|_{L^\infty}&\int_0^t\|\rho^N(s)\|_{H^m}\ud s
	\\ &\leq C\|\rho_0\|_{H^m}\|\nabla\rho_0\|_{L^\infty}\int_0^t\exp\left(C\|\rho_0\|_{L^\infty}s+\frac{\|\nabla\rho_0\|_{L^\infty}}{\|\rho_0\|_{L^\infty}}\left(e^{C\|\rho_0\|_{L^\infty}s}-1\right)\right)\ud s \\
	&\leq C\|\rho_0\|_{H^m}\|\nabla\rho_0\|_{L^\infty}\int_0^{e^{C\|\rho_0\|_{L^\infty}t}-1}\exp\left(\frac{\|\nabla\rho_0\|_{L^\infty}}{\|\rho_0\|_{L^\infty}}r\right)\frac{\ud r}{C\|\rho_0\|_{L^\infty}} \\
	&=\|\rho_0\|_{H^m}\left(\exp\left(\frac{\|\nabla\rho_0\|_{L^\infty}}{\|\rho_0\|_{L^\infty}}\left(\exp(C\|\rho_0\|_{L^\infty}t) - 1\right)\right)-1\right).
    \end{aligned}
\end{equation*}
The latter bound substituted in \eqref{eq:wp-int} yields exactly the result \eqref{eq:wp-rec}. Therefore, for any $T>0$ the sequence $(\rho^N)_N$ is uniformly bounded in $L^\infty(0,T,H^m(\Omega))$.

\paragraph{Equicontinuity.}
We find a uniform bound on $(\p_t\rho^N)_N$ in $H^{m-1}(\Omega)$ to show the equicontinuity of the sequence in $C((0,T),H^{m-1}(\Omega))$. This bound, uniform in $N\in\N$ and $t\in[0,T]$ is obtained thanks to the tame estimate, the bounds \eqref{eq:wp-stokes} and the uniform bound \eqref{eq:wp-rec} on $\rho^N$,
\begin{equation*}
	\begin{aligned}
		\|\p_t\rho^N\|_{H^{m-1}}
		&=\|\bu^{N-1}\cdot\nabla\rho^{N}\|_{H^{m-1}} \\
		&\lesssim\|\bu^{N-1}\|_{L^\infty}\|\nabla\rho^{N}\|_{H^{m-1}}+\|\nabla \rho^N\|_{L^\infty}\|\bu^{N-1}\|_{H^{m-1}} \\
		&\lesssim \|\rho_0\|_{L^\infty}\|\rho^{N}\|_{H^m} + \|\rho^{N-1}\|_{H^m} \|\nabla \rho_0\|_\infty e^{C\|\rho_0\|_\infty t}\\
		&\lesssim \|\rho_0\|_{W^{1,\infty}}\|\rho_0\|_{H^m}\exp\left(C\|\rho_0\|_{L^\infty}t+\frac{\|\nabla\rho_0\|_{L^\infty}}{\|\rho_0\|_{L^\infty}}\left(\exp(C\|\rho_0\|_{L^\infty}t)-1\right)\right).
	\end{aligned}
\end{equation*}
\paragraph{Regularity.}
Let us show that the limit $\rho$ belongs to $L^\infty((0,T),H^m(\Omega))$. For any $t\in[0,T]$, $(\rho^N(t))_N$ is uniformly bounded in $H^m(\Omega)$ with respect to $N$ and $t$. Hence according to Banach--Alaoglu theorem, for any $t$ the sequence is weakly compact in $H^m(\Omega)$. Thus up to an extraction, $\rho^N(t)$ converges weakly toward a $\bar\rho(t) \in H^m(\Omega)$, and this limit satisfies
\begin{equation*}
    \|\bar\rho(t)\|_{H^m}\leq\liminf_N\|\rho^N(t)\|_{H^m},
\end{equation*}
where the right-hand-side is uniformly bounded thanks to \eqref{eq:wp-rec}. As $\rho^N$ already converges weakly in $H^m(\Omega)$, we can identify $\bar\rho$ and $\rho$, which then belongs to $L^\infty((0,T),H^m(\Omega))$. Finally, to reach the regularity $C((0,T),H^m(\Omega))$, Lemma II.5.6 in \cite{BoyerFabrie} tells us that since in particular $\rho \in L^\infty((0,T),H^m(\Omega))\cap C^0_\rw((0,T),H^{m-1}(\Omega))$ then $\rho \in C^0_\rw((0,T),H^m(\Omega))$. Hence it is enough to show that $t\mapsto\|\rho(t)\|_{H^m}$ is continuous to prove the strong continuity of $\rho$ in $H^m(\Omega)$. By weak continuity, we have
\begin{equation*}
	\|\rho_0\|_{H^m} \leq \liminf_{t\searrow 0}\|\rho(t)\|_{H^m}.
\end{equation*}
Also we have by weak convergence
\begin{equation*}
	\|\rho(t)\|_{H^m} \leq\liminf_{N\to\infty}\|\rho^N(t)\|_{H^m} \leq B_{\rho_0}(t),
\end{equation*}
which proves by clamping that $t\mapsto \|\rho(t)\|_{H^m}$ is continuous at $t=0$. This can be performed for any $t\in[0,T]$, hence the continuity. Finally, $\bu \in C^0((0,T),H^{m+2}(\Omega))$ by \eqref{eq:wp-stokes} and linearity of the Stokes equation.

\section{About the bilaplacian equation \label{app:bilaplacian}}

We use throughout the paper the following classical regularity result:

\begin{lemma}[Regularity] \label{lem:bilap}
	Let $f \in H^m(\Omega),m \geq -2$. The problem
	\begin{equation*}
		\Delta^2\psi=f, \qquad \psi|_{\p\Omega}=\p_\bn\psi|_{\p\Omega}= 0,
	\end{equation*}
	admits a unique strong solution $\psi \in H^2_0\cap H^{m+4}(\Omega)$ such that
	\begin{equation*}
		\|\psi\|_{H^{m+4}} \lesssim \|f\|_{H^m}.
	\end{equation*}
\end{lemma}

The eigenvalues and eigenfunctions of the bilaplacian in a channel can be semi-explicitly computed (see \cite{leblondPhD} for the details):

\begin{lemma}[Spectrum of the bilaplacian] \label{lem:basis}
	The eigenvalues of the operator $\Delta^2$ on $H^2_0$ in $\T\times(-1,1)$ are the union, for all $k\in\Z$, of strictly increasing sequences $(\lambda_{n,k})_{n\in\N}$ such that
	\[
	\lambda_{n,k} \simeq (n^2+k^2)^2
	\]
	with associated (unnormalized) eigenfunctions:
	\[
	b_{n,k} = e^{ikx} \begin{cases}
		\cos(\omega_{n,k}z) - \frac{\cos(\omega_{n,k})}{\cosh(r_{n,k})}\cosh(r_{n,k}z), & n\in2\N, \\
		\sin(\omega_{n,k}z) - \frac{\sin(\omega_{n,k})}{\sinh(r_{n,k})}\sinh(r_{n,k}z), & n\in2\N+1,
	\end{cases}
	\]
	with $\omega_{n,k}=(k^2-\lambda_{n,k}^{1/2})^{1/2}$ and $r_{n,k}=(k^2+\lambda_{n,k}^{1/2})^{1/2}$.
	Note that to simplify the calculations, the domain was chosen to be $\T\times(-1,1)$ and not $\Omega=\T\times(0,1)$.
\end{lemma}

\section{Proof of \texorpdfstring{\cref{lem:WP-ODE-BL}}{Lemma 3.2} \label{ap:lemma32}}

The proof of the lemma relies on energy estimates in weighted Sobolev spaces, with weights that grow like $\exp(c Z^{4/5})$ for $Z\gg 1$. 
Unfortunately, we have not been able to treat all four cases for the boundary conditions simultaneously, but we will treat (i) and (iii) (resp. (ii) and (iv)) together.
Note that when \eqref{ODE} is multiplied (formally) by $\Psi w$ or by $-\p_Z \Psi w$, 
where $w\in C^\infty([0, + \infty))$ is an arbitrary weight function,
there are many commutator terms when we integrate by parts the fifth order derivative.
The main idea is that if the weight is adequately chosen, all these commutators can be absorbed in the main order terms, which will be designed to have a positive sign.
Hence we start with the following result, which will allow us to control the commutators:

\begin{lemma}\label{lem:commu-weight}
	Let $\Psi\in C^\infty_c([0,+\infty))$ such that $\Psi(0)=0$, and let $r\in (0,1)$. 
	
	\begin{itemize}
		\item  Let $W\in C^\infty([0, +\infty))$ such that $W(Z)=\exp(Z^{4/5})$ for $Z\geq 1$, and $W\geq 1$, $\p_z^2 W\geq 0$, $W\equiv 1$ in a neighborhood of zero.
		
		Then for $k\in \{1,2\}$, there exists a constant $C_k$, independent of $r$, such that for all $r\in (0,1)$,
		\begin{multline*}
		      \left| \int_0^\infty |\p_Z^{k} \Psi(Z)|^2 \frac{|\p_Z^{3-k} W (rZ)|^2}{W(rZ)}\ud Z \right|\\
		      \leq  C_k r^{-\frac{2}{3}(3-k)}\left[\int_0^\infty |\p_Z^3 \Psi(Z)|^2 W(rZ)\ud Z  +\int_0^\infty \Psi^2(Z) \left(\frac{r\p_Z W(rZ)}{Z} + \frac{W(rZ)}{Z^2}\right)\ud Z   \right].
		\end{multline*}

		\item Let $\Phi:Z\mapsto \Psi(Z)/Z$. Then for $k\in \{1,2,3,4\}$, for all $c>0$, for $r$ sufficiently small,
		\begin{multline*}
		      \int_0^\infty \mathbf 1_{rZ>c} (rZ)^{\frac{2+2k}{5}} \exp((rZ)^{4/5}) \p_Z^k \Phi(Z)^2\ud Z \\
            \lesssim_c r^{\frac{2(k-1)}{3}}\left[ \int_0^\infty \p_Z^4 \Psi(Z)^2\exp((rZ)^{4/5})\ud Z + \int_0^\infty \p_Z \Phi(Z)^2 (rZ)^{4/5} \exp((rZ)^{4/5}) \ud Z\right].
        \end{multline*}

	\end{itemize}
\end{lemma}

\begin{proof}
	$\bullet$	For $k=0,\dots, 3$, let us consider weights $\omega_k\in W^{1,\infty}_\text{loc}((0, +\infty))$ such that
	\[
	\ba
	\forall k \in \{0,\cdots,3\},\quad \forall Z\geq 1, \ \omega_k(Z)=e^{-1} Z^{-\frac{2}{5}(3-k)} \exp(Z^{4/5} ),\\
	\forall Z\in (0,1), \quad \omega_1(Z)=\omega_3(Z)=1,\quad \omega_0(Z)=Z^{-2},\quad \omega_2(Z)=Z^2.
	\ea
	\]
	Note that the weights $\omega_k$ satisfy the following assumptions:
	\begin{itemize}
		\item For $k\in \{1,2\}$, $\omega_k\leq \sqrt{\omega_{k-1}\omega_{k+1}}$;
		\item For $k\in \{1,2\}$, $|\p_Z\omega_k|\leq C_k \sqrt{\omega_k \omega_{k-1}}$ for some constant $C_k$;
		\item $\omega_2(0)=0$;
        \item $\omega_3 \leq C W$, $\omega_0 (Z)\leq C (Z^{-2} W (Z) + Z^{-1} \p_Z W(Z))$.
	\end{itemize}
	Let us now introduce, for $k=0,\dots, 3$
	\[
	I_k:=\int_0^\infty |\p_Z^k \Psi(Z)|^2 \omega_k(rZ)\ud Z .
	\]
	Then by definition of $W$, $\omega_0,\omega_3$, there exists a constant $C$ (independent of $r>0$) such that
	\[
	r^2 I_0 + I_3 \leq C  \int_0^\infty |\p_Z^3 \Psi(Z)|^2 W(rZ)\ud Z  +\int_0^\infty \Psi^2(Z) \left(\frac{r\p_Z W(rZ)}{Z} + \frac{W(rZ)}{Z^2}\right)\ud Z  .
	\]
	Let us set $E=r^2 I_0 + I_3$. For $k=1,2$, integrating by parts and using the conditions $\Psi(0)=\omega_2(0)=0$, we have
	\[
	I_k=-\int_0^\infty \p_Z^{k-1}\Psi(Z) \p_Z^{k+1}\Psi(Z) \omega_k(rZ)\ud Z  - r \int_0^\infty  \p_Z^{k-1}\Psi(Z) \p_Z^{k}\Psi(Z) \p_Z\omega_k(rZ)\ud Z  .
	\]
	Using the properties of $\omega_k$, we deduce that there exist constants $C_k$ such that
	\[
	I_k\leq C_k \left( \sqrt{ I_{k-1} I_{k+1}} + r \sqrt{I_k I_{k-1}}\right).
	\]
	Since $r^2 I_0\leq E$, we deduce first that $I_1\lesssim E + r^{-1} \sqrt{I_2 E}$, and plugging this inequality into the bound on $I_2$, we find, since $r\in (0,1)$,
	\[
	I_1\lesssim r^{-4/3} E,\quad I_2 \lesssim r^{-2/3} E.
	\]
	
	The first inequality from \cref{lem:commu-weight} then follows easily by noticing that
	\[
\frac{(\p_Z W)^2}{W}\lesssim \omega_2,\quad \frac{(\p_Z^2 W)^2}{W}\lesssim \omega_1.
	\]
	
    $\bullet$ Let us now set, for $k\in \{1,\cdots, 4\}$,
	\[
	J_k:=\int_0^\infty \p_Z^k \Phi(Z)^2 \zeta_k(rZ)\ud Z ,
	\]
	where the weights $\zeta_k\in W^{1,\infty}_\text{loc}(\R)$ are such that $\zeta_k\equiv 0$ in a neighborhood of zero, $\zeta_k(Z)= Z^{\frac{2+2k}{5}} \exp(Z^{4/5})$ for $Z$ large enough, and $\zeta_k\lesssim \sqrt{\zeta_{k-1}\zeta_{k+1}}$, $\p_Z \zeta_k\lesssim \sqrt{\zeta_k \zeta_{k-1}}$.	Let $F:=r^{-2} J_4 + J_1$. As above, for $k\in \{2,3\}$, we have
	\[
	J_k\leq C_k\Bigl(\sqrt{J_{k-1}J_k} + r \sqrt{ J_k J_{k-1}}\Bigr).
	\]
	From there, we infer that for $k\in \{1,\cdots 4\}$, $J_k \lesssim r^{2(k-1)/3} F$.
	
	Now, since $\Psi=Z \Phi$, we have $\p_Z^4 \Psi= Z \p_Z^4 \Phi + 4 \p_Z^3 \Phi$. It follows that
	\begin{align*}
		\int_0^\infty (&\p_Z^4 \Psi(Z))^2  e^{(rZ)^{4/5}} \ud Z  \\
    ={}&\int_0^\infty ( Z \p_Z^4 \Phi + 4 \p_Z^3 \Phi)^2  e^{(rZ)^{4/5}} \ud Z \\
		={}& \int_0^\infty  Z^2 (\p_Z^4 \Phi)^2 e^{(rZ)^{4/5}} \ud Z + \int_0^\infty 16 (\p_Z^3 \Phi)^2 e^{(rZ)^{4/5}} \ud Z  -  4\int_0^\infty  (\p_Z^3 \Phi)^2 \frac{\p}{\p Z} (Z e^{(rZ)^{4/5}})\ud Z \\
		={}& \int_0^\infty  Z^2 (\p_Z^4 \Phi)^2 e^{(rZ)^{4/5}} \ud Z +16 \int_0^\infty  (\p_Z^3 \Phi)^2  e^{(rZ)^{4/5}} \ud Z \\
		&- 4 \int_0^\infty \left( 1 + \frac{4}{5} (rZ)^{4/5}\right) (\p_Z^3 \Phi)^2e^{(rZ)^{4/5}} \ud Z .
	\end{align*}
	We split the last integral into two parts, for $rZ\leq 1$ and $rZ\geq 1$. When $rZ\leq 1$, 
	\[ 4 \int_0^{r^{-1}} \left( 1 + \frac{4}{5} (rZ)^{4/5}\right) (\p_Z^3 \Phi)^2 e^{(rZ)^{4/5}} \ud Z  \leq 8  \int_0^{r^{-1}}(\p_Z^3 \Phi)^2  e^{(rZ)^{4/5}} \ud Z .\]
	And for $rZ\geq 1$, for a suitable choice of $\zeta_3$
	\[
	\int_{r^{-1}}^\infty \left( 1 + \frac{4}{5} (rZ)^{4/5}\right) (\p_Z^3 \Phi)^2 e^{(rZ)^{4/5}} \ud Z  \lesssim J_3 \lesssim r^{4/3} F.
	\]
	We infer that
	\[
	\int_0^\infty (\p_Z^4 \Psi(Z))^2  e^{(rZ)^{4/5}} \ud Z  \geq  C^{-1}r^{-2} J_4 - Cr^{4/3} F,
	\]
	and therefore, for $r$ sufficiently small,
	\[
	F\lesssim \int_0^\infty (\p_Z^4 \Psi(Z))^2  e^{(rZ)^{4/5}} \ud Z  + J_1.
	\]
	The result follows.
\end{proof}

We now turn towards the proof of \cref{lem:WP-ODE-BL}. In both cases, we start with a formal \emph{a priori} estimate, from which we deduce an appropriate notion of variational solution in a suitable Hilbert space. Existence and uniqueness then follow in a straightforward manner from the Lax-Milgram Lemma.

\smallskip
\textit{\textbf{First case:} conditions (ii) and (iv):}

As explained above, we start with a formal \emph{a priori} estimate.
Let $w\in C^\infty(\R_+)$ be an arbitrary weight function, and multiply \eqref{ODE} by $\p_Z(\Psi(Z) w(Z))/Z$. 
On the one hand,
\begin{align*}
  \int_0^\infty \p_Z^5 \Psi(Z) \p_Z(\Psi w)(Z)\ud Z ={}& \int_0^\infty \p_Z^3 \Psi(Z) \p_Z^3(\Psi w)(Z)\ud Z \\&
  - \p_Z^4 \Psi(0) \p_Z(\Psi w)(0) + \p_Z^3 \Psi(0) \p_Z^2(\Psi w)(0).  
\end{align*}
Note that the two boundary terms vanish in cases  (ii) and (iv). We obtain
\[
\int_0^\infty \p_Z^5 \Psi(Z) \p_Z(\Psi w)(Z)\ud Z =\int_0^\infty (\p_Z^3 \Psi(Z))^2 w(Z) + \sum_{k=1}^3 \begin{pmatrix}3\\ k\end{pmatrix}\int_0^\infty \p_Z^3 \Psi(Z) \p_Z^{3-k}\Psi (Z) \p_Z^k w(Z)\ud Z .
\]
On the other hand, since $\Psi(0)=0$,
\begin{align*}
\int_0^\infty \Psi(Z) \p_Z(\Psi w) (Z) \frac{\ud Z}{Z}&= \int_0^\infty (\Psi w)(Z) \p_Z (\Psi w)(Z) \frac{\ud Z}{Zw(Z)}\\
&=-\frac{1}{2}\int_0^\infty (\Psi w)^2(Z) \frac{\ud }{ \ud Z}\left(\frac{1}{Z w(Z)}\right) \ud Z.
\end{align*}
Choosing $w$ such that $\p_Z w\geq 0$, the right-hand side has a positive sign.
We then choose $w(Z)=W(rZ)$ for some $W\in C^\infty(\R_+)$ such that $W(\xi)=\exp(\xi^{4/5})$ for $\xi\geq 1$, $W(\xi)=1$ for $\xi$ in a neighborhood of zero, $\p_\xi W\geq 0$, and $r>0$ small enough.
With this choice, the positive terms in the energy are bounded from below by
\[
\int_0^\infty  (\p_Z^3 \Psi(Z)^2 W(rZ)\ud Z  + \int_0^\infty \Psi^2(Z) \left(\frac{W(rZ)}{Z^2} + r \frac{\p_Z W(rZ)}{Z}\right) \ud Z.
\]
\cref{lem:commu-weight} then implies that there exists an explicit constant $\delta>0$ such that for $k=1,2,3$,
\begin{multline*}
	\left| \int_0^\infty \p_Z^3 \Psi(Z) \p_Z^{3-k}\Psi (Z) \p_Z^k w(Z)\ud Z \right| \\
    \leq r^\delta \left[\int_0^\infty  (\p_Z^3 \Psi(Z))^2 W(rZ)\ud Z  + \int_0^\infty \Psi^2(Z) \left(\frac{W(rZ)}{Z^2} + r \frac{\p_Z W(rZ)}{Z}\right) \ud Z\right] .
\end{multline*}

Therefore, for $r>0$ sufficiently small, we obtain
\begin{multline*}
	\int_0^\infty  (\p_Z^3 \Psi(Z))^2 W(rZ)\ud Z  + \int_0^\infty \Psi^2(Z) \left(\frac{W(rZ)}{Z^2} + r \frac{\p_Z W(rZ)}{Z}\right) \ud Z
	\\
	\lesssim \int_0^1 \frac{S(Z)^2}{Z^2}\ud Z  + \int_0^\infty S(Z)^2 W(rZ)\ud Z .
\end{multline*}
This leads us to the following formulation: let 
\[
\ba
\mathcal H:=\biggl\{\Psi \in H^3(\R_+), \Psi(0)=0,\  &\int_0^\infty  (\p_Z^3 \Psi(Z))^2 e^{(rZ)^{4/5}}\ud Z <+\infty,\\ & \int_0^\infty \Psi(Z)^2 (Z^{-2} + Z^{-1/5}) e^{(rZ)^{4/5}}\ud Z <+\infty\biggr\},\ea
\]
and let
\[
  \mathcal H_0:=\{\Psi \in \mathcal H,\ \p_Z \Psi(0)=\p_Z^2 \Psi(0)=0\}.
\]
We endow $\mathcal H$ and $\mathcal H_0$ with the norm
\[
\|\Psi\|_{\mathcal H}^2= \int_0^\infty  (\p_Z^3 \Psi(Z))^2 W(rZ)\ud Z  + \int_0^\infty \Psi^2(Z) \left(\frac{W(rZ)}{Z^2} + r \frac{\p_Z W(rZ)}{Z}\right) \ud Z,
\]
where $W$ is the previous weight. We say that $\Psi\in \mathcal H$ is a solution of \eqref{ODE}-(ii) (resp. 
$\Psi\in \mathcal H_0$ is a solution of \eqref{ODE}-(iv)) if and only if for all $\Phi\in \mathcal H$ (resp. $\Phi\in \mathcal H_0$),
\[
\int_0^\infty \p_Z^3 \Psi \p_Z^3(\Phi W(r\cdot)) + \int_0^\infty \Psi(Z) \frac{ \p_Z (\Phi(Z) W(rZ))}{Z}\ud Z  = \int_0^\infty \frac{S(Z)}{Z}  \p_Z(\Phi(Z) W(rZ))\ud Z .
\]
Existence and uniqueness of solutions of  \eqref{ODE}-(ii) (resp. of \eqref{ODE}-(iv)) in $\mathcal H$ (resp. $\mathcal H_0$) follow easily from the Lax-Milgram Lemma. Using the equation, we then infer that 
\[
\int_0^\infty (\p_Z^5 \Psi(Z))^2  e^{(rZ)^{4/5}}\ud Z <+\infty.
\]
The result follows.

\smallskip 
\textit{\textbf{Second case:} conditions (i) and (iii):}

The estimates in the case of conditions (i) and (iii) are similar, but slightly less straightforward, since we shall need to combine two estimates.

We first multiply \eqref{ODE} by $-{\p_Z^3 \Psi(Z) w_1(Z)}/{Z}$, with a weight $w_1$ to be chosen later. We obtain on the one hand
\[
-\int_0^\infty \p_Z^5\Psi(Z) \p_Z^3 \Psi(Z) w_1(Z)\ud Z =\int_0^\infty (\p_Z^4 \Psi(Z))^2 w_1(Z)\ud Z + \int_0^\infty \p_Z^4 \Psi(Z) \p_Z^3 \Psi(Z)\p_Z w_1(Z)\ud Z .
\]
Note that the boundary term $-\p_Z^4 \Psi(0) \p_Z^3 \Psi(0) w_1(0)$ vanishes in cases (i) and (iii).
The first term gives a positive contribution to the energy, and the second one will be treated below with the help of \cref{lem:commu-weight}.
On the other hand, we obtain for the zero-th order term, noticing that either $\p_Z^2 \Psi(0)=0$ (in case (iii)) or $(Z^{-1} \Psi(Z))\vert_{z=0}=\p_Z \Psi(0)=0$ (in case (i)),
\[
-\int_0^\infty \frac{\Psi(Z)}{Z}\p_Z^3 \Psi(Z) w_1(Z)\ud Z  = \int_0^\infty \p_Z^2\Psi(Z) \frac{\ud}{\ud Z}\left(\frac{\Psi(Z)}{Z} w_1(Z)\right)\ud Z .
\]
As in \cref{lem:commu-weight}, we set $\Phi(Z)=\Psi(Z)/Z$.
Let us write $\p_Z^2\Psi$ as
\[
\p_Z^2\Psi(Z) = \p_Z^2 \left( Z \Phi(Z)\right)= 2 \p_Z\Phi(Z) + Z \p_Z^2 \Phi(Z). 
\]
Performing integrations by parts and assuming that $\p_Z  w_1(0)=0$, we obtain
\[\ba
&\int_0^\infty \p_Z^2 \Psi(Z) \frac{\ud}{\ud Z}\left(\Phi(Z)w_1(Z)\right)\ud Z \\
=&\frac{3}{2}\int_0^\infty \left( \p_Z \Phi \right)^2 (w_1 (Z)- Z \p_Zw_1(Z))\ud Z + \frac{1}{2}\int_0^\infty \Psi(Z)^2 Z^{-1} \p_Z^3 w_1(Z)\ud Z .
\ea
\]
We shall choose $w_1$ so that $\p_Z^3 w_1\geq 0$, so that the last term has a positive sign.
However, for $Z\gg 1$, $w_1 - Z \p_Zw_1<0$, and therefore we need to add another term to the energy. More precisely, we now multiply \eqref{ODE} by $- Z^{-1}\p_Z(\p_Z \Phi w_2)$, with a weight $w_2$ which vanishes identically in a neighborhood of zero. We obtain
\[
-\int_0^\infty \frac{\Psi(Z)}{Z}\p_Z\left( \p_Z \Phi w_2\right)\ud Z  = \int_0^\infty \left( \p_Z\Phi(Z) \right)^2 w_2(Z)\ud Z .
\]
We then take $w_i(Z)=W_i(rZ)$, with $0<r\ll 1$ and $W_1,W_2$ satisfying the following properties:
\begin{itemize}
	\item $W_1\equiv 1$, $W_2\equiv 0$ in a neighborhood of zero;
	
	\item $\p_Z^3 W_1\geq 0$;
	
	\item $W_1(Z)=C \exp(Z^{4/5})$ for $Z$ large enough;
	
	\item $W_2 + \frac{3}{2}(W_1 - Z \p_Z W_1)\gtrsim (1 + Z^{4/5})\exp(Z^{4/5})$.
	
\end{itemize}

Our energy is then
\begin{align*}
	&\int_0^\infty (\p_Z^4 \Psi(Z))^2 W_1(rZ)\ud Z  + \int_0^\infty \left( \p_Z \Phi(Z) \right)^2 \left[ W_2 + \frac{3}{2}(W_1 - Z W_1')\right](rZ)\ud Z \\
	&+ \frac{r^3}{2}\int_0^\infty \Psi(Z)^2 Z^{-1} \p_Z^3 W_1(rZ)\ud Z \\
	&\gtrsim \int_0^\infty (\p_Z^4 \Psi(Z))^2 \exp((rZ)^{4/5})\ud Z + \int_0^\infty \left( \p_Z \Phi(Z) \right)^2  (1 + (rZ)^{4/5})\exp((rZ)^{4/5})\ud Z .
\end{align*}
Let us now consider the two  commutator terms, namely
\[
r \int_0^\infty \p_Z^4 \Psi(Z) \p_Z^3 \Psi(Z) \p_Z W_1(rZ)\ud Z %
\quad\text{and}\quad \int_0^\infty \p_Z^4 \Psi (Z) \p_Z^2\left( \p_Z \Phi w_2\right)\ud Z .
\]
For the first one, we write $\p_Z^3 \Psi= Z \p_Z^3 \Phi + 3\p_Z^2 \Phi$. We note that $\p_ZW_1(rZ) \lesssim \mathbf 1_{rZ>c} (rZ)^{-1/5} \exp((rZ)^{4/5})$. Using the second part of \cref{lem:commu-weight}, we infer that there exists $\delta>0$ such that
\begin{align*}
	r \int_0^\infty & \p_Z^4 \Psi(Z) \p_Z^3 \Psi(Z) W_1'(rZ)\ud Z \\
	\lesssim{} & \left(\int_0^\infty (\p_Z^4 \Psi(Z))^2 W_1(rZ)\ud Z \right)^{1/2}\left(\int_0^\infty \mathbf 1_{rZ>c} (rZ)^{8/5} (\p_Z^3 \Phi(Z))^2\exp((rZ)^{4/5})\ud Z \right)^{1/2}\\
	&+ r \left(\int_0^\infty (\p_Z^4 \Psi(Z))^2 W_1(rZ)\ud Z \right)^{1/2}\left(\int_0^\infty \mathbf 1_{rZ>c} (rZ)^{-2/5} (\p_Z^2 \Phi(Z))^2 \exp((rZ)^{4/5})\ud Z \right)^{1/2}\\
	\lesssim{}& r^\delta \left[  \int_0^\infty (\p_Z^4 \Psi(Z))^2 \exp((rZ)^{4/5})\ud Z + \int_0^\infty \left( \p_Z \Phi(Z) \right)^2   (rZ)^{4/5}\exp((rZ)^{4/5}\ud Z \right].
\end{align*}
Let us now address the second commutator term. We have for instance, using once again \cref{lem:commu-weight},
\begin{align*}
	\int_0^\infty &\p_Z^4 \Psi(Z) \p_Z^3 \Phi w_2(rZ)\ud Z \\
	&\lesssim \left(\int_0^\infty (\p_Z^4 \Psi(Z))^2 W_1(rZ)\ud Z \right)^{1/2} \left( \int_0^\infty \mathbf 1_{rZ>c} (rZ)^{8/5} (\p_Z^3 \Phi)^2 \exp((rZ)^{4/5}) \ud Z\right)^{1/2}\\
	&\lesssim r^{2/3} \left[  \int_0^\infty (\p_Z^4 \Psi(Z))^2 \exp((rZ)^{4/5})\ud Z + \int_0^\infty \left( \p_Z \Phi(Z) \right)^2   (rZ)^{4/5})\exp((rZ)^{4/5}\ud Z \right].
\end{align*}
The two other terms are treated in a similar fashion.
As in the first case, we find that for $r$ small enough, the energy is controlled by 
\[
\int_0^1 \frac{S(Z)^2}{Z^2}\ud Z  + \int_0^\infty S(Z)^2 (1+ (rZ)^{2/5})\exp{(rZ)^{4/5}}\ud Z .
\]
We conclude by a Lax-Milgram type argument.\qed

\paragraph{Acknowledgments.}
The authors thank Miguel Rodrigues for fruitful discussions especially for the identification of the limit profile, and  the referees for the quality and thoroughness of their comments.
This work has received funding from the European Research Council (ERC) under the European Union's Horizon 2020 research and innovation programme (Grant agreement No. \href{https://cordis.europa.eu/project/id/637653}{637653}, project BLOC) and by the French National Research Agency (Grants \href{https://anr.fr/Projet-ANR-18-CE40-0027}{ANR-18-CE40-0027}, project SingFlows and \href{https://anr.fr/Projet-ANR-23-CE40-0014}{ANR-23-CE40-0014-01}, project BOURGEONS.
A.-L. Dalibard acknowledges the support of the Institut Universitaire de France.
J. Guillod acknowledges the support of the Initiative d'Excellence (Idex) of Sorbonne University through the Emergence program.

\bibliographystyle{alphaurl}
\bibliography{paper}
\end{document}